\documentclass[a4paper,11pt]{article}

\usepackage[utf8]{inputenc}
\usepackage[british]{babel}

\usepackage[mono=false]{libertine}
\usepackage[T1]{fontenc}

\usepackage{amsthm}
\usepackage[cmintegrals,libertine]{newtxmath}
\usepackage[cal=euler, scr=boondoxo]{mathalfa}
\useosf

\usepackage{microtype}

\usepackage[margin=2.5cm]{geometry}

\usepackage{tikz}

\usepackage{hyperref}
\hypersetup{
	colorlinks=true,
		citecolor=blue!60!black,
		linkcolor=red!60!black,
		urlcolor=green!40!black,
		filecolor=yellow!50!black,
	breaklinks=true,
	pdfpagemode=UseNone,
	bookmarksopen=false,
}

\usepackage{enumitem}

\newtheorem{thm}{Theorem}
\newtheorem{lem}{Lemma}
\newtheorem{prop}{Proposition}
\newtheorem{cor}{Corollary}

\theoremstyle{definition}
\newtheorem{rem}{Remark}
\newtheorem{defi}{Definition}

\newcommand{\ensembles}[1]{\mathbf{#1}}
	\newcommand{\N}{\ensembles{N}}
	\newcommand{\Z}{\ensembles{Z}}
	\newcommand{\R}{\ensembles{R}}
	\newcommand{\U}{\ensembles{U}}
\newcommand{\ind}[1]{\ensembles{1}_{\{#1\}}}
	\renewcommand{\P}{\ensembles{P}}
	\newcommand{\E}{\ensembles{E}}

\renewcommand{\Pr}[1]{\P\left(#1\right)}
\newcommand{\Prc}[2]{\P\left(#1 \;\middle|\; #2\right)}
\newcommand{\Es}[1]{\E\left[#1\right]}
\newcommand{\Esc}[2]{\E\left[#1 \;\middle|\; #2\right]}

\newcommand{\Map}{\mathbf{M}}
\newcommand{\barMap}{\Map^\star}
\newcommand{\map}{\mathcal{M}}

\newcommand{\Mapb}{\Map}

\newcommand{\tree}{\mathbf{T}}
\newcommand{\bartree}{\mathbf{LT}}
\newcommand{\T}{\mathcal{T}}

\newcommand{\n}{\mathbf{n}}
\newcommand{\m}{\mathbf{m}}
\newcommand{\q}{\mathbf{q}}

\newcommand{\dgr}{d_{\mathrm{gr}}}

\newcommand{\exc}{\mathbf{e}}
\renewcommand{\d}{\mathrm{d}}
\newcommand{\e}{\mathrm{e}}
\newcommand{\CRT}{\mathscr{T}}

\newcommand{\BDG}{\mathsf{BDG}}
\newcommand{\JS}{\mathsf{JS}}
\newcommand{\Good}{\mathsf{Good}}
\newcommand{\Bin}{\mathsf{Bin}}
\newcommand{\Cont}{\mathsf{Cont}}
\newcommand{\LR}{\mathsf{LR}}
\newcommand{\GW}{\mathbf{GW}}
\newcommand{\SG}{\mathbf{SG}}

\newcommand{\cv}[1][n]{\enskip\mathop{\longrightarrow}^{}_{#1 \to \infty}\enskip}
\newcommand{\cvloi}[1][n]{\enskip\mathop{\longrightarrow}^{(d)}_{#1 \to \infty}\enskip}

\newcommand{\cvproba}[1][n]{\enskip\mathop{\longrightarrow}^{\P}_{#1 \to \infty}\enskip}

\newcommand{\eqloi}{\enskip\mathop{=}^{(d)}\enskip}

\title{Scaling limits of random bipartite planar maps with a prescribed degree sequence}

\author{Cyril \textsc{Marzouk}
\thanks{Laboratoire de Math\'ematiques, 
Univ. Paris-Sud, Universit\'e Paris-Saclay.
\hfill
\href{mailto:cyril.marzouk@math.u-psud.fr}{\texttt{cyril.marzouk@math.u-psud.fr}}
\newline
This worked was supported by the Fondation Sciences Mathématiques de Paris and the Univ. Pierre et Marie Curie for a major part, and then the Fondation Mathématique Jacques Hadamard; partial support also from the grant \texttt{ANR-14-CE25-0014} (ANR GRAAL).
}
}

\begin{document}

\maketitle

\begin{abstract}
We study the asymptotic behaviour of uniform random maps with a prescribed face-degree sequence, in the bipartite case, as the number of faces tends to infinity. Under mild assumptions, we show that, properly rescaled, such maps converge in distribution towards the Brownian map in the Gromov--Hausdorff sense. This result encompasses a previous one of Le Gall for uniform random $q$-angulations where $q$ is an even integer. It applies also to random maps sampled from a Boltzmann distribution, under a second moment assumption only, conditioned to be large in either of the sense of the number of edges, vertices, or faces. The proof relies on the convergence of so-called ``discrete snakes'' obtained by adding spatial positions to the nodes of uniform random plane trees with a prescribed child sequence recently studied by Broutin \& Marckert. This paper can alternatively be seen as a contribution to the study of the geometry of such trees.
\end{abstract}

\section{Introduction}

\subsection{Random planar maps as metric spaces}

The study of scaling limits of large random maps, viewed as metric spaces, towards a universal object called the \emph{Brownian map} has seen numerous developments over the last decade. This paper is another step towards this universality as we show that the Brownian map appears as limit of maps with a prescribed face-degree sequence. This particular model is introduced in the next subsection, let us the first discuss the general idea of such studies and recall some previous results.

Recall that a (planar) map is an embedding of a finite connected graph into the two-dimensional sphere, viewed up to orientation-preserving homeomorphisms. For technical reasons, the maps we consider will always be \emph{rooted}, which means that an oriented edge is distinguished. Maps have been widely studied in combinatorics and \emph{random} maps are of interest in theoretical physics, for which they are a natural discretised version of random geometry, in particular in the theory of quantum gravity (see e.g. \cite{Ambjorn-Durhuus-Jonsson:Quantum_geometry_a_statistical_field_theory_approach}). One can view a map as a (finite) metric space by endowing the set of vertices with the graph distance: the distance between two vertices is the minimal number of edges of a path going from one to the other; throughout this paper, if $M$ is a map, we shall denote the associated metric space, with a slight abuse of notation, by $(M, \dgr)$. The set of all compact metric spaces, considered up to isometry, can be equipped with a metric, called the Gromov--Hausdorff distance, which makes it separable and complete \cite{Gromov:Metric_structures_for_Riemannian_and_non_Riemannian_spaces,Burago-Burago-Ivanov:A_course_in_metric_geometry}; we can then study the convergence in distribution of random maps viewed as metric spaces. 

The first and fondamental result in this direction has been obtained simultaneously by Le Gall \cite{Le_Gall:Uniqueness_and_universality_of_the_Brownian_map} and Miermont \cite{Miermont:The_Brownian_map_is_the_scaling_limit_of_uniform_random_plane_quadrangulations} using different approaches. We call \emph{faces} of a map the connected components of the complement of the edges; the \emph{degree} of a face is then the number of edges incident to it, with the convention that if both sides of an edge are incident to the same face, then it is counted twice. A \emph{quadrangulation} is a map in which all faces have degree $4$. In \cite{Le_Gall:Uniqueness_and_universality_of_the_Brownian_map} and \cite{Miermont:The_Brownian_map_is_the_scaling_limit_of_uniform_random_plane_quadrangulations}, it is shown that if $\mathcal{Q}_n$ is a uniform random rooted quadrangulation with $n$ faces, then the convergence in distribution
\[\left(\mathcal{Q}_n, \left(\frac{9}{8 n}\right)^{1/4} \dgr\right)
\cvloi
(\mathscr{M}, \mathscr{D}),\]
holds in the sense of Gromov--Hausdorff, where the limit $(\mathscr{M}, \mathscr{D})$, called the \emph{Brownian map}, is a random compact metric space, which is almost surely homeomorphic to the $2$-sphere (Le Gall \& Paulin \cite{Le_Gall-Paulin:Scaling_limits_of_bipartite_planar_maps_are_homeomorphic_to_the_2_sphere}, Miermont \cite{Miermont:On_the_sphericity_of_scaling_limits_of_random_planar_quadrangulations}) and has Hausdorff dimension $4$ (Le Gall \cite{Le_Gall:The_topological_structure_of_scaling_limits_of_large_planar_maps}). Let us mention that the Brownian map first appeared in the work of Marckert \& Mokkadem \cite{Marckert-Mokkadem:Limit_of_normalized_random_quadrangulations_the_Brownian_map} as a limit of rescaled quadrangulations for a distance different than the Gromov--Hausdorff distance.

Le Gall \cite{Le_Gall:Uniqueness_and_universality_of_the_Brownian_map} designs also a general method to prove such a limit theorem for other classes of random maps, using the above convergence of quadrangulations. Indeed, the main result in \cite{Le_Gall:Uniqueness_and_universality_of_the_Brownian_map} is stated for \emph{$q$-angulations} (which are maps in which each face has degree $q$) with $n$ faces, for any $q \in \{3, 4, 6, 8, \dots\}$ fixed. The limit is always the Brownian map as well as the scaling factor $n^{-1/4}$, only the multiplicative constant $(9/8)^{1/4}$ above depends on $q$ (see the precise statement below). Note that apart from the case $q=3$ of \emph{triangulations}, \cite{Le_Gall:Uniqueness_and_universality_of_the_Brownian_map} only deals with maps with even face-degrees, which corresponds in the planar case to \emph{bipartite} maps. The non-bipartite case is technically more involved and we henceforth restrict ourselves to bipartite maps as well. In this paper, we consider a large class of maps which enables us to recover and extend previous results, but we stress that it does \emph{not} recover the one above on quadrangulations; as a matter of fact, as in \cite{Le_Gall:Uniqueness_and_universality_of_the_Brownian_map}, we use the latter in our proof.

\subsection{Main result and notation}
\label{sec:resultat}

We generalise $q$-angulations by considering maps with possibly faces of different degrees. For every integer $n \ge 2$, we are given a sequence $\n = (n_i ; i \ge 1)$ of non-negative integers satisfying
\[\sum_{i \ge 1} n_i = n,\]
and we denote by $\Map(\n)$ the finite\footnote{Its cardinal was first calculated by Tutte \cite{Tutte:A_census_of_slicings} who considered the dual maps, i.e. Eulerian maps with a prescribed vertex-degree sequence.} set of rooted planar maps with $n_i$ faces of degree $2i$ for every $i \ge 1$. Let us introduce the notation that we shall use throughout this paper. Set
\begin{equation}\label{eq:nombre_aretes_sites_carte}
N_\n = \sum_{i \ge 1} i n_i
\qquad\text{and}\qquad
n_0 = 1 + N_\n - n.
\end{equation}
It is easy to see that every map in $\Map(\n)$ contains $n$ faces and $N_\n$ edges so, according to Euler's formula, it has $2 + N_\n - n = n_0+1$ vertices (this shift by one will simplify some statements later). We next define a probability measure and its variance by
\[p_\n(i) = \frac{n_i}{N_\n+1}
\quad\text{for}\quad i \ge 0
\qquad\text{and}\qquad
\sigma^2_\n = \sum_{i \ge 1} i^2 p_\n(i) - \left(\frac{N_\n}{N_\n+1}\right)^2.\]
The probability $p_\n$ is (up to the fact that there are $n_0+1$ vertices) the empirical half face-degree distribution of a map in $\Map(\n)$ if one sees the vertices as faces of degree $0$. Last, let us denote by
\[\Delta_\n = \max\{i \ge 0 : n_i > 0\}\]
the right edge of the support of $p_\n$. 

Our main assumption is the following: there exists a probability measure $p = (p(i) ; i \ge 0)$ with mean $1$ and variance $\sigma_p^2 = \sum_{i \ge 1} i^2 p(i) - 1 \in (0, \infty)$ such that, as $|\n| = n \to \infty$,
\begin{equation}\label{eq:H}\tag{\textbf{H}}
p_\n \Rightarrow p,
\qquad
\sigma^2_\n \to \sigma_p^2
\qquad\text{and}\qquad
n^{-1/2} \Delta_\n \to 0,
\end{equation}
where ``$\Rightarrow$'' denotes the weak convergence of probability measures, which is here equivalent to $p_\n(i) \to p(i)$ for every $i \ge 0$.

\begin{thm}\label{thm:cv_carte}
Under \eqref{eq:H}, if $\map_n$ is sampled uniformly at random in $\Map(\n)$ for every $n \ge 2$, then the following convergence in distribution holds in the sense of Gromov--Hausdorff:
\[\left(\map_n, \left(\frac{9}{4} \frac{1-p(0)}{\sigma_p^2} \frac{1}{n}\right)^{1/4} \dgr\right)
\cvloi
(\mathscr{M}, \mathscr{D}).\]
\end{thm}

Since the graph distance is defined in terms of edges, it would be natural to make the rescaling depend on $N_\n$ rather than $n$. Under \eqref{eq:H}, we have $n/N_\n \to 1-p(0)$ as $n \to \infty$ so the previous convergence is equivalent to
\[\left(\map_n, \left(\frac{9}{4 \sigma_p^2} \frac{1}{N_\n}\right)^{1/4} \dgr\right)
\cvloi
(\mathscr{M}, \mathscr{D}).\]

This result recovers the aforementioned one of Le Gall \cite{Le_Gall:Uniqueness_and_universality_of_the_Brownian_map} for $2\kappa$-angulations for $\kappa \ge 2$. Indeed, these correspond to $\Map(\n)$ where $n_i = n$ if $i = \kappa$ and $n_i=0$ otherwise. In this case $N_\n = n\kappa$ and \eqref{eq:H} is fulfilled with
\[p(\kappa) = 1 - p(0) = \kappa^{-1}
\qquad\text{and so}\qquad
\sigma_p^2 = \kappa-1.\]
Theorem \ref{thm:cv_carte} therefore immediately yields:

\begin{cor}[Le Gall \cite{Le_Gall:Uniqueness_and_universality_of_the_Brownian_map}]
\label{cor:cv_angulations}
Fix $\kappa \ge 2$ and for every $n \ge 2$, let $\map_n^{(\kappa)}$ be a uniform random $2\kappa$-angulation with $n$ faces. The following convergence in distribution holds in the sense of Gromov--Hausdorff:
\[\left(\map_n^{(\kappa)}, \left(\frac{9}{4 \kappa(\kappa-1)} \frac{1}{n}\right)^{1/4} \dgr\right)
\cvloi
(\mathscr{M}, \mathscr{D}).\]
\end{cor}

\subsection{Boltzmann random maps}
\label{subsec:intro_Boltzmann}

Theorem \ref{thm:cv_carte} also applies to random maps sampled from a \emph{Boltzmann distribution}. Given a sequence $\q = (q_k ; k \ge 1)$ of non-negative real numbers, we define a measure $W^\q$ on the set $\Mapb$ of rooted bipartite maps by the formula
\[W^\q(\map) = \prod_{f \in \mathrm{Faces}(\map)} q_{\mathrm{deg}(f)/2},
\qquad \map \in \Mapb,\]
where $\mathrm{Faces}(\map)$ is the set of faces of $\map$ and $\mathrm{deg}(f)$ is the degree of such a face $f$. Set $Z_\q = W^\q(\Mapb)$; whenever it is finite, the formula
\[\P^\q(\cdot) = \frac{1}{Z_\q} W^\q(\cdot)\]
defines a probability measure on $\Mapb$. We consider next such random maps conditioned to have a large size for several notions of size. For every integer $n \ge 1$, let $\Mapb_{E=n}$, $\Mapb_{V=n}$ and $\Mapb_{F=n}$ be the subsets of $\Mapb$ of those maps with respectively $n$ edges, $n$ vertices and $n$ faces. For every $S = \{E, V, F\}$ and every $n \ge 1$, we define
\[\P^\q_{S=n}(\map) = \P^\q(\map \mid \map \in \Mapb_{S=n}),
\qquad \map \in \Mapb_{S=n},\]
the law of a Boltzmann map conditioned to have size $n$; here and later, we shall always, if necessary, implicitly restrict ourselves to those values of $n$ for which $W^\q(\Mapb_{S=n}) \ne 0$, and limits shall be understood along this subsequence.

Under mild integrability conditions on $\q$, we prove in Section \ref{sec:Boltzmann} that for every $S \in \{E, V, F\}$, there exists a constant $K^\q_S > 0$ such that if $\map_n$ is sampled from $\P^\q_{S=n}$ for every $n \ge 1$, then the convergence in distribution
\[\left(\map_n, \left(\frac{K^\q_S}{n}\right)^{1/4} \dgr\right)
\cvloi
(\mathscr{M}, \mathscr{D}),\]
holds in the sense of Gromov--Hausdorff. We refer to Theorem \ref{thm:cartes_Boltzmann} for a precise statement. Observe that for any choice $S \in \{E, V, F\}$, if $\map_n$ is sampled from $\P^\q_{S=n}$ then, conditional on its degree sequence, say, $\nu_{\map_n} = (\nu_{\map_n}(i) ; i \ge 1)$, it has the uniform distribution in $\Map(\nu_{\map_n})$. The proof of the above convergence consists in showing that $\nu_{\map_n}$ satisfies \eqref{eq:H} in probability for some deterministic limit law $p_\q$. Indeed, by Skorohod's representation Theorem, there exists then a probability space where versions of $\nu_{\map_n}$ under $\P^\q_{S=n}$ satisfy \eqref{eq:H} almost surely so we may apply Theorem \ref{thm:cv_carte} and conclude the convergence in law of the rescaled maps.

The case $S=V$ was obtained by Le Gall \cite[Theorem 9.1]{Le_Gall:Uniqueness_and_universality_of_the_Brownian_map}, relying on results of Marckert \& Miermont \cite{Marckert-Miermont:Invariance_principles_for_random_bipartite_planar_maps}, when $\q$ is \emph{regular critical}, meaning that the distribution $p_\q$ (which is roughly that of the half-degree of a typical face when we see vertices as faces of degree $0$) admits small exponential moments. Here, we generalise this result (and consider other conditionings) to all \emph{generic critical} sequences $\q$, i.e. those for which $p_\q$ admits a second moment.

Let us mention that Le Gall \& Miermont \cite{Le_Gall-Miermont:Scaling_limits_of_random_planar_maps_with_large_faces} have also considered Boltzmann random maps with $n$ vertices in which the distribution of the degree of a typical face is in the domain of attraction of a stable distribution with index $\alpha \in (1, 2)$ and obtained different objects at the limit (after extraction of a subsequence). Also, Janson \& Stef\'{a}nsson \cite{Janson-Stefansson:Scaling_limits_of_random_planar_maps_with_a_unique_large_face} have studied maps with $n$ edges which exhibit a condensation phenomenon and converge, after rescaling, towards the Brownian tree: a unique giant face emerges and its boundary collapses into a tree.

The conditioning $S=E$ by the number of edges is somewhat different since the set $\Mapb_{E=n}$ is finite\footnote{See Walsh \cite[Equation 7]{Walsh:Hypermaps_versus_bipartite_maps} for an expression of its cardinal.} so the distribution $\P^\q_{E=n}(\cdot) = W^\q(\cdot)/W^\q(\Mapb_{E=n})$ on $\Mapb_{E=n}$ makes sense even if $W^\q(\Map)$ is infinite; we shall see that the above convergence still holds in this case (Theorem \ref{thm:cartes_Boltzmann_n_aretes}). The simplest example is the constant sequence $q_k=1$ for every $k \ge 1$, in which case $\P^\q_{E=n}$ corresponds to the uniform distribution in $\Mapb_{E=n}$; in this case, we calculate $K^\q_E = 1/2$, which recovers a result first due to Abraham \cite{Abraham:Rescaled_bipartite_planar_maps_converge_to_the_Brownian_map}:

\begin{cor}[Abraham \cite{Abraham:Rescaled_bipartite_planar_maps_converge_to_the_Brownian_map}]
\label{cor:cv_cartes_biparties}
For every $n \ge 1$, let $\mathcal{B}_n$ be a uniform random bipartite map with $n$ edges. The following convergence in distribution holds in the sense of Gromov--Hausdorff:
\[\left(\mathcal{B}_n, \left(\frac{1}{2 n}\right)^{1/4} \dgr\right)
\cvloi
(\mathscr{M}, \mathscr{D}).\]
\end{cor}

\subsection{Approach and organisation of the paper}

Our approach to proving Theorem \ref{thm:cv_carte} follows closely the robust one of Le Gall \cite{Le_Gall:Uniqueness_and_universality_of_the_Brownian_map}. Specifically, we code our map $\map_n$ by a certain \emph{labelled} (or \emph{spatial}) \emph{two-type tree} $(\T_n, \ell_n)$ via a bijection due to Bouttier, Di Francesco \& Guitter \cite{Bouttier-Di_Francesco-Guitter:Planar_maps_as_labeled_mobiles}: $\T_n$ is a plane tree and $\ell_n$ is a function which associates with each vertex of $\T_n$ a label (or a spatial position) in $\Z$. Such a labelled tree is itself encoded by a pair of discrete paths $(\mathcal{C}^\circ_n, \mathcal{L}^\circ_n)$; we show that under \eqref{eq:H}, this pair, suitably rescaled, converges in distribution towards a pair $(\exc, Z)$ called in the literature the ``head of the Brownian snake'' (e.g. \cite{Marckert-Mokkadem:States_spaces_of_the_snake_and_its_tour_convergence_of_the_discrete_snake, Janson-Marckert:Convergence_of_discrete_snakes, Marckert:The_lineage_process_in_Galton_Watson_trees_and_globally_centered_discrete_snakes}). The construction of the Brownian map from $(\exc, Z)$ is analogous to the Bouttier--Di Francesco--Guitter bijection; as it was shown by Le Gall \cite{Le_Gall:Uniqueness_and_universality_of_the_Brownian_map}, Theorem 1 follows from this functional limit theorem as well as a certain``invariance under re-rooting'' of our maps.

To prove such an invariance principle for $(\T_n, \ell_n)$, we further rely on a more recent bijection due to Janson \& Stef{\'a}nsson \cite{Janson-Stefansson:Scaling_limits_of_random_planar_maps_with_a_unique_large_face} which maps two-type trees to \emph{one-type} trees which are easier to control. As a matter of fact, if $\map_n$ is uniformly distributed in $\Map(\n)$ and $(T_n, l_n)$ is its corresponding labelled one-type tree, then the unlabelled tree $T_n$ is a uniform random tree with a prescribed degree (in the sense of offspring) sequence as studied by Broutin \& Marckert \cite{Broutin-Marckert:Asymptotics_of_trees_with_a_prescribed_degree_sequence_and_applications}. The labelled tree $(T_n, l_n)$ is again encoded by a pair of functions $(H_n, L_n)$ and the main result of \cite{Broutin-Marckert:Asymptotics_of_trees_with_a_prescribed_degree_sequence_and_applications} is, under the very same assumption \eqref{eq:H}, the convergence of $H_n$ suitably rescaled towards $\exc$. Our main contribution, see Theorem \ref{thm:cv_fonctions_serpent}, consists in strengthening this result by adding the labels to show that the pair $(H_n, L_n)$, suitably rescaled, converges towards $(\exc, Z)$, and then transporting this invariance principle back to the two-type tree $(\T_n, \ell_n)$.

The previous works on the convergence of large random labelled trees focus on the case when the tree is a size-conditioned (one or multi-type) Galton--Watson tree and a lot of effort has been put to reduce the assumptions of the labels as much as possible, maintaining quite strong assumption on the tree itself; a common assumption is indeed to consider a Galton--Watson tree whose offspring distribution admits small exponential moments; in order to reduce the assumption on the labels, Marckert \cite{Marckert:The_lineage_process_in_Galton_Watson_trees_and_globally_centered_discrete_snakes} even supposes the offsprings to be uniformly bounded. In this paper, we take the opposite direction: we focus only on the labels given by the bijection with planar maps, which satisfy rather strong assumptions, and work under weak assumptions on the tree (essentially a second moment condition). Furthermore, we consider trees with a prescribed degree sequence, which are more general than Galton--Watson trees and on which the literature is limited, which explains the length of this work.

Let us mention that other convergences towards the Brownian map similar to Theorem \ref{thm:cv_carte} have been obtained using also other bijections with labelled trees: Beltran \& Le Gall \cite{Beltran_Le_Gall:Quadrangulations_with_no_pendant_vertices} studied random quadrangulations without vertices of degree one, Addario-Berry \& Albenque \cite{Addario_Berry-Albenque:The_scaling_limit_of_random_simple_triangulations_and_random_simple_triangulations} considered random triangulations and quadrangulations without loops or multiple edges and
Bettinelli, Jacob \& Miermont \cite{Bettinelli-Jacob-Miermont:The_scaling_limit_of_uniform_random_plane_maps_via_the_Ambjorn_Budd_bijection} uniform random maps with $n$ edges.

This work leaves open two questions that we plan to investigate in the future. First, one can consider non-bipartite maps with a prescribed degree sequence; we restricted ourselves here to bipartite maps because (except in the notable case of triangulations), in the non-bipartite case, the Bouttier--Di Francesco--Guitter bijection yields a more complicated labelled three-type tree which is harder to analyse; moreover, the Janson--Stef{\'a}nsson bijection does not apply to such trees so the method of proof should be different. A second direction of future work would be to relax the assumption \eqref{eq:H}, in particular to consider maps with large faces. A first step would be to extend the work of Broutin \& Marckert \cite{Broutin-Marckert:Asymptotics_of_trees_with_a_prescribed_degree_sequence_and_applications} on plane trees; we believe that the family of so-called inhomogeneous continuum random trees introduced in \cite{Aldous-Pitman:Inhomogeneous_continuum_random_trees_and_the_entrance_boundary_of_the_additive_coalescent, Camarri-Pitman:Limit_distributions_and_random_trees_derived_from_the_birthday_problem_with_unequal_probabilities} appears at the limit; one would then construct a family of random maps from these trees, replacing the Brownian excursion $\exc$ by their ``exploration process'' studied in \cite{Aldous-Miermont-Pitman:The_exploration_process_of_inhomogeneous_continuum_random_trees_and_an_extension_of_Jeulin_s_local_time_identity}.

This paper is organised as follows. In Section \ref{sec:abres_codages_bijections}, we first introduce the notion of labelled one-type and two-type trees and their encoding by functions, then we describe the Bouttier--Di Francesco--Guitter and Janson--Stef{\'a}nsson bijections. In Section \ref{sec:serpent_et_enonces}, we define the pair $(\exc, Z)$ and the Brownian map and we state our main results on the convergence of discrete paths. Section \ref{sec:epine} is a technical section in which we extend a ``backbone decomposition'' of Broutin \& Marckert \cite{Broutin-Marckert:Asymptotics_of_trees_with_a_prescribed_degree_sequence_and_applications}, the results are stated there and proved in Appendix \ref{sec:appendice_epine}. We prove the convergence of the pairs $(\mathcal{C}^\circ_n, \mathcal{L}^\circ_n)$ and $(H_n, L_n)$, which encode the labelled trees $(\T_n, \ell_n)$ and $(T_n, l_n)$ respectively, in Section \ref{sec:convergences_fonctions}. Then we prove Theorem \ref{thm:cv_carte} in section \ref{sec:carte_brownienne}. Finally, we apply our results to Boltzmann random maps in Section \ref{sec:Boltzmann}.

\paragraph{Acknowledgments}
I am deeply indebted to Gr\'egory Miermont for suggesting me to study this model, for several discussions during the preparation of this work and for pointing me out some inaccuracies in a first draft. Many thanks also to all the persons I have asked about Remark \ref{rem:moments_records_pont_echangeable}, and in particular Olivier H\'enard who then asked me every morning if I succeeded proving it until I tried another way (Corollary~\ref{cor:bon_evenement_tension_labels}). Finally, I am grateful to the two anonymous referees for their constructive remarks.

\section{Maps and trees}
\label{sec:abres_codages_bijections}

\subsection{Plane trees and their encoding with paths}

Let $\N = \{1, 2, \dots\}$ be the set of all positive integers, set $\N^0 = \{\varnothing\}$ and consider the set of words
\begin{equation*}
\U = \bigcup_{n \ge 0} \N^n.
\end{equation*}
For every $u = (u_1, \dots, u_n) \in \U$, we denote by $|u| = n$ the length of $u$; if $n \ge 1$, we define its \emph{prefix} $pr(u) = (u_1, \dots, u_{n-1})$ and for $v = (v_1, \dots, v_m) \in \U$, we let $uv = (u_1, \dots, u_n, v_1, \dots, v_m) \in \U$ be the concatenation of $u$ and $v$. We endow $\U$ with the \emph{lexicographical order}: given $u, v \in \U$, let $w \in \U$ be their longest common prefix, that is $u = w(u_1, \dots, u_n)$, $v = w(v_1, \dots, v_m)$ and $u_1 \ne v_1$, then $u < v$ if $u_1 < v_1$.

A \emph{plane tree} is a non-empty, finite subset $\tau \subset \U$ such that:
\begin{enumerate}
\item $\varnothing \in \tau$;
\item if $u \in \tau$ with $|u| \ge 1$, then $pr(u) \in \tau$;
\item if $u \in \tau$, then there exists an integer $k_u \ge 0$ such that $ui \in \tau$ if and only if $1 \le i \le k_u$.
\end{enumerate}

We shall denote the set of plane trees by $\tree$. We will view each vertex $u$ of a tree $\tau$ as an individual of a population for which $\tau$ is the genealogical tree. The vertex $\varnothing$ is called the \emph{root} of the tree and for every $u \in \tau$, $k_u$ is the number of \emph{children} of $u$ (if $k_u = 0$, then $u$ is called a \emph{leaf}, otherwise, $u$ is called an \emph{internal vertex}) and $u1, \dots, uk_u$ are these children from left to right, $|u|$ is its \emph{generation}, $pr(u)$ is its \emph{parent} and more generally, the vertices $u, pr(u), pr \circ pr (u), \dots, pr^{|u|}(u) = \varnothing$ are its \emph{ancestors}; the longest common prefix of two elements is their \emph{last common ancestor}. We shall denote by $\llbracket u , v \rrbracket$ the unique non-crossing path between $u$ and $v$.

Fix a tree $\tau$ with $N$ edges and let $\varnothing = u_0 < u_1 < \dots < u_N$ be its vertices, listed in lexicographical order. We describe three discrete paths which each encode $\tau$. First, its \emph{{\L}ukasiewicz path} $W = (W(j) ; 0 \le j \le N+1)$ is defined by $W(0) = 0$ and for every $0 \le j \le N$,
\[W(j+1) = W(j) + k_{u_j}-1.\]
One easily checks that $W(j) \ge 0$ for every $0 \le j \le N$ but $W(N+1)=-1$. Next, we define the \emph{height process} $H = (H(j); 0 \le j \le N)$ by setting for every $0 \le j \le N$,
\[H(j) = |u_j|.\]
Finally, define the \emph{contour sequence} $(c_0, c_1, \dots ,c_{2N})$ of $\tau$ as follows: $c_0 = \varnothing$ and for each $i \in \{0, \dots, 2N-1\}$, $c_{i+1}$ is either the first child of $c_i$ which does not appear in the sequence $(c_0, \dots, c_i)$, or the parent of $c_i$ if all its children already appear in this sequence. The lexicographical order on the tree corresponds to the depth-first search order, whereas the contour order corresponds to ``moving around the tree in clockwise order''. The \emph{contour process} $C = (C(j) ; 0 \le j \le 2N)$ is defined by setting for every $0 \le j \le 2N$,
\[C(j) = |c_j|.\]

Without further notice, throughout this work, every discrete path shall also be viewed as a continuous function after interpolating linearly between integer times.

\subsection{Labelled plane trees and label processes}
\label{subsec:arbres_etiquetes}

\paragraph{Two-type trees}
We will use the expression ``two-type tree'' for a plane tree in which we distinguish vertices at even and odd generation; call the former \emph{white} and the latter \emph{black}, we denote by $\circ(\T)$ and $\bullet(\T)$ the sets of white and black vertices of a two-type tree $\T$. We denote by $\tree_{\circ, \bullet}$ the set of two-type trees. Let $N$ be the number of edges of such a tree $\T$, denote by $(c_0, \dots, c_{2N})$ its contour sequence and $\mathcal{C} = (\mathcal{C}(k) ; 0 \le k \le 2N)$ its contour process; for every $0 \le k \le N$, set $c^\circ_k = c_{2k}$, the sequence $(c^\circ_0, \dots, c^\circ_N)$ is called the \emph{white contour sequence} of $\T$ and we define its \emph{white contour process} $\mathcal{C}^\circ = (\mathcal{C}^\circ(k) ; 0 \le k \le N)$ by $\mathcal{C}^\circ(k) = |c^\circ_k|/2$ for every $0 \le k \le N$. One easily sees that $\sup_{t \in [0,1]} |\mathcal{C}(2Nt) - 2\mathcal{C}^\circ(Nt)| = 1$ so $\mathcal{C}^\circ$ encodes the geometry of the tree up to a small error.

A \emph{labelling} $\ell$ of a two-type tree $\T$ is a function defined on the set $\circ(\T)$ of its white vertices to $\Z$ such that
\begin{itemize}
\item the root of $\T$ is labelled $0$,
\item for every black vertex, the increments of the labels of its white neighbours in clockwise order are greater than or equal to $-1$.
\end{itemize}
We define the \emph{white label process} $\mathcal{L}^\circ = (\mathcal{L}^\circ(k) ; 0 \le k \le N)$ of $\T$ by $\mathcal{L}^\circ(k) = \ell(c^\circ_k)$ for every $0 \le k \le N$. The labelled tree $(\T, \ell)$ is, up to a small error, encoded by the pair $(\mathcal{C}^\circ, \mathcal{L}^\circ)$, see Figure \ref{fig:arbre_deux_types}.

\begin{figure}[!ht]\centering
\def\r{.6}
\begin{tikzpicture}[scale=.7]
\coordinate (1) at (0,0);
	\coordinate (2) at (-1.5,1);
		\coordinate (3) at (-2.5,2);
		\coordinate (4) at (-1.5,2);
		\coordinate (5) at (-.5,2);
			\coordinate (6) at (-.5,3);
				\coordinate (7) at (-.5,4);
					\coordinate (8) at (-1.25,5);
					\coordinate (9) at (.25,5);
						\coordinate (10) at (-.75,6);
						\coordinate (11) at (.25,6);
						\coordinate (12) at (1.25,6);
	\coordinate (13) at (2,1);
		\coordinate (14) at (2,2);
			\coordinate (15) at (2,3);
				\coordinate (16) at (1.25,4);
				\coordinate (17) at (2.75,4);

\draw
	(1) -- (2)	(1) -- (13)
	(2) -- (3)	(2) -- (4)	(2) -- (5)
	(5) -- (6) -- (7)
	(7) -- (8)	(7) -- (9)
	(9) -- (10)	(9) -- (11)	(9) -- (12)
	(13) -- (14) -- (15)
	(15) -- (16)	(15) -- (17)
;

\draw[fill=black]
	(2) circle (3pt)
	(6) circle (3pt)
	(8) circle (3pt)
	(9) circle (3pt)
	(13) circle (3pt)
	(15) circle (3pt)
;

\begin{scriptsize}
\node[circle, minimum size=\r cm, fill=white, draw] at (1) {$0$};
\node[circle, minimum size=\r cm, fill=white, draw] at (3) {$-1$};
\node[circle, minimum size=\r cm, fill=white, draw] at (4) {$-2$};
\node[circle, minimum size=\r cm, fill=white, draw] at (5) {$1$};
\node[circle, minimum size=\r cm, fill=white, draw] at (7) {$0$};
\node[circle, minimum size=\r cm, fill=white, draw] at (10) {$-1$};
\node[circle, minimum size=\r cm, fill=white, draw] at (11) {$-2$};
\node[circle, minimum size=\r cm, fill=white, draw] at (12) {$-1$};
\node[circle, minimum size=\r cm, fill=white, draw] at (14) {$-1$};
\node[circle, minimum size=\r cm, fill=white, draw] at (16) {$-2$};
\node[circle, minimum size=\r cm, fill=white, draw] at (17) {$0$};

\end{scriptsize}
\end{tikzpicture}
\qquad
\begin{scriptsize}
\begin{tikzpicture}[scale=.55]
\draw[thin, ->]	(0,0) -- (16.5,0);
\draw[thin, ->]	(0,0) -- (0,3.5);
\foreach \x in {1, 2, ..., 3}
	\draw[dotted]	(0,\x) -- (16.5,\x);
\foreach \x in {0, 1, ..., 3}
	\draw (.1,\x)--(-.1,\x)	(0,\x) node[left] {$\x$};
\foreach \x in {2, 4, ..., 16}
	\draw (\x,.1)--(\x,-.1)	(\x,0) node[below] {$\x$};
\coordinate (0) at (0,0);
\coordinate (1) at (1, 1);
\coordinate (2) at (2, 1);
\coordinate (3) at (3, 1);
\coordinate (4) at (4, 2);
\coordinate (5) at (5, 2);
\coordinate (6) at (6, 3);
\coordinate (7) at (7, 3);
\coordinate (8) at (8, 3);
\coordinate (9) at (9, 2);
\coordinate (10) at (10, 1);
\coordinate (11) at (11, 0);
\coordinate (12) at (12, 1);
\coordinate (13) at (13, 2);
\coordinate (14) at (14, 2);
\coordinate (15) at (15, 1);
\coordinate (16) at (16, 0);
\newcommand{\lastx}{0}
\foreach \x [remember=\x as \lastx] in {1, 2, 3, ..., 16} \draw (\lastx) -- (\x);
\foreach \x in {0, 1, ..., 16}	\draw[fill=black] (\x)	circle (2pt);
\begin{scope}[shift={(0,-3)}]
\draw[thin, ->]	(0,0) -- (16.5,0);
\draw[thin, ->]	(0,-2) -- (0,1.5);
\foreach \x in {-2, -1, 1}
	\draw[dotted]	(0,\x) -- (16.5,\x);
\foreach \x in {-2, -1, 0, 1}
	\draw (.1,\x)--(-.1,\x)	(0,\x) node[left] {$\x$};
%
\coordinate (0) at (0, 0);
\coordinate (1) at (1, -1);
\coordinate (2) at (2, -2);
\coordinate (3) at (3, 1);
\coordinate (4) at (4, 0);
\coordinate (5) at (5, 0);
\coordinate (6) at (6, -1);
\coordinate (7) at (7, -2);
\coordinate (8) at (8, -1);
\coordinate (9) at (9, 0);
\coordinate (10) at (10, 1);
\coordinate (11) at (11, 0);
\coordinate (12) at (12, -1);
\coordinate (13) at (13, -2);
\coordinate (14) at (14, 0);
\coordinate (15) at (15, -1);
\coordinate (16) at (16, 0);
\renewcommand{\lastx}{0}
\foreach \x [remember=\x as \lastx] in {1, 2, 3, ..., 16} \draw (\lastx) -- (\x);

\foreach \x in {0, 1, 2, 3, ..., 16} \draw [fill=black] (\x)	circle (2pt);
\end{scope}
\end{tikzpicture}
\end{scriptsize}
\caption{A two-type labelled tree, its white contour process on top and its white label process below.}
\label{fig:arbre_deux_types}
\end{figure}
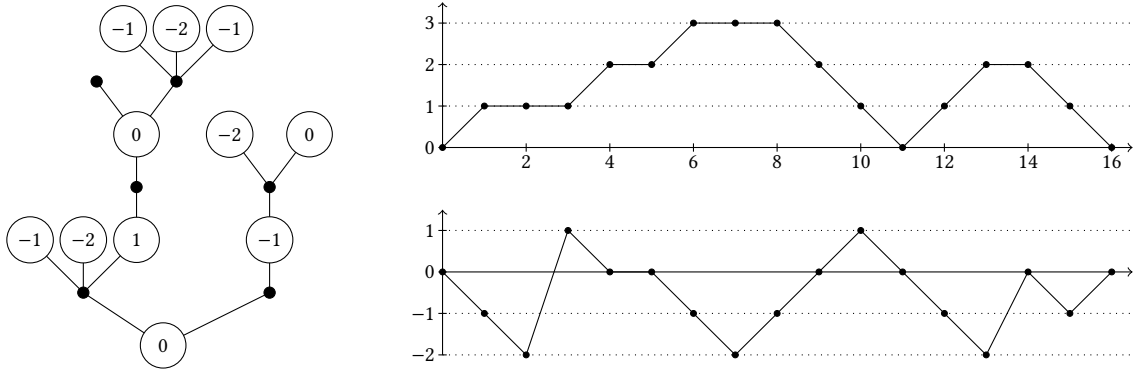

\paragraph{One-type trees}
As opposed to two-type trees, plane trees in which vertices at even and odd generation play the same role will be called ``one-type trees'' and not just ``trees'' to emphasise the difference. Recall that the geometry of a one-type tree $T$ is encoded by its height process $H$. A \emph{labelling} $l$ of such a tree is a function defined on the set of vertices to $\Z$ such that
\begin{itemize}
\item the root of $T$ is labelled $0$,
\item for every internal vertex, its right-most child carries the same label as itself,
\item for every internal vertex, the label increment between itself and its first child is greater than or equal to $-1$, and so are the increments between every two consecutive children from left to right.
\end{itemize}
Define the \emph{label process} $L(k) = l(u_k)$, where $(u_0, \dots, u_N)$ is the sequence of vertices of $T$ in lexicographical order; the labelled tree is (exactly) encoded by the pair $(H, L)$, see Figure \ref{fig:arbre_un_type}.

\begin{figure}[!ht]\centering
\def\r{.6}
\def\longueur{1.8}
\begin{tikzpicture}[scale=.6]
\coordinate (1) at (0,0*\longueur);
	\coordinate (2) at (-4.25,1*\longueur);
	\coordinate (3) at (-1.5,1*\longueur);
	\coordinate (4) at (1.5,1*\longueur);
		\coordinate (5) at (.75,2*\longueur);
			\coordinate (6) at (.75,3*\longueur);
				\coordinate (7) at (-.75,4*\longueur);
				\coordinate (8) at (.25,4*\longueur);
				\coordinate (9) at (1.25,4*\longueur);
				\coordinate (10) at (2.25,4*\longueur);
		\coordinate (11) at (2.25,2*\longueur);
	\coordinate (12) at (4.25,1*\longueur);
		\coordinate (13) at (3.5,2*\longueur);
			\coordinate (14) at (2.5,3*\longueur);
			\coordinate (15) at (3.5,3*\longueur);
			\coordinate (16) at (4.5,3*\longueur);
		\coordinate (17) at (5,2*\longueur);

\draw
	(1) -- (2)	(1) -- (3)	(1) -- (4)	(1) -- (12)
	(4) -- (5)	(4) -- (11)
	(5) -- (6)
	(6) -- (7)	(6) -- (8)	(6) -- (9)	(6) -- (10)
	(12) -- (13)	(12) -- (17)
	(13) -- (14)	(13) -- (15)	(13) -- (16)
;

\begin{scriptsize}
\node[circle, minimum size=\r cm, fill=white, draw] at (2) {$-1$};
\node[circle, minimum size=\r cm, fill=white, draw] at (3) {$-2$};
\node[circle, minimum size=\r cm, fill=white, draw] at (7) {$-1$};
\node[circle, minimum size=\r cm, fill=white, draw] at (8) {$-2$};
\node[circle, minimum size=\r cm, fill=white, draw] at (9) {$-1$};
\node[circle, minimum size=\r cm, fill=white, draw] at (10) {$0$};
\node[circle, minimum size=\r cm, fill=white, draw] at (11) {$1$};
\node[circle, minimum size=\r cm, fill=white, draw] at (14) {$-2$};
\node[circle, minimum size=\r cm, fill=white, draw] at (15) {$0$};
\node[circle, minimum size=\r cm, fill=white, draw] at (16) {$-1$};
\node[circle, minimum size=\r cm, fill=white, draw] at (17) {$0$};

\node[circle, minimum size=\r cm, fill=white, draw] at (1) {$0$};
\node[circle, minimum size=\r cm, fill=white, draw] at (4) {$1$};
\node[circle, minimum size=\r cm, fill=white, draw] at (5) {$0$};
\node[circle, minimum size=\r cm, fill=white, draw] at (6) {$0$};
\node[circle, minimum size=\r cm, fill=white, draw] at (12) {$0$};
\node[circle, minimum size=\r cm, fill=white, draw] at (13) {$-1$};
\end{scriptsize}
\end{tikzpicture}
\begin{scriptsize}
\begin{tikzpicture}[scale=.55]
\draw[thin, ->]	(0,0) -- (16.5,0);
\draw[thin, ->]	(0,0) -- (0,4.5);
\foreach \x in {1, 2, ..., 4}
	\draw[dotted]	(0,\x) -- (16.5,\x);
\foreach \x in {0, 1, ..., 4}
	\draw (.1,\x)--(-.1,\x)	(0,\x) node[left] {$\x$};
\foreach \x in {2, 4, ..., 16}
	\draw (\x,.1)--(\x,-.1)	(\x,0) node[below] {$\x$};
\coordinate (0) at (0, 0);
\coordinate (1) at (1, 1);
\coordinate (2) at (2, 1);
\coordinate (3) at (3, 1);
\coordinate (4) at (4, 2);
\coordinate (5) at (5, 3);
\coordinate (6) at (6, 4);
\coordinate (7) at (7, 4);
\coordinate (8) at (8, 4);
\coordinate (9) at (9, 4);
\coordinate (10) at (10, 2);
\coordinate (11) at (11, 1);
\coordinate (12) at (12, 2);
\coordinate (13) at (13, 3);
\coordinate (14) at (14, 3);
\coordinate (15) at (15, 3);
\coordinate (16) at (16, 2);
\newcommand{\lastx}{0}
\foreach \x [remember=\x as \lastx] in {1, 2, 3, ..., 16} \draw (\lastx) -- (\x);

\foreach \x in {0, 1, 2, 3, ..., 16} \draw [fill=black] (\x)	circle (2pt);
\begin{scope}[shift={(0,-2.5)}]
\draw[thin, ->]	(0,0) -- (16.5,0);
\draw[thin, ->]	(0,-2) -- (0,1.5);
\foreach \x in {-2, -1, 1}
	\draw[dotted]	(0,\x) -- (16.5,\x);
\foreach \x in {-2, -1, 0, 1}
	\draw (.1,\x)--(-.1,\x)	(0,\x) node[left] {$\x$};
%
\coordinate (0) at (0, 0);
\coordinate (1) at (1, -1);
\coordinate (2) at (2, -2);
\coordinate (3) at (3, 1);
\coordinate (4) at (4, 0);
\coordinate (5) at (5, 0);
\coordinate (6) at (6, -1);
\coordinate (7) at (7, -2);
\coordinate (8) at (8, -1);
\coordinate (9) at (9, 0);
\coordinate (10) at (10, 1);
\coordinate (11) at (11, 0);
\coordinate (12) at (12, -1);
\coordinate (13) at (13, -2);
\coordinate (14) at (14, 0);
\coordinate (15) at (15, -1);
\coordinate (16) at (16, 0);
\renewcommand{\lastx}{0}
\foreach \x [remember=\x as \lastx] in {1, 2, 3, ..., 16} \draw (\lastx) -- (\x);

\foreach \x in {0, 1, 2, 3, ..., 16} \draw [fill=black] (\x)	circle (2pt);
\end{scope}
\end{tikzpicture}
\end{scriptsize}
\caption{A one-type labelled tree, its height process on top and its label process below.}
\label{fig:arbre_un_type}
\end{figure}
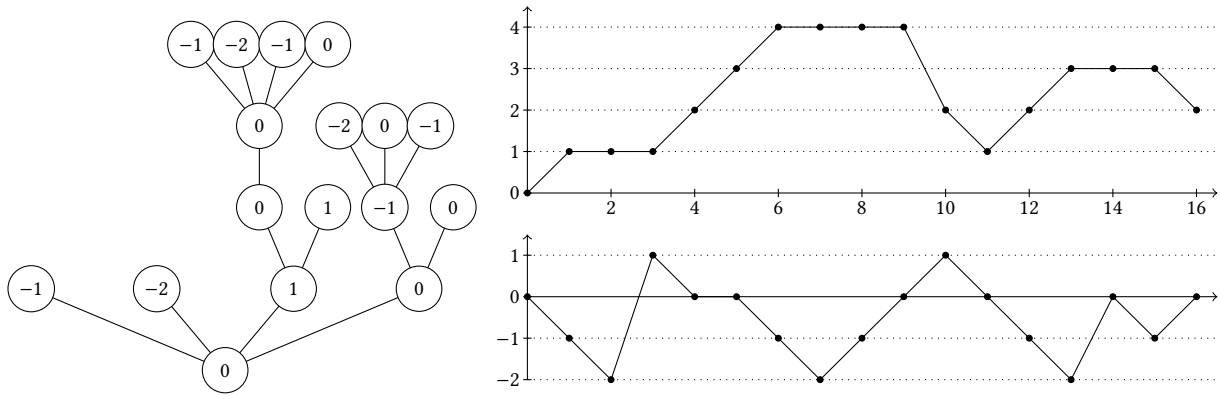

\paragraph{Notational remark}
We use roman letters $T$, $l$, $H$, $L$ for one-type trees and calligraphic letters $\T$, $\ell$, $\mathcal{C}$, $\mathcal{L}$ for two-type trees. We stress also that we consider the \emph{contour} order for two-type trees and the \emph{lexicographical} order for one-type trees.

\subsection{The Bouttier--Di Francesco--Guitter bijection}
\label{subsec:BDG}

A map is said to be \emph{pointed} if a vertex is distinguished. Given a sequence $\n$ of non-negative integers, we denote by $\barMap(\n)$ the set of rooted and pointed planar maps with $n_i$ faces with degree $2i$ for every $i \ge 1$. Let $\tree_{\circ, \bullet}(\n)$ denote the set of two-type trees with $n_i$ black vertices with degree $i$ for every $i \ge 1$; note that such a tree has $n_0$ white vertices and $N_\n$ edges, which are both defined in \eqref{eq:nombre_aretes_sites_carte}. Let further $\bartree_{\circ, \bullet}(\n)$ be the set of such labelled two-type trees.

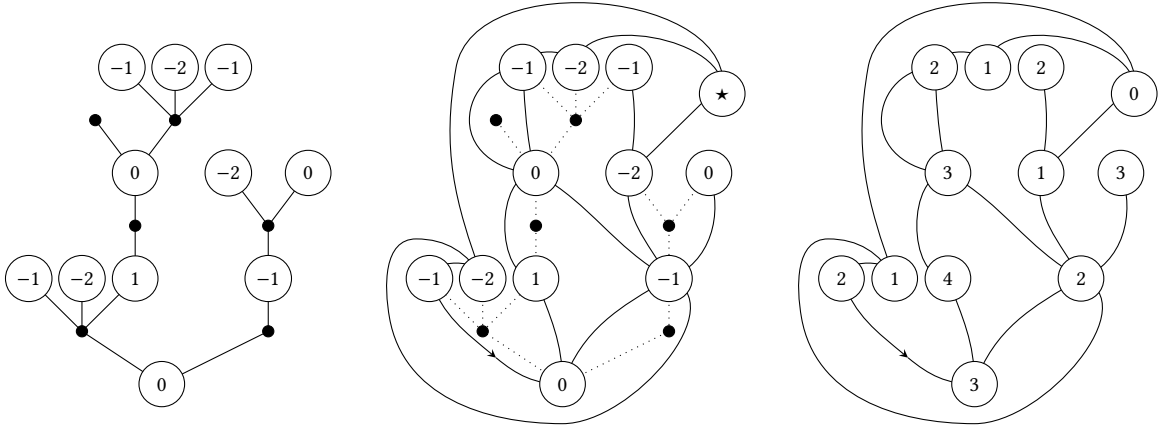
\begin{figure}[!ht] \centering
\def\r{.6}
\def\longueur{1}
\begin{scriptsize}
\begin{tikzpicture}[scale=.7]
\coordinate (1) at (0,0*\longueur);
	\coordinate (2) at (-1.5,1*\longueur);
		\coordinate (3) at (-2.5,2*\longueur);
		\coordinate (4) at (-1.5,2*\longueur);
		\coordinate (5) at (-.5,2*\longueur);
			\coordinate (6) at (-.5,3*\longueur);
				\coordinate (7) at (-.5,4*\longueur);
					\coordinate (8) at (-1.25,5*\longueur);
					\coordinate (9) at (.25,5*\longueur);
						\coordinate (10) at (-.75,6*\longueur);
						\coordinate (11) at (.25,6*\longueur);
						\coordinate (12) at (1.25,6*\longueur);
	\coordinate (13) at (2,1*\longueur);
		\coordinate (14) at (2,2*\longueur);
			\coordinate (15) at (2,3*\longueur);
				\coordinate (16) at (1.25,4*\longueur);
				\coordinate (17) at (2.75,4*\longueur);

\draw
	(1) -- (2)	(1) -- (13)
	(2) -- (3)	(2) -- (4)	(2) -- (5)
	(5) -- (6) -- (7)
	(7) -- (8)	(7) -- (9)
	(9) -- (10)	(9) -- (11)	(9) -- (12)
	(13) -- (14) -- (15)
	(15) -- (16)	(15) -- (17)
;

\draw[white] 	(14) to [out=-20,in=0] (0,-.75*\longueur) to [out=180,in=-80] (-3.25,1.25*\longueur) to [out=100,in=180] (-2.75,2.75*\longueur) to [out=0,in=120] (4);

\draw[fill=black]
	(2) circle (3pt)
	(6) circle (3pt)
	(8) circle (3pt)
	(9) circle (3pt)
	(13) circle (3pt)
	(15) circle (3pt)
;

\node[circle, minimum size=\r cm, fill=white, draw] at (1) {$0$};
\node[circle, minimum size=\r cm, fill=white, draw] at (3) {$-1$};
\node[circle, minimum size=\r cm, fill=white, draw] at (4) {$-2$};
\node[circle, minimum size=\r cm, fill=white, draw] at (5) {$1$};
\node[circle, minimum size=\r cm, fill=white, draw] at (7) {$0$};
\node[circle, minimum size=\r cm, fill=white, draw] at (10) {$-1$};
\node[circle, minimum size=\r cm, fill=white, draw] at (11) {$-2$};
\node[circle, minimum size=\r cm, fill=white, draw] at (12) {$-1$};
\node[circle, minimum size=\r cm, fill=white, draw] at (14) {$-1$};
\node[circle, minimum size=\r cm, fill=white, draw] at (16) {$-2$};
\node[circle, minimum size=\r cm, fill=white, draw] at (17) {$0$};

\end{tikzpicture}
\qquad
\begin{tikzpicture}[scale=.7]
\coordinate (0) at (3,5.5*\longueur);

\coordinate (1) at (0,0*\longueur);
	\coordinate (2) at (-1.5,1*\longueur);
		\coordinate (3) at (-2.5,2*\longueur);
		\coordinate (4) at (-1.5,2*\longueur);
		\coordinate (5) at (-.5,2*\longueur);
			\coordinate (6) at (-.5,3*\longueur);
				\coordinate (7) at (-.5,4*\longueur);
					\coordinate (8) at (-1.25,5*\longueur);
					\coordinate (9) at (.25,5*\longueur);
						\coordinate (10) at (-.75,6*\longueur);
						\coordinate (11) at (.25,6*\longueur);
						\coordinate (12) at (1.25,6*\longueur);
	\coordinate (13) at (2,1*\longueur);
		\coordinate (14) at (2,2*\longueur);
			\coordinate (15) at (2,3*\longueur);
				\coordinate (16) at (1.25,4*\longueur);
				\coordinate (17) at (2.75,4*\longueur);

\draw
	(1) to [out=180,in=-40] (-1.25,.5*\longueur)
	(3) to [out=90,in=90] (4) 
	(4) to [out=110,in=260] (-2,6*\longueur) to [out=80,in=80] (0)
	(5) to [out=180,in=180] (7)
	(7) to [out=180,in=-90] (-1.75,5*\longueur) to [out=90,in=180] (10)
	(7) to [out=110,in=90] (10)
	(10) to [out=90,in=90] (11)
	(11) to [out=90,in=80] (0)
	(12) to [out=-80,in=80] (16)
	(7) to [out=-30,in=150] (14)
	(5) to [out=-70,in=90] (1)
	(1) to [out=70,in=-150] (14)
	(14) to [out=120,in=-100] (16)
	(16) to [out=40,in=60] (0)
	(17) to [out=-70,in=20] (14)
	(14) to [out=-20,in=0] (0,-.75*\longueur) to [out=180,in=-80] (-3.25,1.25*\longueur) to [out=100,in=180] (-2.75,2.75*\longueur) to [out=0,in=120] (4)
;

\draw[<-, >=stealth]
	(-1.25,.5*\longueur) to [out=140,in=-70] (3);

\draw[dotted]
	(1) -- (2)	(1) -- (13)
	(2) -- (3)	(2) -- (4)	(2) -- (5)
	(5) -- (6) -- (7)
	(7) -- (8)	(7) -- (9)
	(9) -- (10)	(9) -- (11)	(9) -- (12)
	(13) -- (14) -- (15)
	(15) -- (16)	(15) -- (17)
;

\draw[fill=black]
	(2) circle (3pt)
	(6) circle (3pt)
	(8) circle (3pt)
	(9) circle (3pt)
	(13) circle (3pt)
	(15) circle (3pt)
;

\node[circle, minimum size=\r cm, fill=white, draw] at (0) {$\star$};
\node[circle, minimum size=\r cm, fill=white, draw] at (1) {$0$};
\node[circle, minimum size=\r cm, fill=white, draw] at (3) {$-1$};
\node[circle, minimum size=\r cm, fill=white, draw] at (4) {$-2$};
\node[circle, minimum size=\r cm, fill=white, draw] at (5) {$1$};
\node[circle, minimum size=\r cm, fill=white, draw] at (7) {$0$};
\node[circle, minimum size=\r cm, fill=white, draw] at (10) {$-1$};
\node[circle, minimum size=\r cm, fill=white, draw] at (11) {$-2$};
\node[circle, minimum size=\r cm, fill=white, draw] at (12) {$-1$};
\node[circle, minimum size=\r cm, fill=white, draw] at (14) {$-1$};
\node[circle, minimum size=\r cm, fill=white, draw] at (16) {$-2$};
\node[circle, minimum size=\r cm, fill=white, draw] at (17) {$0$};

\end{tikzpicture}
\qquad
\begin{tikzpicture}[scale=.7]
\coordinate (0) at (3,5.5*\longueur);

\coordinate (1) at (0,0*\longueur);
	\coordinate (2) at (-1.5,1*\longueur);
		\coordinate (3) at (-2.5,2*\longueur);
		\coordinate (4) at (-1.5,2*\longueur);
		\coordinate (5) at (-.5,2*\longueur);
			\coordinate (6) at (-.5,3*\longueur);
				\coordinate (7) at (-.5,4*\longueur);
					\coordinate (8) at (-1.25,5*\longueur);
					\coordinate (9) at (.25,5*\longueur);
						\coordinate (10) at (-.75,6*\longueur);
						\coordinate (11) at (.25,6*\longueur);
						\coordinate (12) at (1.25,6*\longueur);
	\coordinate (13) at (2,1*\longueur);
		\coordinate (14) at (2,2*\longueur);
			\coordinate (15) at (2,3*\longueur);
				\coordinate (16) at (1.25,4*\longueur);
				\coordinate (17) at (2.75,4*\longueur);

\draw
	(1) to [out=180,in=-40] (-1.25,.5*\longueur)
	(3) to [out=90,in=90] (4) 
	(4) to [out=110,in=260] (-2,6*\longueur) to [out=80,in=80] (0)
	(5) to [out=180,in=180] (7)
	(7) to [out=180,in=-90] (-1.75,5*\longueur) to [out=90,in=180] (10)
	(7) to [out=110,in=90] (10)
	(10) to [out=90,in=90] (11)
	(11) to [out=90,in=80] (0)
	(12) to [out=-80,in=80] (16)
	(7) to [out=-30,in=150] (14)
	(5) to [out=-70,in=90] (1)
	(1) to [out=70,in=-150] (14)
	(14) to [out=120,in=-100] (16)
	(16) to [out=40,in=60] (0)
	(17) to [out=-70,in=20] (14)
	(14) to [out=-20,in=0] (0,-.75*\longueur) to [out=180,in=-80] (-3.25,1.25*\longueur) to [out=100,in=180] (-2.75,2.75*\longueur) to [out=0,in=120] (4)
;

\draw[<-, >=stealth]
	(-1.25,.5*\longueur) to [out=140,in=-70] (3);

\node[circle, minimum size=\r cm, fill=white, draw] at (0) {$0$};
\node[circle, minimum size=\r cm, fill=white, draw] at (1) {$3$};
\node[circle, minimum size=\r cm, fill=white, draw] at (3) {$2$};
\node[circle, minimum size=\r cm, fill=white, draw] at (4) {$1$};
\node[circle, minimum size=\r cm, fill=white, draw] at (5) {$4$};
\node[circle, minimum size=\r cm, fill=white, draw] at (7) {$3$};
\node[circle, minimum size=\r cm, fill=white, draw] at (10) {$2$};
\node[circle, minimum size=\r cm, fill=white, draw] at (11) {$1$};
\node[circle, minimum size=\r cm, fill=white, draw] at (12) {$2$};
\node[circle, minimum size=\r cm, fill=white, draw] at (14) {$2$};
\node[circle, minimum size=\r cm, fill=white, draw] at (16) {$1$};
\node[circle, minimum size=\r cm, fill=white, draw] at (17) {$3$};
\end{tikzpicture}
\caption{The Bouttier--Di Francesco--Guitter bijection.}
\label{fig:BDG}
\end{scriptsize}
\end{figure}

Bouttier, Di Francesco \& Guitter \cite{Bouttier-Di_Francesco-Guitter:Planar_maps_as_labeled_mobiles} show that $\barMap(\n)$ and $\{-1, +1\} \times \bartree_{\circ, \bullet}(\n)$ are in bijection, we shall refer to it as the $\BDG$ bijection. Let us only recall how a map is constructed from a labelled two-type tree $(\T, \ell)$, as depicted by Figure \ref{fig:BDG}. Let $N$ be the number of edges of $\T$, we write $(c^\circ_0, \dots, c^\circ_N)$ for its white contour sequence and we adopt the convention that $c^\circ_{N+i} = c^\circ_i$ for every $0 \le i \le N$. A white \emph{corner} is a sector around a white vertex delimited by two consecutive edges; there are $N$ white corners, corresponding to the vertices $c^\circ_0, \dots, c^\circ_{N-1}$; for every $0 \le i \le 2N$ we denote by $e_i$ the corner corresponding to $c^\circ_i$. We add an extra vertex $\star$ outside the tree $\T$ and construct a map on the vertex-set of $\T$ and $\star$ by drawing edges as follows: for every $0 \le i \le N-1$,
\begin{itemize}
\item if $\ell(c^\circ_i) > \min_{0 \le k \le N-1} \ell(c^\circ_k)$, then we draw an edge between $e_i$ and $e_j$ where $j = \min\{k > i: \ell(c^\circ_k) = \ell(c^\circ_i)-1\}$,
\item if $\ell(c^\circ_i) = \min_{0 \le k \le N-1} \ell(c^\circ_k)$, then we draw an edge between $e_i$ and $\star$.
\end{itemize}
It is shown in \cite{Bouttier-Di_Francesco-Guitter:Planar_maps_as_labeled_mobiles} that this procedure indeed produces a planar map $\map$, pointed at $\star$, and rooted at the first edge that we drew, for $i=0$, oriented according to an external choice $\epsilon \in \{-1, +1\}$ and, further, that this operation is invertible. Observe that $\map$ has $N$ edges, as many as $\T$, and that the faces of $\map$ correspond to the black vertices of $\T$; one can check that the degree of a face is twice that of the corresponding black vertex, we conclude that the above procedure indeed realises a bijection between $\barMap(\n)$ and $\{-1, +1\} \times \bartree_{\circ, \bullet}(\n)$. One may be concerned with the fact that the vertices of $\map$ different from $\star$ are labelled, which seems at first sight to be an extra information; shift these labels by adding to each the quantity $1-\min_{c^\circ \in \circ(\T)} \ell(c^\circ)$ and label $0$ the vertex $\star$, then the label of each vertex corresponds to its graph distance in $\map$ to the origin $\star$.

\subsection{The Janson--Stef{\'a}nsson bijection}

Let $\tree(\n)$ denote the set of one-type trees possessing $n_i$ vertices with $i$ children for every $i \ge 0$; note that such a tree has $N_\n$ edges and that $p_\n$ defined in Section \ref{sec:resultat} is its empirical offspring distribution. Uniform random trees in $\tree(\n)$ have been studied by Addario-Berry \cite{Addario_Berry:Tail_bounds_for_the_height_and_width_of_a_random_tree_with_a_given_degree_sequence} who obtained uniform sub-Gaussian tail bounds for their height and width and Broutin \& Marckert \cite{Broutin-Marckert:Asymptotics_of_trees_with_a_prescribed_degree_sequence_and_applications} who showed that, properly rescaled, under our assumption \eqref{eq:H}, they converge in distribution in the sense of Gromov--Hausdorff, towards the celebrated Brownian tree, see \eqref{eq:Broutin_Marckert} below.

Janson \& Stef{\'a}nsson \cite{Janson-Stefansson:Scaling_limits_of_random_planar_maps_with_a_unique_large_face} show that $\tree(\n)$ and $\tree_{\circ, \bullet}(\n)$ are in bijection, we shall refer to it as the $\JS$ bijection. In this bijection, the white vertices of the tree in $\tree_{\circ, \bullet}(\n)$ are mapped onto the leaves of the tree in $\tree(\n)$ and the black vertices in the former, with degree $k \ge 1$, are mapped onto (internal) vertices of the latter with $k$ children. Let us recall the construction of this bijection in the two directions.

Let us start with a two-type tree $\T$; we construct a one-type tree $T$ with the same vertex-set as follows. First, if $\T = \{\varnothing\}$ is a singleton, then set $T = \{\varnothing\}$; otherwise, for every white vertex $u \in \circ(\T)$, do the following:
\begin{itemize}
\item if $u$ is a leaf of $\T$, then draw an edge between $u$ and $pr(u)$;
\item if $u$ is an internal vertex, with $k_u \ge 1$ children, then draw edges between any two consecutive black children $u1$ and $u2$, $u2$ and $u3$, \dots, $u(k_u-1)$ and $uk_u$, draw also an edge between $u$ and $uk_u$;
\item if furthermore $u \ne \varnothing$, then draw an edge between its first child $u1$ and its parent $pr(u)$ in the first corner at the left of the edge between $u$ and $pr(u)$.
\end{itemize}
We root the new tree $T$ at the first child of the root of $\T$. See Figure \ref{fig:JS_2_1} for an illustration.

\begin{figure}[!ht] \centering
\begin{tikzpicture}[scale=.7]
\coordinate (0) at (0,-1);
\coordinate (1) at (0,0);
	\coordinate (2) at (-1.5,1);
		\coordinate (3) at (-2.25,2);
		\coordinate (4) at (-1.5,2);
		\coordinate (5) at (-.75,2);
			\coordinate (6) at (-.75,3);
				\coordinate (7) at (-.75,4);
					\coordinate (8) at (-1.25,5);
					\coordinate (9) at (-.25,5);
						\coordinate (10) at (-.75,6);
						\coordinate (11) at (-.25,6);
						\coordinate (12) at (.25,6);
	\coordinate (13) at (2,1);
		\coordinate (14) at (2,2);
			\coordinate (15) at (2,3);
				\coordinate (16) at (1.25,4);
				\coordinate (17) at (2.75,4);

\draw
	(1) -- (2)	(1) -- (13)
	(2) -- (3)	(2) -- (4)	(2) -- (5)
	(5) -- (6) -- (7)
	(7) -- (8)	(7) -- (9)
	(9) -- (10)	(9) -- (11)	(9) -- (12)
	(13) -- (14) -- (15)
	(15) -- (16)	(15) -- (17)
;

\draw[fill=white]
	(1) circle (3pt)
	(3) circle (3pt)
	(4) circle (3pt)
	(5) circle (3pt)
	(7) circle (3pt)
	(10) circle (3pt)
	(11) circle (3pt)
	(12) circle (3pt)
	(14) circle (3pt)
	(16) circle (3pt)
	(17) circle (3pt)
;
\draw[fill=black]
	(2) circle (3pt)
	(6) circle (3pt)
	(8) circle (3pt)
	(9) circle (3pt)
	(13) circle (3pt)
	(15) circle (3pt)
;
\end{tikzpicture}
\qquad
%
\begin{tikzpicture}[scale=.7]
\coordinate (0) at (0,-1);
\coordinate (1) at (0,0);
	\coordinate (2) at (-1.5,1);
		\coordinate (3) at (-2.25,2);
		\coordinate (4) at (-1.5,2);
		\coordinate (5) at (-.75,2);
			\coordinate (6) at (-.75,3);
				\coordinate (7) at (-.75,4);
					\coordinate (8) at (-1.25,5);
					\coordinate (9) at (-.25,5);
						\coordinate (10) at (-.75,6);
						\coordinate (11) at (-.25,6);
						\coordinate (12) at (.25,6);
	\coordinate (13) at (2,1);
		\coordinate (14) at (2,2);
			\coordinate (15) at (2,3);
				\coordinate (16) at (1.25,4);
				\coordinate (17) at (2.75,4);

\draw[dotted]
	(1) -- (2)	(1) -- (13)
	(2) -- (3)	(2) -- (4)	(2) -- (5)
	(5) -- (6) -- (7)
	(7) -- (8)	(7) -- (9)
	(9) -- (10)	(9) -- (11)	(9) -- (12)
	(13) -- (14) -- (15)
	(15) -- (16)	(15) -- (17)
;

\draw
	(2) to[bend left] (13) to[bend left] (1)
	(2) to[bend left] (3)
	(2) to[bend left] (4)
	(2) to[bend left=10] (6) to[bend left] (5)
	(6) to[bend left] (8) to[bend left] (9) to[bend left] (7)
	(9) to[bend left] (10)
	(9) to[bend left] (11)
	(9) to[bend left] (12)
	(13) to[bend left] (15) to[bend left] (14)
	(15) to[bend left] (16)
	(15) to[bend left] (17)	
;
\draw[thick, <-, >=stealth] (-1.75,.75) -- ++ (-.5,-.5);

\draw[fill=white]
	(1) circle (3pt)
	(3) circle (3pt)
	(4) circle (3pt)
	(5) circle (3pt)
	(7) circle (3pt)
	(10) circle (3pt)
	(11) circle (3pt)
	(12) circle (3pt)
	(14) circle (3pt)
	(16) circle (3pt)
	(17) circle (3pt)
;
\draw[fill=black]
	(2) circle (3pt)
	(6) circle (3pt)
	(8) circle (3pt)
	(9) circle (3pt)
	(13) circle (3pt)
	(15) circle (3pt)
;
\end{tikzpicture}
\qquad
%
\begin{tikzpicture}[scale=.7]
\coordinate (0) at (0,-1);
\coordinate (1) at (0,0);
	\coordinate (2) at (-2.25,1);
	\coordinate (3) at (-.75,1);
	\coordinate (4) at (.75,1);
		\coordinate (5) at (.25,2);
			\coordinate (6) at (.25,3);
				\coordinate (7) at (-.5,4);
				\coordinate (8) at (0,4);
				\coordinate (9) at (.5,4);
				\coordinate (10) at (1,4);
		\coordinate (11) at (1.25,2);
	\coordinate (12) at (2.25,1);
		\coordinate (13) at (1.75,2);
			\coordinate (14) at (1.25,3);
			\coordinate (15) at (1.75,3);
			\coordinate (16) at (2.25,3);
		\coordinate (17) at (2.75,2);

\draw
	(1) -- (2)	(1) -- (3)	(1) -- (4)	(1) -- (12)
	(4) -- (5)	(4) -- (11)
	(5) -- (6)
	(6) -- (7)	(6) -- (8)	(6) -- (9)	(6) -- (10)
	(12) -- (13)	(12) -- (17)
	(13) -- (14)	(13) -- (15)	(13) -- (16)
;

\draw[fill=white]
	(2) circle (3pt)
	(3) circle (3pt)
	(7) circle (3pt)
	(8) circle (3pt)
	(9) circle (3pt)
	(10) circle (3pt)
	(11) circle (3pt)
	(14) circle (3pt)
	(15) circle (3pt)
	(16) circle (3pt)
	(17) circle (3pt)
;
\draw[fill=black]
	(1) circle (3pt)
	(4) circle (3pt)
	(5) circle (3pt)
	(6) circle (3pt)
	(12) circle (3pt)
	(13) circle (3pt)
;
\end{tikzpicture}
\caption{The Janson--Stef{\'a}nsson bijection from two-type trees to one-type trees.}
\label{fig:JS_2_1}
\end{figure}
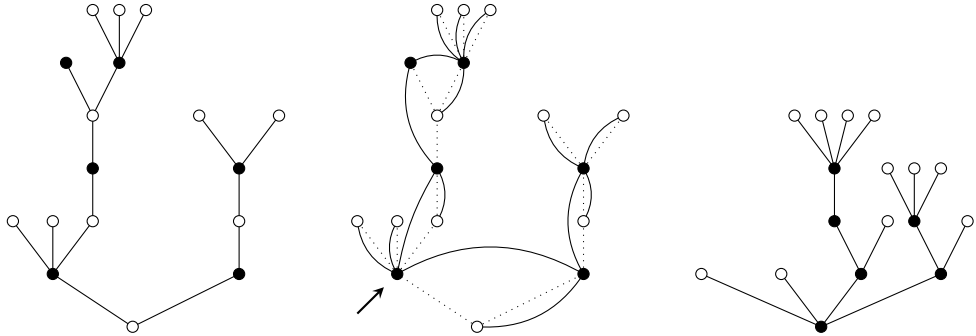

Conversely, given a one-type tree $T$, we construct a two-type tree $\T$ as follows. Again, set $\T = \{\varnothing\}$ whenever $T = \{\varnothing\}$; otherwise, for every leaf $u$ of $T$, denote by $u^\star$ its last ancestor whose last child is not an ancestor of $u$; formally set
\[u^\star = \sup\left\{w \in \llbracket \varnothing, u\llbracket : w k_w \notin \rrbracket \varnothing, u\rrbracket\right\}.\]
The set on the right may be empty, in which case $u^\star = \varnothing$ by convention. Then draw an edge between $u$ and every vertex $v \in \llbracket u^\star, u\llbracket$, in the first corner at the right of the edge between $v$ and its only child which belongs to $\rrbracket u^\star, u\rrbracket$. This yields a tree that we root at the last leaf of $T$. See Figure \ref{fig:JS_1_2} for an illustration. One can check that the two procedures are the inverse of one another.

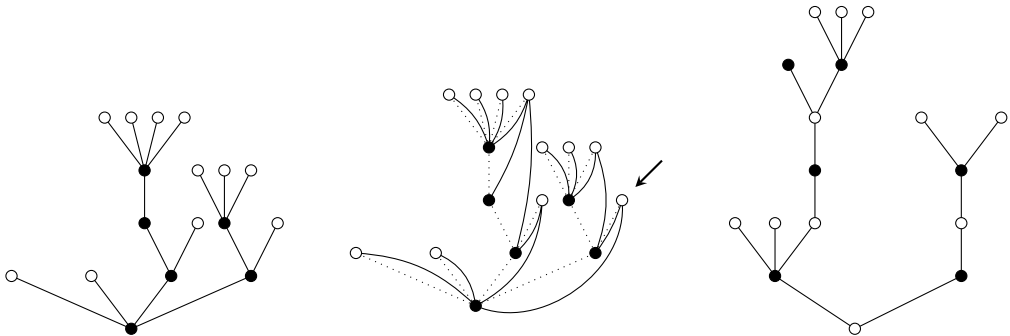
\begin{figure}[!ht] \centering
\begin{tikzpicture}[scale=.7]
\coordinate (0) at (0,-1);
\coordinate (1) at (0,0);
	\coordinate (2) at (-2.25,1);
	\coordinate (3) at (-.75,1);
	\coordinate (4) at (.75,1);
		\coordinate (5) at (.25,2);
			\coordinate (6) at (.25,3);
				\coordinate (7) at (-.5,4);
				\coordinate (8) at (0,4);
				\coordinate (9) at (.5,4);
				\coordinate (10) at (1,4);
		\coordinate (11) at (1.25,2);
	\coordinate (12) at (2.25,1);
		\coordinate (13) at (1.75,2);
			\coordinate (14) at (1.25,3);
			\coordinate (15) at (1.75,3);
			\coordinate (16) at (2.25,3);
		\coordinate (17) at (2.75,2);

\draw
	(1) -- (2)	(1) -- (3)	(1) -- (4)	(1) -- (12)
	(4) -- (5)	(4) -- (11)
	(5) -- (6)
	(6) -- (7)	(6) -- (8)	(6) -- (9)	(6) -- (10)
	(12) -- (13)	(12) -- (17)
	(13) -- (14)	(13) -- (15)	(13) -- (16)
;

\draw[fill=white]
	(2) circle (3pt)
	(3) circle (3pt)
	(7) circle (3pt)
	(8) circle (3pt)
	(9) circle (3pt)
	(10) circle (3pt)
	(11) circle (3pt)
	(14) circle (3pt)
	(16) circle (3pt)
	(17) circle (3pt)
	(15) circle (3pt)
;
\draw[fill=black]
	(1) circle (3pt)
	(4) circle (3pt)
	(5) circle (3pt)
	(6) circle (3pt)
	(12) circle (3pt)
	(13) circle (3pt)
;
\end{tikzpicture}
\qquad
%
\begin{tikzpicture}[scale=.7]
\coordinate (0) at (0,-1);
\coordinate (1) at (0,0);
	\coordinate (2) at (-2.25,1);
	\coordinate (3) at (-.75,1);
	\coordinate (4) at (.75,1);
		\coordinate (5) at (.25,2);
			\coordinate (6) at (.25,3);
				\coordinate (7) at (-.5,4);
				\coordinate (8) at (0,4);
				\coordinate (9) at (.5,4);
				\coordinate (10) at (1,4);
		\coordinate (11) at (1.25,2);
	\coordinate (12) at (2.25,1);
		\coordinate (13) at (1.75,2);
			\coordinate (14) at (1.25,3);
			\coordinate (15) at (1.75,3);
			\coordinate (16) at (2.25,3);
		\coordinate (17) at (2.75,2);

\draw[dotted]
	(1) -- (2)	(1) -- (3)	(1) -- (4)	(1) -- (12)
	(4) -- (5)	(4) -- (11)
	(5) -- (6)
	(6) -- (7)	(6) -- (8)	(6) -- (9)	(6) -- (10)
	(12) -- (13)	(12) -- (17)
	(13) -- (14)	(13) -- (15)	(13) -- (16)
;

\draw
	(2) to[bend left=20] (1)
	(3) to[bend left] (1)
	(7) to[bend left=20] (6)
	(8) to[bend left=20] (6)
	(9) to[bend left=20] (6)
	(10) to[bend left=20] (6)
	(10) to[bend left=10] (5)
	(10) to[bend left=10] (4)
	(11) to[bend left=20] (4)
	(11) to[bend left] (1)
	(14) to[bend left] (13)
	(15) to[bend left] (13)
	(16) to[bend left] (13)
	(16) to[bend left=20] (12)
	(17) to[bend left=10] (12)
	(17) to[bend left=60] (1)
;
\draw[thick, <-, >=stealth] (3, 2.25) -- ++ (.5,.5);

\draw[fill=white]
	(2) circle (3pt)
	(3) circle (3pt)
	(7) circle (3pt)
	(8) circle (3pt)
	(9) circle (3pt)
	(10) circle (3pt)
	(11) circle (3pt)
	(14) circle (3pt)
	(15) circle (3pt)
	(16) circle (3pt)
	(17) circle (3pt)
;
\draw[fill=black]
	(1) circle (3pt)
	(4) circle (3pt)
	(5) circle (3pt)
	(6) circle (3pt)
	(12) circle (3pt)
	(13) circle (3pt)
;
\end{tikzpicture}
\qquad
%
\begin{tikzpicture}[scale=.7]
\coordinate (0) at (0,-1);
\coordinate (1) at (0,0);
	\coordinate (2) at (-1.5,1);
		\coordinate (3) at (-2.25,2);
		\coordinate (4) at (-1.5,2);
		\coordinate (5) at (-.75,2);
			\coordinate (6) at (-.75,3);
				\coordinate (7) at (-.75,4);
					\coordinate (8) at (-1.25,5);
					\coordinate (9) at (-.25,5);
						\coordinate (10) at (-.75,6);
						\coordinate (11) at (-.25,6);
						\coordinate (12) at (.25,6);
	\coordinate (13) at (2,1);
		\coordinate (14) at (2,2);
			\coordinate (15) at (2,3);
				\coordinate (16) at (1.25,4);
				\coordinate (17) at (2.75,4);

\draw
	(1) -- (2)	(1) -- (13)
	(2) -- (3)	(2) -- (4)	(2) -- (5)
	(5) -- (6) -- (7)
	(7) -- (8)	(7) -- (9)
	(9) -- (10)	(9) -- (11)	(9) -- (12)
	(13) -- (14) -- (15)
	(15) -- (16)	(15) -- (17)
;
\draw[fill=white]
	(1) circle (3pt)
	(3) circle (3pt)
	(4) circle (3pt)
	(5) circle (3pt)
	(7) circle (3pt)
	(10) circle (3pt)
	(11) circle (3pt)
	(12) circle (3pt)
	(14) circle (3pt)
	(16) circle (3pt)
	(17) circle (3pt)
;
\draw[fill=black]
	(2) circle (3pt)
	(6) circle (3pt)
	(8) circle (3pt)
	(9) circle (3pt)
	(13) circle (3pt)
	(15) circle (3pt)
;
\end{tikzpicture}
\caption{The Janson--Stef{\'a}nsson bijection from one-type trees to two-type trees.}
\label{fig:JS_1_2}
\end{figure}

Let further $\bartree(\n)$ be the set of labelled one-type trees possessing $n_i$ vertices with $i$ children for every $i \ge 0$, the $\JS$ bijection extends to a bijection between $\bartree(\n)$ and $\bartree_{\circ, \bullet}(\n)$ if every black vertex of a two-type tree is given the label of its white parent. 
Let us explain how this bijection translates in terms of the processes encoding the labelled trees (one may look at Figures \ref{fig:arbre_deux_types} and \ref{fig:arbre_un_type} for an illustration). Fix $(\T, \ell)$ a two-type labelled tree and denote by $\mathcal{C}^\circ$ its white contour process and $\mathcal{L}^\circ$ its white label process (in contour order). Fix also $(T,l)$ a one-type labelled tree and denote by $H$ its height process and $L$ its label process (in lexicographical order). Finally, introduce a modified version of the height process: let $N$ be the number of edges of $T$ and $(u_0, \dots, u_N)$ be its vertices listed in lexicographical order; for each integer $j \in \{0, \dots, N\}$, we let $\widetilde{H}(j)$ denote the number of strict ancestors of $u_j$ whose last child is not an ancestor of $u_j$, i.e.
\[\widetilde{H}(j) = \#\left\{w \in \llbracket \varnothing, u_j\llbracket : w k_w \notin \rrbracket \varnothing, u_j\rrbracket\right\}.\]

\begin{lem}\label{lem:processus_bijection_JS}
If $(T, l)$ and $(\T, \ell)$ are related by the $\JS$ bijection, then \[\mathcal{L}^\circ = L
\qquad\text{and}\qquad
\mathcal{C}^\circ = \widetilde{H}.\]
\end{lem}

\begin{proof}
Let us first prove the equality of the label processes. We use the observation from \cite{Kortchemski-Marzouk:Simply_generated_non_crossing_partitions} that the lexicographical order on the vertices of $T$ corresponds to the contour order on the black corners of $\T$ which, by a shift, corresponds to the contour order on the white corners of $\T$. Specifically, let $N$ be the number of edges of both trees, fix $j \in \{0, \dots, N\}$ and consider the $j$-th white corner of $\T$: it is a sector around a white vertex delimited by two consecutive edges, whose other extremity is therefore black; consider the previous black corner in contour order, in the construction of the $\JS$ bijection, an edge of $T$ starts from this corner and we claim that the other extremity of this edge is $u_j$ the $j$-th vertex of $T$ in lexicographical order. We refer to the proof of Proposition 2.1 and Figure 4 in \cite{Kortchemski-Marzouk:Simply_generated_non_crossing_partitions}.

It follows that if $c^\circ_j \in \circ(\T)$ is the white vertex of $\T$ visited at the $j$-th step in the white contour sequence, then the image of $u_j$ by the $\JS$ bijection is
\begin{itemize}
\item either $c^\circ_j$: this is the case when $c^\circ_j$ is a leaf or when the white corner is the one between the last child of $c^\circ_j$ and its parent;
\item or a child of $c^\circ_j$: precisely, its first child if the white corner is the one between the parent of $c^\circ_j$ and its first child, and its $k$-th child if the corner is the one between the $k-1$st and $k$-th children of $c^\circ_j$.
\end{itemize}
Since a black vertex inherits the label of its white parent, we conclude that in both cases we have $L(j) = l(u_j) = \ell(c^\circ_j) = \mathcal{L}^\circ(j)$.

Next, for every $u \in T$, set
\[\widetilde{H}(u) = \#\left\{w \in \llbracket \varnothing, u\llbracket : w k_w \notin \rrbracket \varnothing, u\rrbracket\right\};\]
if $\widetilde{H}(u) \ne 0$, recall the definition
\[u^\star = \sup\left\{w \in \llbracket \varnothing, u\llbracket : w k_w \notin \rrbracket \varnothing, u\rrbracket\right\}.\]
Fix $v \in \circ(\T)$ a white vertex of $\T$ and $w \in \bullet(\T)$ one of its children, if it has any. Denote by $\JS(v), \JS(w) \in T$ their image by the $\JS$ bijection, we argue that $\widetilde{H}(\JS(v))$ and $\widetilde{H}(\JS(w))$ are both equal to half the generation of $v$ in $\T$. Denote by $u = \JS(v)$; from the construction of the $\JS$ bijection, if $v$ is different from the root of $\T$, then its parent in $\T$ is mapped onto $u^\star$ and its children onto $\rrbracket u^\star, u\llbracket$, thus
\[\widetilde{H}(\JS(w)) = \widetilde{H}(\JS(v)) = \widetilde{H}(u) = \widetilde{H}(u^\star)+1 = \widetilde{H}(\JS(pr(v))) + 1.\]
If $v$ is the root of $\T$, then $u$ is the right-most leaf of $T$ and $v$ and its children are mapped onto the vertices of $T$ for which $\widetilde{H} = 0$. We conclude after an induction on the generation of $v$ that indeed, $\widetilde{H}(\JS(w))$ and $\widetilde{H}(\JS(v))$ are equal, and their common value is given by half the generation of $v$ in $\T$.

Recall the notation $c^\circ_j \in \circ(\T)$ for the white vertex of $\T$ visited at the $j$-th step in the white contour sequence and $u_j$ for the $j$-th vertex of $T$ in lexicographical order. Since the image of $u_j$ by the $\JS$ bijection is either $c^\circ_j$ or one of its children (if it has any), we conclude in both cases that $\widetilde{H}(u_j)$ is half the generation of $c^\circ_j$ in $\T$, i.e. $\widetilde{H}(j) = \mathcal{C}^\circ(j)$. 
\end{proof}

Recall the well-known identity between the height process $H$ and the {\L}ukasiewicz path $W$ of a one-type tree (see e.g. Le Gall \& Le Jan\cite{Le_Gall-Le_Jan:Branching_processes_in_Levy_processes_the_exploration_process}):
\begin{equation}\label{eq:hauteur_juka}
H(j) = \#\left\{i \in \{0, \dots, j-1\} : W(i) \le \inf_{[i+1, j]} W\right\}
\quad\text{for each}\quad 0 \le j \le N.
\end{equation}
Indeed, for $i < j$, we have $W(i) \le \inf_{[i+1, j]} W$ if and only if $u_i$ is an ancestor of $u_j$; moreover, the inequality is an equality if and only if the last child of $u_i$ is also an ancestor of $u_j$. A consequence of Lemma \ref{lem:processus_bijection_JS} is therefore the identity
\begin{equation}\label{eq:contour_blanc_luka}
\mathcal{C}^\circ(j) = \#\left\{i \in \{0, \dots, i-1\} : W(i) < \inf_{[i+1, j]} W\right\}
\quad\text{for each}\quad 0 \le j \le N.
\end{equation}
The latter was already observed by Abraham \cite[Equation 5]{Abraham:Rescaled_bipartite_planar_maps_converge_to_the_Brownian_map} without the formalism of the $\JS$ bijection, where $W$ (which corresponds to $Y-1$ there) was defined directly from the two-type tree.

\section{The Brownian map}
\label{sec:serpent_et_enonces}

\subsection{The Brownian snake and the Brownian map}
\label{subsec:serpent_brownien}

Denote by $\exc = (\exc_t ; t \in [0,1])$ the standard Brownian excursion. For every $s, t \in [0,1]$, set
\[m_\exc(s,t) = \min_{r \in [s \wedge t, s \vee t]} \exc_r
\qquad\text{and}\qquad
d_\exc(s,t) = \exc_s + \exc_t - 2 m_\exc(s,t).\]
One easily checks that $d_\exc$ is a random pseudo-metric on $[0,1]$, we then define an equivalence relation on $[0,1]$ by setting $s \sim_\exc t$ whenever $d_\exc(s,t)=0$. Consider the quotient space $\CRT_\exc = [0,1] / \sim_\exc$, we let $\pi_\exc$ be the canonical projection $[0,1] \to \CRT_\exc$; $d_\exc$ induces a metric on $\CRT_\exc$ that we still denote by $d_\exc$. The space $(\CRT_\exc, d_\exc)$ is a so-called compact real-tree, naturally rooted at $\pi_\exc(0) = \pi_\exc(1)$, called the \emph{Brownian tree} coded by $\exc$, introduced by Aldous \cite{Aldous:The_continuum_random_tree_3}.

We construct next another process $Z = (Z_t ; t \in [0,1])$ on the same probability space as $\exc$ which, conditional on $\exc$, is a centred Gaussian process satisfying for every $s,t \in [0,1]$,
\[\Esc{|Z_s - Z_t|^2}{\exc} = d_\exc(s,t)
\qquad\text{or, equivalently,}\qquad
\Esc{Z_s Z_t}{\exc} = m_\exc(s,t).\]
It is known (see, e.g. Le Gall \cite[Chapter IV.4]{Le_Gall:Nachdiplomsvorlesung} on a more general path-valued process called the \emph{Brownian snake} whose $Z$ is only the ``tip'') that the pair $(\exc, Z)$ admits a continuous version and, without further notice, we shall work throughout this paper with this version. Observe that, almost surely, $Z_0=0$ and $Z_s = Z_t$ whenever $s \sim_\exc t$ so $Z$ can be seen as a Brownian motion indexed by $\CRT_\exc$ by setting $Z_{\pi_\exc(t)} = Z_t$ for every $t \in [0,1]$. We interpret $Z_x$ as the label of an element $x \in \CRT_\exc$; the pair $(\CRT_\exc, (Z_x; x \in \CRT_\exc))$ is a continuous analog of labelled plane trees and the construction of the Brownian map from this pair, that we next recall, is somewhat an analog of the $\BDG$ bijection presented above.

Let us follow Le Gall \cite{Le_Gall:The_topological_structure_of_scaling_limits_of_large_planar_maps} to which we refer for details. For every $s, t \in[0, 1]$, define
\[\check{Z}(s,t)=
\begin{cases}
\min\{Z_r ; r \in [s,t]\}  & \text{if }  s \le t,\\
\min\{Z_r ; r \in [s,1] \cup [0,t]\} & \text{otherwise,}
\end{cases}\]
and then
\[D_Z(s,t) = Z_s + Z_t - 2 \max\{\check{Z}(s,t); \check{Z}(t,s)\}.\]
For every $x, y \in \CRT_\exc$, set
\[D_Z(x,y) = \inf\left\{D_Z(s,t) ; s,t \in [0,1], x=\pi_\exc(s) \text{ and }  y=\pi_\exc(t)\right\},\]
and finally
\[\mathscr{D}(x,y) = \inf\left\{\sum_{i=1}^k D_Z(a_{i-1}, a_i) ; k \ge 1, (x=a_0, a_1, \dots, a_{k-1}, a_k=y) \in \CRT_\exc\right\}.\]
The function $\mathscr{D}$ is a pseudo-distance on $\CRT_\exc$, we define an equivalence relation by setting $x \approx y$ whenever $\mathscr{D}(x,y) = 0$ for $x,y \in \CRT_\exc$. The Brownian map is the quotient space $\mathscr{M} = \CRT_\exc / \approx$ equipped with the metric induced by $\mathscr{D}$, that we still denote by $\mathscr{D}$. Note that $\mathscr{D}$ can be seen as a pseudo-distance on $[0,1]$ by setting $\mathscr{D}(s,t) = \mathscr{D}(\pi_\exc(s),\pi_\exc(t))$ for every $s,t \in [0,1]$, thus $\mathscr{M}$ can be seen as a quotient space of $[0,1]$.

The following observation shall be used later on. As a function on $\CRT_\exc^2$, we clearly have $\mathscr{D} \le D_Z$ and in fact, $\mathscr{D}$ is the largest pseudo-distance on $\CRT_\exc$ satisfying this property. Indeed, if $D$ is another such pseudo-distance, then for every $x,y \in \CRT_\exc$, for every $k \ge 1$ and every $a_0, a_1, \dots, a_{k-1}, a_k \in \CRT_\exc$ with $a_0=x$ and $a_k=y$, by the triangle inequality $D(x, y) \le \sum_{i=1}^k D(a_{i-1}, a_i) \le \sum_{i=1}^k D_Z(a_{i-1}, a_i)$ and so $D(x, y) \le \mathscr{D}(x, y)$. Furthermore, if we view $\mathscr{D}$ as a function on $[0,1]^2$, then for all $s,t \in [0,1]$ such that $d_\exc(s,t) = 0$ we have $\pi_\exc(s) = \pi_\exc(t)$ and so $\mathscr{D}(\pi_\exc(s),\pi_\exc(t)) = 0$. We deduce from the previous maximality property that $\mathscr{D}$ is the largest pseudo-distance $D$ on $[0,1]$ satisfying the following two properties:
\[D \le D_Z
\qquad\text{and}\qquad
d_\exc(s,t) = 0 \quad\text{implies}\quad D(s,t) = 0.\]

\subsection{Functional invariance principles}
\label{subsec:enonces_convergences_fonctions}

Let $T_n \in \tree(\n)$ be a one-type tree; it has $N_\n = \sum_{i \ge 1} i n_i$ edges, we denote by $W_n$, $H_n$ and $C_n$ respectively its {\L}ukasiewicz path, its height process and its contour process. The main result of Broutin \& Marckert \cite{Broutin-Marckert:Asymptotics_of_trees_with_a_prescribed_degree_sequence_and_applications} is the following: under \eqref{eq:H}, if $T_n$ is sampled uniformly at random in $\tree(\n)$ for every $n \ge 1$, then the following convergence in distribution holds in $\mathscr{C}([0,1], \R^3)$:
\begin{equation}\label{eq:Broutin_Marckert}
\left(\frac{W_n(N_\n t)}{N_\n^{1/2}}, \frac{H_n(N_\n t)}{N_\n^{1/2}}, \frac{C_n(2N_\n t)}{N_\n^{1/2}}\right)_{t \in [0,1]}
\cvloi
\left(\sigma_p \exc, \frac{2}{\sigma_p} \exc, \frac{2}{\sigma_p} \exc\right)_{t \in [0,1]}.
\end{equation}

Denote by $L_n$ the label process (in lexicographical order) of a labelled tree $(T_n, l_n) \in \bartree(\n)$. Consider also a labelled two-type tree $(\mathcal{T}_n, \ell_n) \in \bartree_{\circ, \bullet}(\n)$; it has $N_\n$ edges as well, we denote by $\mathcal{C}^\circ_n$ its white contour function and by $\mathcal{L}^\circ_n$ its label function (in contour order).

\begin{thm}
\label{thm:cv_fonctions_serpent}
If $(T_n, l_n)$ and $(\T_n, \ell_n)$ are related by the $\JS$ bijection and have the uniform distribution in $\bartree(\n)$ and $\bartree_{\circ, \bullet}(\n)$ respectively for every $n \ge 1$, then, under \eqref{eq:H}, the following convergences in distribution 
hold jointly in $\mathscr{C}([0,1], \R^2)$:
\begin{equation}\label{eq:cv_hauteur_labels}
\left(\left(\frac{\sigma_p^2}{4} \frac{1}{N_\n}\right)^{1/2} H_n(N_\n t), \left(\frac{9}{4 \sigma_p^2} \frac{1}{N_\n}\right)^{1/4} L_n(N_\n t)\right)_{t \in [0,1]}
\cvloi
(\exc_t, Z_t)_{t \in [0,1]},
\end{equation}
and
\begin{equation}\label{eq:cv_contour_labels}
\left(\left(\frac{\sigma_p^2}{4p_0^2} \frac{1}{N_\n}\right)^{1/2} \mathcal{C}^\circ_n(N_\n t), \left(\frac{9}{4 \sigma_p^2} \frac{1}{N_\n}\right)^{1/4} \mathcal{L}^\circ_n(N_\n t)\right)_{t \in [0,1]}
\cvloi
(\exc_t, Z_t)_{t \in [0,1]}.
\end{equation}
\end{thm}

\begin{rem}\label{rem:contour_contour_blanc}
Denote by $\mathcal{C}_n$ the contour function of $\mathcal{T}_n$. We have already observed in Section \ref{subsec:arbres_etiquetes} that $\sup_{t \in [0,1]} |\mathcal{C}_n(2N_\n t) - 2\mathcal{C}^\circ_n(N_\n t)| = 1$, so \eqref{eq:cv_contour_labels} implies
\[\left(\left(\frac{\sigma_p^2}{16p_0^2} \frac{1}{N_\n}\right)^{1/2} \mathcal{C}_n(2N_\n t)\right)_{t \in [0,1]}
\cvloi
(\exc_t)_{t \in [0,1]}.\]
Consequently, we have the joint convergences in the sense of Gromov--Hausdorff:
\[\left(\T_n, N_\n^{-1/2} \dgr\right)
\cvloi
\left(\CRT_\exc, \frac{4 p_0}{\sigma_p} d_\exc\right),
\qquad\text{and}\qquad
\left(T_n, N_\n^{-1/2} \dgr\right)
\cvloi
\left(\CRT_\exc, \frac{2}{\sigma_p} d_\exc\right).\]
\end{rem}

\begin{rem}
By definition, if $(T, l)$ is a labelled one-type tree and $u$ is a vertex of $T$ with $r \ge 1$ children, then the sequence $(0, l(u1)-l(u), \dots, l(ur)-l(u))$ belongs to the set of bridges
\begin{equation}\label{eq:def_pont_sans_saut_negatif}
\mathcal{B}_r^+ = \left\{(x_0, \dots, x_r): x_0=x_r=0 \text{ and } x_j-x_{j-1} \in \{-1, 0, 1, 2, \dots\} \text{ for } 1 \le j \le r\right\}.
\end{equation}
Since the cardinal of $\mathcal{B}_r^+$ is $\binom{2r-1}{r-1}$, it follows that a one-type tree $T$ possesses
\begin{equation}\label{eq:nombre_etiquetages_arbre}
\prod_{u \in T : k_u \ge 1} \binom{2k_u-1}{k_u-1}
\end{equation}
possible labellings. Observe that this quantity is constant over $\tree(\n)$ so if we first sample an unlabelled tree $T_n$ uniformly at random in $\tree(\n)$ and if we then add labels uniformly at random, in the sense that the sequences $(0, l(u1)-l(u), \dots, l(uk_u)-l(u))_{u \in T_n}$ are sampled independently and uniformly at random in $\mathcal{B}^+_{k_u}$ respectively, then the labelled tree has the uniform distribution in $\bartree(\n)$.
\end{rem}

Let us comment on the constants in Theorem \ref{thm:cv_fonctions_serpent}. The one in front of $H_n$ is taken from \eqref{eq:Broutin_Marckert}. Next, the label of a vertex $u \in T_n$ is the sum of the increments of the labels between consecutive ancestors; there are $|u|$ such terms, which are independent and distributed, when an ancestor has $i$ children and the one on the path to $u$ is the $j$-th one, as the $j$-th marginal of a uniform random bridge in $\mathcal{B}_i^+$, as defined in \eqref{eq:def_pont_sans_saut_negatif}; the latter is a centred random variable with variance $2j(i-j)/(i+1)$. As we will see, there is typically a proportion about $p_\n(i)$ of such ancestors so $L_n(u)$ has variance about
\[\sum_{i \ge 1} \sum_{j=1}^i |u| p_\n(i) \frac{2j(i-j)}{i+1}
= |u| \sum_{i \ge 1} p_\n(i) \frac{i(i-1)}{3}
\approx |u| \frac{\sigma_p^2}{3}.\]
If $u$ is the vertex visited at time $\lfloor N_\n t\rfloor$ in lexicographical order, then $|u| \approx (4N_\n/\sigma_p^2)^{1/2} \exc_t$ so we expect $L_n(N_\n t)$, once rescaled by $N_\n^{1/4}$, to be asymptotically Gaussian with variance
\[\left(\frac{4}{\sigma_p^2}\right)^{1/2} \exc_t \frac{\sigma_p^2}{3}
= \left(\frac{4\sigma_p^2}{9}\right)^{1/2} \exc_t.\]
Regarding the two-type tree, the proof of the convergence of $\mathcal{C}^\circ_n$ relies on showing that, as $n \to \infty$, it is close to $p_0 H_n$ when $\T_n$ and $T_n$ are related by the $\JS$ bijection. Finally, according to Lemma \ref{lem:processus_bijection_JS}, when $\T_n$ and $T_n$ are related by the $\JS$ bijection, then the processes $\mathcal{L}^\circ_n$ and $L_n$ are equal.

We next explain how Theorem \ref{thm:cv_fonctions_serpent} will follow from several results proved in Section \ref{sec:convergences_fonctions}.

\begin{proof}[Proof of Theorem \ref{thm:cv_fonctions_serpent}]
Recall from Lemma \ref{lem:processus_bijection_JS} that the processes $L_n$ and $\mathcal{L}^\circ_n$ are equal. Appealing to this lemma, we shall also obtain in Proposition \ref{prop:cv_contour} below the joint convergence
\[\left(\left(\frac{\sigma_p^2}{4} \frac{1}{N_\n}\right)^{1/2} H_n(N_\n t), \left(\frac{\sigma_p^2}{4p_0^2 N_\n}\right)^{1/2} \mathcal{C}^\circ_n(N_\n t)\right)_{t \in [0,1]}
\cvloi
(\exc_t, \exc_t)_{t \in [0,1]}.\]
In Proposition \ref{prop:cv_label_unif_multi}, we shall prove that, jointly with this convergence, for every $k \ge 1$, if $(U_1, \dots, U_k)$ are i.i.d. uniform random variables in $[0,1]$ independent of the trees, then the convergence
\begin{equation}\label{eq:cv_label_unif_multi}
\left(\frac{9}{4 \sigma_p^2} \frac{1}{N_\n}\right)^{1/4} \left(L_n(N_\n U_1), \dots, L_n(N_\n U_k)\right)
\cvloi
\left(Z_{U_1}, \dots, Z_{U_k}\right)
\end{equation}
holds in $\R^k$, where the process $Z$ is independent of $(U_1, \dots, U_k)$. Finally, in Proposition \ref{prop:tension_labels}, we shall prove that the sequence
\[\left(N_\n^{-1/4} L_n(N_\n t) ; t \in [0,1]\right)_{n \ge 1}\]
is tight in $\mathscr{C}([0,1], \R)$. This ensures that the sequences on the left-hand side of \eqref{eq:cv_hauteur_labels} and \eqref{eq:cv_contour_labels} are tight in $\mathscr{C}([0,1], \R^2)$. Using the equicontinuity given by this tightness, as well as the uniform continuity of the pair $(\exc, Z)$, one may transpose \eqref{eq:cv_label_unif_multi} to a convergence for deterministic times, by approximating them by i.i.d. uniform random times, see e.g. Addario-Berry \& Albenque \cite[proof of Proposition 6.1]{Addario_Berry-Albenque:The_scaling_limit_of_random_simple_triangulations_and_random_simple_triangulations} for a detailed argument; this characterises the sub-sequential limits of \eqref{eq:cv_hauteur_labels} and \eqref{eq:cv_contour_labels} in $\mathscr{C}([0,1], \R^2)$ as $(\exc, Z)$.
\end{proof}

The proofs of the above intermediate results are deferred to Section \ref{sec:convergences_fonctions}, they rely on a precise description of the branches from the root of $T_n$ to i.i.d. vertices which is the content of the next section.

\section{Spinal decompositions}
\label{sec:epine}

In this section, we describe the branches from the root to i.i.d. vertices in a tree $T_n$ sampled uniformly at random in $\tree(\n)$, extending results due to Broutin \& Marckert \cite{Broutin-Marckert:Asymptotics_of_trees_with_a_prescribed_degree_sequence_and_applications}. We only state the results, the proofs are technical and are deferred to Appendix \ref{sec:appendice_epine} for the sake of clarity.

\subsection{A one-point decomposition}
\label{sec:epine_1}

For a given vertex $u$ in a plane tree $T$, we denote by $A_i(u)$ its number of strict ancestors with $i$ children:
\[A_i(u) =\# \left\{v \in \llbracket\varnothing, u\llbracket : k_v = i\right\}.\]
We write $\mathbf{A}(u) = (A_i(u) ; i \ge 1)$; note that $|u| = |\mathbf{A}(u)| = \sum_{i \ge 1} A_i(u)$. The quantity $\mathbf{A}(u)$ is crucial in order to control the label $l_n(u)$ of the vertex $u \in T_n$ when $(T_n, l_n)$ is chosen uniformly at random in $\bartree(\n)$. Indeed, one can write
\[l_n(u) = \sum_{v \in \rrbracket \varnothing, u \rrbracket} l_n(v) - l_n(pr(v)),\]
and, conditional on $T_n$, the random variables $l_n(v) - l_n(pr(v))$ are independent and their law depends on the number of children of $pr(v)$.

If $\m = (m_i ; i \ge 1)$ is a sequence of non-negative integers, then we set
\[\LR(\m) = 1 + \sum_{i \ge 1} (i-1) m_i.\]
The notation comes from the fact that removal of the path $\llbracket \varnothing, u \llbracket$ produces a forest of $\LR(\mathbf{A}(u))$ trees, so, in other words, $\LR(\mathbf{A}(u))$ is the number of vertices lying directly on the \emph{left} or on the \emph{right} of this path (and the component ``above''). For every $x > 0$ define the following set of ``good'' sequences:
\[\Good(n,x) = \left\{\m \in \Z_+^\N : \LR(\m) \le x N_\n^{1/2} \enskip\text{and}\enskip |\m| \le x N_\n^{1/2}\right\}.\]
Consider also the more restrictive set
\[\Good^+(n,x) = \left\{\m \in \Z_+^\N : \LR(\m) \le x N_\n^{1/2} \enskip\text{and}\enskip x^{-1} N_\n^{1/2} \le |\m| \le x N_\n^{1/2}\right\}.\]

The following result has been obtained by Broutin \& Marckert \cite{Broutin-Marckert:Asymptotics_of_trees_with_a_prescribed_degree_sequence_and_applications}; it is not written explicitly there but the arguments that we recall in Appendix \ref{sec:appendice_epine} can be found in Sections 3 and 5.2 there.

\begin{lem}
\label{lem:lignee_multinomiale}
For every $n \ge 1$, sample $T_n$ uniformly at random in $\tree(\n)$ and then sample a vertex $u_n$ uniformly at random in $T_n$. For every $\varepsilon > 0$, there exists $x > 0$ such that, under \eqref{eq:H},
\[\liminf_{n \ge 1} \Pr{\mathbf{A}(u) \in \Good(n,x) \text{ for all } u \in T_n} \ge 1 - \varepsilon,\]
and
\[\liminf_{n \ge 1} \Pr{\mathbf{A}(u_n) \in \Good^+(n,x)} \ge 1 - \varepsilon.\]
Furthermore, there exists a constant $C > 0$ (which depends on $x$) such that for every sequence $\m \in \Good(n,x)$, setting $h = |\m|$, we have
\[\Pr{\mathbf{A}(u_n) = \m} \le C \cdot N_\n^{-1/2} \cdot \Pr{\Xi_\n^{(h)} = \m},\]
where $\Xi_\n^{(h)} = (\Xi_{\n,i}^{(h)}; i\ge 1)$ has the multinomial distribution with parameters $h$ and $(i n_i / N_\n; i\ge 1)$.
\end{lem}

Observe that replacing $\mathbf{A}(u_n)$ by such a multinomial sequence means that the random variables $(k_{pr(v)}; v \in \rrbracket\varnothing, u_n\rrbracket)$ are independent and distributed according to the size-biased law $(i n_i / N_\n; i\ge 1)$. Also, clearly, conditional on $(k_{pr(v)}; v \in \rrbracket\varnothing, u_n\rrbracket)$, the random variables $(\chi_v; v \in \rrbracket\varnothing, u_n\rrbracket)$ are independent and each one has the uniform distribution in $\{1, \dots, k_{pr(v)}\}$ respectively.

The following corollary, which shall be used in Section \ref{subsec:tension_labels}, sheds some light on Lemma \ref{lem:lignee_multinomiale}. The argument used in the proof shall be used at several other occasions.

\begin{cor}\label{cor:bon_evenement_tension_labels}
Recall the notation $\chi_w \in \{1, \dots, k_{pr(w)}\}$ for the relative position of a vertex $w \in T_n$ among its siblings. Let $c = 1 - \frac{p_0}{2}$ and $h_\n = \frac{16}{p_0^2} \ln N_\n$ and consider the event
\[\mathcal{E}_n = 
\bigg\{\frac{\#\{w \in \rrbracket u, v\rrbracket : \chi_w = 1\}}{\#\rrbracket u, v\rrbracket}
\le c
\text{ for every } u, v \in T_n \text{ such that } u \in \llbracket \varnothing, v\llbracket \text{ and }\#\rrbracket u, v\rrbracket > h_\n\bigg\}.\]
If $T_n$ is sampled uniformly at random in $\tree(\n)$, then under \eqref{eq:H}, we have $\P(\mathcal{E}_n) \to 1$ as $n \to \infty$.
\end{cor}

In words, this means that in $T_n$, there is no branch longer than some constant times $\ln n$ along which the proportion of individuals which are the left-most child of their parent is too large.

\begin{proof}
For every $v \in T_n$, for every $1 \le j \le |v|$, let us denote by $a_j(v)$ the unique element of $\llbracket \varnothing, v\rrbracket$ such that $\#\llbracket a_j(v), v\rrbracket = j$, then set $X_j(v) = 1$ if $\chi_{a_j(v)} = 1$ and $X_j(v) = 0$ otherwise so
\[\mathcal{E}_n
= \bigcap_{v \in T_n} \bigcap_{h_\n \le j \le |v|} \bigg\{\#\{1\le i \le j : X_i(v) = 1\} \le c \cdot j\bigg\}
= \bigcap_{v \in T_n} \bigcap_{h_\n \le j \le |v|} \bigg\{\sum_{i=1}^j X_i(v) \le c \cdot j\bigg\}.\]
Let $u_0, \dots, u_{N_\n}$ be the vertices of $T_n$ listed in lexicographical order. Sample $q_n$ uniformly at random in $\{1, \dots, N_\n\}$ and independently of $T_n$, let $v_n = u_{q_n}$ and let $\Xi_\n^{(h)}$ denote a random sequence with the multinomial distribution with parameters $h$ and $(i n_i / N_\n; i\ge 1)$. Fix $\varepsilon > 0$, and let $x > 0$ and $C > 0$ as in Lemma \ref{lem:lignee_multinomiale}. Then for $n$ large enough,
\begin{align*}
\Pr{\mathcal{E}_n^c}
&\le \varepsilon + \sum_{1 \le q \le N_\n} \sum_{h_\n \le j \le xN_\n^{1/2}} \sum_{j \le h \le x N_\n^{1/2}} \sum_{\substack{\m \in \Good(n,x) \\ |m|=h}} \Pr{\sum_{i=1}^j X_i(u_q) > c \cdot j \text{ and } \mathbf{A}(u_q) = \m}
\\
&\le \varepsilon + C x^2 N_\n^{3/2} \sup_{j \ge h_\n} \sup_{h \ge j} \sum_{\substack{\m \in \Good(n,x) \\ |m|=h}} \Prc{\sum_{i=1}^j X_i(v_n) > c \cdot j}{\mathbf{A}(v_n) = \m} \Pr{\Xi_\n^{(h)} = \m}.
\end{align*}
Observe that conditional on the offsprings $k_{a_i}(v_n)$'s of the ancestors $a_i(v_n)$'s, the $X_i(v_n)$'s are independent and have the Bernoulli distribution with parameter $1/k_{a_i}(v_n)$ respectively. We thus have
\[\sum_{\substack{\m \in \Good(n,x) \\ |m|=h}} \Prc{\sum_{i=1}^j X_i(v_n) > c \cdot j}{\mathbf{A}(v_n) = \m} \Pr{\Xi_\n^{(h)} = \m}
= \Pr{\sum_{i=1}^j Y_{\n,i} > c \cdot j},\]
where the $Y_{\n,i}$'s are independent and have the Bernoulli distribution with parameter
\[\sum_{r \ge 1} \frac{1}{r} \cdot \frac{r n_r}{N_\n} = 1 - \frac{n_0 - 1}{N_\n}.\]
Recall that $c = 1 - \frac{p_0}{2}$; fix $n$ large enough so that, according to \eqref{eq:H}, $\frac{n_0-1}{N_\n} > \frac{3 p_0}{4}$ and so $c - (1 - \frac{n_0-1}{N_\n}) = \frac{n_0-1}{N_\n} - \frac{p_0}{2} > \frac{p_0}{4}$. The Chernoff bound then reads
\[\Pr{\sum_{i=1}^j Y_{\n,i} > c \cdot j}
\le \Pr{\sum_{i=1}^j (Y_{\n,i} - \Es{Y_{\n,i}}) > \frac{p_0}{4} \cdot j}
\le \exp\left(-\frac{p_0^2}{8} \cdot j\right),\]
so finally, for $n$ large enough,
\[\Pr{\mathcal{E}_n^c} \le \varepsilon + C x^2 N_\n^{3/2} \exp\left(-\frac{p_0^2}{8} \cdot h_\n\right),\]
which converges to $\varepsilon$ as $n \to \infty$ from our choice of $h_\n$.
\end{proof}

\subsection{A multi-point decomposition}
\label{sec:epine_k}

We next extend the previous decomposition according to several i.i.d. uniform random vertices. Let us first introduce some notation. Fix a plane tree $T$ and $k$ distinct vertices $u_1, \dots, u_k$ of $T$ and denote by $T(u_1, \dots, u_k)$ the tree $T$ \emph{reduced} to its root and these vertices:
\[T(u_1, \dots, u_k) = \bigcup_{1 \le j \le k} \llbracket\varnothing, u_j\rrbracket,\]
which naturally inherits a plane tree structure from $T$. Denote by $k' \le k-1$ the number of branch-points of $T(u_1, \dots, u_k)$ and by $v_1, \dots, v_{k'}$ these branch-points. Let $F(u_1, \dots, u_k)$ be the forest obtained from $T(u_1, \dots, u_k)$ by removing the edges linking these branch-points to their children; note that $F(u_1, \dots, u_k)$ contains $k+k'$ connected components which are only single paths, i.e. each one contains one root and only one leaf and the latter is either one of the $u_i$'s or one of the $v_i$'s. Let us rank these connected components in increasing lexicographical order of their root and denote by $\varnothing_j$ and $\lambda_j$ respectively the root and the leaf of the $j$-th one. For every $1 \le j \le k+k'$ and every $i \ge 1$, we set
\[A_i^{(j)}(u_1, \dots, u_k) = \#\left\{z \in \llbracket\varnothing_j, \lambda_j\llbracket : k_z = i\right\},\]
where $k_z$ must be understood as the number of children \emph{in the original tree} $T$ of the vertex $z$. We set
\[\mathbf{A}(u_1, \dots, u_k)
= \left(\mathbf{A}^{(1)}(u_1, \dots, u_k), \dots, \mathbf{A}^{(k+k')}(u_1, \dots, u_k)\right).\]

Fix $n, k \ge 1$, sample $T_n$ uniformly at random in $\tree(\n)$ and then sample i.i.d. uniform random vertices $u_{n,1}, \dots, u_{n,k}$ in $T_n$; denote by $\Bin_k$ the following event:  the reduced tree $T_n(u_{n,1}, \dots, u_{n,k})$ is binary, has $k$ leaves and its root has only one child. Note that on this event, the $u_{n,i}$'s are distinct and the number of branch-points of the reduced tree is $k' = k-1$. Let us also denote by $\Bin_k^+ = \{\max_{a \in T_n} |a| \le N_\n^{3/4}\} \cap \Bin_k$.
The next result is proved in Appendix \ref{sec:appendice_epine}.

\begin{lem}\label{lem:lignees_multinomiales_k}
For every $n \ge 1$, sample $T_n$ uniformly at random in $\tree(\n)$ and then sample i.i.d. uniform random vertices $u_{n,1}, \dots, u_{n,k}$ in $T_n$. For every $\varepsilon > 0$, there exists $x > 0$ such that, under \eqref{eq:H},
\[\liminf_{n \ge 1} \Pr{\Bin_k^+ \cap \bigcap_{i=1}^{2k-1} \left\{\mathbf{A}^{(i)}(u_{n,1}, \dots, u_{n,k}) \in \Good^+(n,x)\right\}} \ge 1 - \varepsilon.\]
Furthermore, there exists $C > 0$ (which depends on $x$) such that for every sequences $\m^{(1)}, \dots, \m^{(2k-1)} \in \Good(n,x)$, setting $|\m^{(j)}|=h_j$ for each $1 \le j \le 2k-1$, we have
\[\Prc{\mathbf{A}(u_{n,1}, \dots, u_{n,k}) = (\m^{(1)}, \dots, \m^{(2k-1)})}{\Bin_k^+} \le C \cdot N_\n^{-(2k-1)/2} \cdot \prod_{j=1}^{2k-1} \Pr{\Xi_\n^{(h_j)} = \m},\]
where $\Xi_\n^{(h_j)} = (\Xi_{\n,i}^{(h_j)}; i\ge 1)$ has the multinomial distribution with parameters $h_j$ and $(i n_i / N_\n; i\ge 1)$.
\end{lem}

\section{Functional invariance principles}
\label{sec:convergences_fonctions}

We state and prove in this section the intermediate results used in the proof of Theorem \ref{thm:cv_fonctions_serpent}. Let $(T_n, l_n)$ be a uniform random labelled tree in $\bartree(\n)$ and let $H_n$ and $L_n$ denote its height and label processes. Let also $\T_n$ be its associated two-type tree, which has the uniform distribution in $\tree_{\circ, \bullet}(\n)$, with white contour process $\mathcal{C}^\circ_n$. Our aim is to show that, under \eqref{eq:H}, the three convergences
\begin{equation}\label{eq:cv_contour_blanc}
\left(\left(\frac{\sigma_p^2}{4p_0^2 N_\n}\right)^{1/2} \mathcal{C}^\circ_n(N_\n t) ; t \in [0,1]\right)
\cvloi
(\exc_t ; t \in [0,1])
\end{equation}
as well as
\begin{equation}\label{eq:cv_hauteur}
\left(\left(\frac{\sigma_p^2}{4} \frac{1}{N_\n}\right)^{1/2} H_n(N_\n t) ; t \in [0,1]\right)
\cvloi
(\exc_t ; t \in [0,1])
\end{equation}
and
\begin{equation}\label{eq:cv_labels}
\left(\left(\frac{9}{4 \sigma_p^2} \frac{1}{N_\n}\right)^{1/4} L_n(N_\n t) ; t \in [0,1]\right)
\cvloi
(Z_t ; t \in [0,1]),
\end{equation}
hold jointly in $\mathscr{C}([0,1], \R)$. The second one is the main result of \cite{Broutin-Marckert:Asymptotics_of_trees_with_a_prescribed_degree_sequence_and_applications} recalled in \eqref{eq:Broutin_Marckert}. We prove \eqref{eq:cv_contour_blanc} in the next subsection. Then we prove the convergence of random finite-dimensional marginals of $(N_\n^{-1/4} L_n(N_\n \cdot))_{n \ge 1}$ in Section \ref{subsec:marginales_labels} and the tightness of this sequence in Section \ref{subsec:tension_labels}.

\subsection{Convergence of the contour}
\label{sec:convergence_contour}

Let $T_n$ have the uniform distribution in $\tree(\n)$ and let $\T_n$ be its associated two-type tree, which has the uniform distribution in $\tree_{\circ, \bullet}(\n)$.

\begin{prop}\label{prop:cv_contour}
Under \eqref{eq:H}, we have the convergence in distribution in $\mathscr{C}([0,1], \R^2)$
\[\left(\left(\frac{\sigma_p^2}{4} \frac{1}{N_\n}\right)^{1/2} H_n(N_\n t), \left(\frac{\sigma_p^2}{4p_0^2 N_\n}\right)^{1/2} \mathcal{C}^\circ_n(N_\n t)\right)_{t \in [0,1]}
\cvloi
(\exc_t, \exc_t)_{t \in [0,1]}.\]
\end{prop}

The key observation is the identity from Lemma \ref{lem:processus_bijection_JS}:
\[\mathcal{C}^\circ_n = \widetilde{H}_n,\]
where $\widetilde{H}_n(j)$ is the number of strict ancestors of the $j$-th vertex of $T_n$ whose last child is not one of its ancestors. We have seen in the previous section that for a ``typical'' vertex $u$ of $T_n$, at generation $|u|$, the number of ancestors having $i$ children for $i \ge 1$ forms approximately a multinomial sequence with parameters $|u|$ and $(i n_i / N_\n; i\ge 1)$; further, for each such ancestor, there is a probability $1-1/i$ that its last child is not an ancestor of $u$ and therefore contributes to $\mathcal{C}^\circ_n$. Since $\sum_{i \ge 1} (1-1/i) (i n_i / N_\n) \to 1 - (1-p_0) = p_0$, we conclude that, at a ``typical'' time, $\mathcal{C}^\circ_n \approx p_0 H_n$.

\begin{proof}
The convergence of the first marginal comes from \eqref{eq:Broutin_Marckert}; since, under \eqref{eq:H}, we have $p_0 = \lim_{n \to \infty} (n_0-1)/N_\n$ it suffices then to prove that
\[N_\n^{-1/2} \sup_{0 \le t \le 1} \left|\widetilde{H}_n(N_\n t) - \frac{n_0-1}{N_\n} H_n(N_\n t)\right| \cvproba 0.\]
Note that we may restrict ourselves to times $t$ of the form $i/N_\n$ with $i \in \{1, \dots, N_\n\}$. We proceed as in the proof of Corollary \ref{cor:bon_evenement_tension_labels}. Let $i_n$ be a uniform random integer in $\{1, \dots, N_\n\}$ and $u_n$ the $i_n$-th vertex of $T_n$ in lexicographical order. Fix $\delta, \varepsilon > 0$ and choose $x > 0$ and $C > 0$ as in Lemma \ref{lem:lignee_multinomiale}. Then for $n$ large enough,
\begin{align*}
&\Pr{\sup_{1 \le i \le N_\n} \left|\widetilde{H}_n(i) - \frac{n_0-1}{N_\n} H_n(i)\right| > \delta N_\n^{1/2}}
\\
&\qquad\le \varepsilon + x N_\n^{3/2} \sup_{1 \le h \le x N_\n^{1/2}} \sum_{\substack{\m \in \Good(n,x) \\ |m|=h}} \Pr{\mathbf{A}(u_n) = \m} \Prc{\left|\widetilde{H}_n(i_n) - \frac{n_0-1}{N_\n} h\right| > \delta N_\n^{1/2}}{\mathbf{A}(u_n) = \m}.
\\
&\qquad\le \varepsilon + C x N_\n \sup_{1 \le h \le x N_\n^{1/2}} \sum_{\substack{\m \in \Good(n,x) \\ |m|=h}} \Pr{\Xi_\n^{(h)} = \m} \Prc{\left|\widetilde{H}_n(i_n) - \frac{n_0-1}{N_\n} h\right| > \delta N_\n^{1/2}}{\mathbf{A}(u_n) = \m}.
\end{align*}
Observe that conditional on the vector $(k_v ; v \in \llbracket\varnothing, u_n\llbracket)$, the random variable $\widetilde{H}_n(i_n)$ is a sum of independent Bernoulli random variables, with respective parameter $(1 - k_v^{-1} ; v \in \llbracket\varnothing, u_n\llbracket)$. Note that
\[\sum_{i \ge 1} \left(1 - \frac{1}{i}\right) \cdot \frac{i n_i}{N_\n} = \frac{n_0-1}{N_\n},\]
we let $(Y_{\n,i} ; 1 \le i \le h)$ be independent Bernoulli random variables with parameter $(n_0-1)/N_\n$. We then conclude, applying the Chernoff bound for the second inequality, that for every $n$ large enough,
\begin{align*}
\Pr{\sup_{1 \le i \le N_\n} \left|\widetilde{H}_n(i) - \frac{n_0-1}{N_\n} H_n(i)\right| > \delta N_\n^{1/2}}
&\le \varepsilon + C x N_\n \sup_{1 \le h \le x N_\n^{1/2}} \Pr{\left|\sum_{i=1}^h Y_{\n,i} - \frac{n_0-1}{N_\n} h\right| > \delta N_\n^{1/2}}
\\
&\le \varepsilon + C x N_\n \sup_{1 \le h \le x N_\n^{1/2}} 2 \e^{-2 \delta^2 N_\n / h},
\end{align*}
which converges to $\varepsilon$ as $n \to \infty$.
\end{proof}

\subsection{Maximal displacement at a branch-point}
\label{subsec:borne_pont}

Recall that for every vertex $u$, we denote by $k_u$ its number of children and these children by $u1, \dots, uk_u$.

\begin{prop}\label{prop:deplacement_max_autour_site}
For every $n \ge 1$, sample $(T_n, l_n)$ uniformly at random in $\bartree(\n)$. Under \eqref{eq:H}, we have the convergence in probability
\[N_\n^{-1/4} \max_{u \in T_n} \left|\max_{1 \le j \le k_u} l_n(uj) - \min_{1 \le j \le k_u} l_n(uj)\right| \cvproba 0.\]
\end{prop}

To prove this result, we shall need the following sub-Gaussian tail bound for the maximal gap in a random walk bridge. The proof is easy, we refer to Appendix \ref{sec:appendice_pont_echangeable}.

\begin{lem}\label{lem:borne_ecart_max_pont}
Let $(S_k;k \ge 0)$ be a random walk such that $S_0 = 0$ and $(S_{k+1}-S_k; k \ge 0)$ are i.i.d. random variables, taking values in $\Z \cap [-b, \infty)$ for some $b \ge 0$, centred and with variance $\sigma^2 \in (0, \infty)$. There exists two constants $c, C > 0$ which only depend on $b$ and $\sigma$ such that for every $r \ge 1$ and $x \ge 0$, we have
\[\Prc{\max_{0 \le k \le r} S_k - \min_{0 \le k \le r} S_k \ge x}{S_r=0} \le C \e^{-cx^2/r}.\]
\end{lem}

\begin{proof}[Proof of Proposition \ref{prop:deplacement_max_autour_site}]
Recall that conditional on $T_n$, the sequences $(0, l_n(u1)-l_n(u), \dots, l_n(uk_u)-l_n(u))_{u \in T_n}$ are independent and distributed respectively uniformly at random in $\mathcal{B}_r^+$ defined in \eqref{eq:def_pont_sans_saut_negatif}, with $r = k_u$, and that there are $n_r$ such vertices in $T_n$. Consider the random walk $(S_i;i \ge 0)$ such that $S_0 = 0$ and $(S_{i+1}-S_i; i \ge 0)$ are i.i.d. random variables, distributed as a shifted geometric law: $\Pr{S_1 = k} = 2^{-(k+2)}$ for every $k \ge -1$. Then it is easy to check that for every $r \ge 1$, on the event $\{S_r = 0\}$, the path $(S_0, \dots, S_r)$ has the uniform distribution in $\mathcal{B}_r^+$. Therefore, according to Lemma \ref{lem:borne_ecart_max_pont}, there exists two universal constants $c, C > 0$ such that for every $\varepsilon > 0$, for every $n$ large enough,
\begin{align*}
\Pr{\max_{u \in T_n} \left|\max_{1 \le i \le k_u} l_n(ui) - \min_{1 \le i \le k_u} l_n(ui)\right| \le \varepsilon N_\n^{1/4}}
&= \prod_{r=1}^{\Delta_\n} \Prc{\max_{0 \le k \le r} S_k - \min_{0 \le k \le r} S_k \le \varepsilon N_\n^{1/4}}{S_r=0}^{n_r}
\\
&\ge \prod_{r=1}^{\Delta_\n} \left(1 - C \exp\left(-c \varepsilon^2 N_\n^{1/2}/r\right)\right)^{n_r}
\\
&\ge \exp\left(- \sum_{r=1}^{\Delta_\n} n_r \frac{C \exp\left(-c \varepsilon^2 N_\n^{1/2}/r\right)}{1-C \exp\left(-c \varepsilon^2 N_\n^{1/2}/r\right)}\right)
\\
&\ge \exp\left(- C \sum_{r=1}^{\Delta_\n} n_r \exp\left(-c \varepsilon^2 N_\n^{1/2}/r\right) (1+o(1))\right),
\end{align*}
where we have used the bound $\ln(1-x) \ge -\frac{x}{1-x}$ for $x < 1$, jointly with the fact that, under \eqref{eq:H}, we have $\sup_{1 \le r \le \Delta_\n} \exp(-c \varepsilon^2 N_\n^{1/2}/r) \to 0$ since $\Delta_\n = o(N_\n^{1/2})$. Recall furthermore that under \eqref{eq:H}, we have $\sum_{r=1}^{\Delta_\n} r^2 n_r/N_\n \to \sigma_p^2+1 < \infty$, we conclude that for every $n$ large enough, since $x \mapsto x^2\e^{-x}$ is decreasing on $[2, \infty)$,
\[\sum_{r=1}^{\Delta_\n} n_r \exp\left( - c \varepsilon^2 \frac{N_\n^{1/2}}{r}\right)
\le \sum_{r=1}^{\Delta_\n} \frac{r^2 n_r}{N_\n} \times \frac{N_\n}{\Delta_\n^2} \exp\left( - c \varepsilon^2 \frac{N_\n^{1/2}}{\Delta_\n}\right)
\cv 0,\]
and the claim follows.
\end{proof}

\subsection{Random finite-dimensional convergence}
\label{subsec:marginales_labels}

As in Section \ref{sec:epine}, in order to make the notation easier to follow, we first treat the one-dimensional case.

\begin{prop}\label{prop:cv_label_unif}
For every $n \ge 1$, sample independently $(T_n, l_n)$ uniformly at random in $\bartree(\n)$ and $U$ uniformly at random in $[0,1]$. Under \eqref{eq:H}, the convergence in distribution
\[\left(\frac{9}{4 \sigma_p^2} \frac{1}{N_\n}\right)^{1/4} L_n(N_\n U)
\cvloi
Z_U\]
holds jointly with \eqref{eq:cv_hauteur}, where the process $Z$ is independent of $U$.
\end{prop}

\begin{proof}
The approach of the proof was described in Section \ref{subsec:enonces_convergences_fonctions} when explaining the constant $(9/(4 \sigma_p^2))^{1/4}$. 
Note that the vertex $u_n$ visited at the time $\lceil N_\n U \rceil$ in lexicographical order has the uniform distribution in $T_n$;\footnote{Precisely $u_n$ has the uniform distribution in $T_n \setminus \{\varnothing\}$, but we omit this detail for the sake of clarity.} denote by $l_n(u_n) = L_n(\lceil N_\n U \rceil)$ its label and by $|u_n| = H_n(\lceil N_\n U \rceil)$ its height and observe that
\[\left(\frac{9}{4 \sigma_p^2} \frac{1}{N_\n}\right)^{1/4} l_n(u_n)
= \sqrt{\sqrt{\frac{\sigma_p^2}{4} \frac{1}{N_\n}} |u_n|} \cdot \sqrt{\frac{3}{\sigma_p^2}} \frac{1}{\sqrt{|u_n|}} l_n(u_n).\]
Since, according to \eqref{eq:cv_hauteur}, the first term on the right converges in distribution towards $\exc_U$, 
it is equivalent to show that, jointly with \eqref{eq:cv_hauteur}, we have
\begin{equation}\label{eq:cv_label_unif}
\frac{1}{\sqrt{|u_n|}} l_n(u_n) \quad\mathop{\Longrightarrow}_{n \to \infty}\quad \mathcal{N}\left(0, \frac{\sigma_p^2}{3}\right),
\end{equation}
where $\mathcal{N}(0, \sigma_p^2/3)$ denotes the centred Gaussian distribution with variance $\sigma_p^2/3$ and ``$\Rightarrow$'' is a slight abuse of notation to refer to the weak convergence of the law of the random variable.

Recall that we denote by $A_i(u_n)$ the number of strict ancestors of $u_n$ with $i$ children:
\[A_i(u_n) =\# \left\{v \in \llbracket\varnothing, u_n\llbracket : k_v = i\right\};\]
denote further by $A_{i,j}(u_n)$ the number of strict ancestors of $u_n$ with $i$ children, among which the $j$-th one is again an ancestor of $u_n$:
\[A_{i,j}(u_n) =\# \left\{v \in \llbracket\varnothing, u_n\llbracket : k_v = i \text{ and } vj \in \rrbracket\varnothing, u_n\rrbracket\right\}.\]
We have seen in Section \ref{sec:epine} that when $T_n$ is uniformly distributed in $\tree(\n)$ and $u_n$ is uniformly distributed in $T_n$, then $\mathbf{A}(u_n) = (A_i(u_n) ; i \ge 1)$ can be compared to a multinomial sequence with parameters $|u_n|$ and $(i n_i / N_\n; i\ge 1)$. Observe further that given the sequence $\mathbf{A}(u_n)$, the vectors $(A_{i,j}(u_n) ; 1 \le j\le i)_{i \ge 1}$ are independent and distributed respectively according to the multinomial distribution with parameters $A_i(u_n)$ and $(\frac{1}{i}, \dots, \frac{1}{i})$.

Let $(X_{i,j,k} ; 1 \le j \le i \le \Delta_\n, k \ge 1)$ be a collection of independent random variables which is also independent of $\mathbf{A}(u_n)$, and such that  $X_{i,j,k}$ has the law of the $j$-th marginal of a uniform random bridge in $\mathcal{B}_i^+$; note that the latter is centred and has variance, say, $\sigma_{i,j}^2$. Then let us write
\[l_n(u_n) = \sum_{i=1}^{\Delta_\n} \sum_{j=1}^i \sum_{k=1}^{A_{i,j}(u_n)} X_{i,j,k},
\qquad\text{and}\qquad
l_n^{K}(u_n) = \sum_{i=1}^K \sum_{j=1}^i \sum_{k=1}^{A_{i,j}(u_n)} X_{i,j,k},
\quad\text{for } K \ge 1.\]
The proof of \eqref{eq:cv_label_unif} is divided into two steps: we first show that for every $K \ge 1$, $l_n^{K}(u_n)/\sqrt{|u_n|}$ converges towards a limit which depends on $K$ and which in turn converges towards $\mathcal{N}(0, \sigma_p^2/3)$ as $K \to \infty$, and then we show that $|l_n(u_n) - l_n^{K}(u_n)|/\sqrt{|u_n|}$ can be made arbitrarily small uniformly for $n$ large enough by choosing $K$ large enough.

Let us first prove the convergence of $l_n^{K}(u_n)$ as $n \to \infty$. For every $h \ge 1$, let $\Xi^{(h)}_\n = (\Xi^{(h)}_{\n,i}; i\ge 1)$ denote a random sequence with the multinomial distribution with parameters $h$ and $(i n_i / N_\n; i\ge 1)$ and fix $\varepsilon > 0$, and let $x > 0$ and $C > 0$ as in Lemma \ref{lem:lignee_multinomiale}.

Fix $i \ge 1$ such that $p(i) \ne 0$. Since $\Xi^{(h)}_{\n,i}$ has the binomial distribution with parameters $h$ and $i n_i / N_\n$, Lemma \ref{lem:lignee_multinomiale} and Markov inequality yield for every $\delta > 0$ and every $n$ large enough,
\begin{align*}
\Pr{\left|\frac{N_\n}{|u_n| i n_i} A_i(u_n) - 1\right| > \delta}
&\le \varepsilon + C x \sup_{x^{-1} N_\n^{1/2} \le h \le x N_\n^{1/2}} \Pr{\left|\frac{N_\n}{h i n_i} \Xi^{(h)}_{\n,i} - 1\right| > \delta}
\\
&\le \varepsilon + C x \sup_{x^{-1} N_\n^{1/2} \le h \le x N_\n^{1/2}} h^{-1} \delta^{-2} \left(\frac{N_\n}{i n_i} - 1\right),
\end{align*}
which converges to $\varepsilon$ as $n \to \infty$ since $i n_i / N_\n \to ip(i) \in(0, 1)$. Given $A_i(u_n)$, the vector $(A_{i,j}(u_n) ; 1 \le j\le i)$ has the multinomial distribution with parameters $A_i(u_n)$ and $(\frac{1}{i}, \dots, \frac{1}{i})$ so for every $1 \le j \le i$, we further have
\[\frac{N_\n}{|u_n| n_i} A_{i,j}(u_n) \cvproba 1.\]
Since the random variables $X_{i,j,k}$ are independent, centred and have variance $\sigma_{i,j}^2$, the central limit theorem then reads, when $p(i) \ne 0$,
\begin{equation}\label{eq:cv_etiquette_pi_non_nul}
\frac{1}{\sqrt{|u_n|}} \sum_{k=1}^{A_{i,j}(u_n)} X_{i,j,k}
\quad\mathop{\Longrightarrow}_{n \to \infty}\quad
\mathcal{N}\left(0, p(i) \sigma_{i,j}^2\right).
\end{equation}

In the case $p(i)=0$, we claim that
\begin{equation}\label{eq:cv_etiquette_pi_nul}
\frac{1}{\sqrt{|u_n|}} \sum_{j=1}^i \sum_{k=1}^{A_{i,j}(u_n)} X_{i,j,k}
\cvproba 0.
\end{equation}
Indeed, with the same argument as above, it suffices to show that for every $\delta > 0$, we have
\[\lim_{n \to \infty} 
\sup_{x^{-1} N_\n^{1/2} \le h \le x N_\n^{1/2}} \sum_{|\m|=h}
\Pr{\Xi^{(h)}_\n = \m}
\Pr{\left|\sum_{j=1}^i \sum_{k=1}^{M_{i,j}} X_{i,j,k}\right| \ge \delta \sqrt{h}}
=0,\]
where the vector $(M_{i,j} ; 1 \le j\le i)$ has the multinomial distribution with parameters $m_i$ and $(\frac{1}{i}, \dots, \frac{1}{i})$ and is independent of the $X_{i,j,k}$'s. For every sequence $\m$, we have
\begin{align*}
\Pr{\left|\sum_{j=1}^i \sum_{k=1}^{M_{i,j}} X_{i,j,k}\right| \ge \delta \sqrt{h}}
\le \frac{1}{\delta^2 h} \sum_{j=1}^i \Es{M_{i,j}} \sigma_{i,j}^2
= \frac{1}{\delta^2 h} \frac{m_i}{i} \sum_{j=1}^i \sigma_{i,j}^2,
\end{align*}
whence
\begin{align*}
\sum_{|\m|=h}
\Pr{\Xi^{(h)}_\n = \m}
\Pr{\left|\sum_{j=1}^i \sum_{k=1}^{M_{i,j}} X_{i,j,k}\right| \ge \delta \sqrt{h}}
&\le \sum_{|\m|=h}
\Pr{\Xi^{(h)}_\n = \m}
\frac{1}{\delta^2 h} \frac{m_i}{i} \sum_{j=1}^i \sigma_{i,j}^2
\\
&\le \Es{\Xi^{(h)}_{\n,i}} \frac{1}{\delta^2 h} \frac{1}{i} \sum_{j=1}^i \sigma_{i,j}^2
\\
&\le \frac{n_i}{N_\n} \frac{1}{\delta^2} \sum_{j=1}^i \sigma_{i,j}^2.
\end{align*}
Under \eqref{eq:H}, we have $n_i/N_\n \to p(i) = 0$ as $n \to \infty$ and \eqref{eq:cv_etiquette_pi_nul} follows.

We conclude using \eqref{eq:cv_etiquette_pi_non_nul}, \eqref{eq:cv_etiquette_pi_nul} and the independence of the $X_{i,j,k}$'s as $i$ and $j$ vary that for every $K \ge 1$, the convergence 
\[\frac{1}{\sqrt{|u_n|}} l_n^{K}(u_n) \quad\mathop{\Longrightarrow}_{n \to \infty}\quad \mathcal{N}\left(0, \sum_{i=1}^K p(i) \sum_{j=1}^i \sigma_{i,j}^2\right)\]
holds. Marckert \& Miermont \cite[page 1664]{Marckert-Miermont:Invariance_principles_for_random_bipartite_planar_maps}\footnote{Note that they consider uniform random bridges in $\mathcal{B}_{i+1}^+$!} have calculated the variance of the random variables $X_{i,j,k}$:
\[\sigma_{i,j}^2 = \frac{2j(i-j)}{i+1}
\qquad\text{so}\qquad
\sum_{j=1}^i \sigma_{i,j}^2 = \frac{i(i-1)}{3}.\]
Consequently,
\[\sum_{i=1}^K p(i) \sum_{j=1}^i \sigma_{i,j}^2
\quad\mathop{\longrightarrow}^{}_{K \to \infty}\quad
\sum_{i=1}^\infty p(i) \frac{i(i-1)}{3}
= \frac{\sigma_p^2}{3},\]
which implies
\[\mathcal{N}\left(0, \sum_{i=1}^K p(i) \sum_{j=1}^i \sigma_{i,j}^2\right)
\quad\mathop{\Longrightarrow}_{K \to \infty}\quad
\mathcal{N}\left(0, \frac{\sigma_p^2}{3}\right).\]

It only remains to show that for every $\delta > 0$, we have
\begin{equation}\label{eq:borne_labels_reste_K}
\lim_{K \to \infty} \limsup_{n \to \infty} \Pr{\left|l_n(u_n) - l_n^{K}(u_n)\right| \ge \delta \sqrt{|u_n|}} = 0.
\end{equation}
Again, with the same notation as above, it is enough to show that for every $x>0$ and every $\delta > 0$, we have
\[\lim_{K \to \infty} \limsup_{n \to \infty} 
\sup_{x^{-1} N_\n^{1/2} \le h \le x N_\n^{1/2}} \sum_{|\m|=h}
\Pr{\Xi^{(h)}_\n = \m}
\Pr{\left|\sum_{i = K}^{\Delta_\n} \sum_{j=1}^i \sum_{k=1}^{M_{i,j}} X_{i,j,k}\right| \ge \delta \sqrt{h}}
=0.\]
By the same calculation as above,
\begin{align*}
\sum_{|\m|=h}
\Pr{\Xi^{(h)}_\n = \m}
\Pr{\left|\sum_{i = K}^{\Delta_\n} \sum_{j=1}^i \sum_{k=1}^{M_{i,j}} X_{i,j,k}\right| \ge \delta \sqrt{h}}
&\le \sum_{|\m|=h}
\Pr{\Xi^{(h)}_\n = \m}
\frac{1}{\delta^2 h} \sum_{i = K}^{\Delta_\n} \frac{m_i}{i} \sum_{j=1}^i \sigma_{i,j}^2
\\
&=\frac{1}{\delta^2 h} \sum_{i = K}^{\Delta_\n} \frac{1}{i} \Es{\Xi^{(h)}_{\n,i}} \sum_{j=1}^i \sigma_{i,j}^2
\\
&=\frac{1}{\delta^2} \sum_{i = K}^{\Delta_\n} \frac{n_i}{N_\n} \frac{i(i-1)}{3},
\end{align*}
Under \eqref{eq:H}, we have
\[\sum_{i = K}^{\Delta_\n} \frac{n_i}{N_\n} i(i-1)
\cv \sum_{i \ge K} p(i) i(i-1)
\cv[K] 0.\]
This concludes the proof of \eqref{eq:borne_labels_reste_K}.
\end{proof}

We next give a multi-dimensional extension of Proposition \ref{prop:cv_label_unif}. The proof of the latter relied on Lemma \ref{lem:lignee_multinomiale}, the proof of its extension appeals to Lemma \ref{lem:lignees_multinomiales_k}.

\begin{prop}\label{prop:cv_label_unif_multi}
For every $n \ge 1$, sample independently $(T_n, l_n)$ uniformly at random in $\bartree(\n)$ and $U_1, \dots, U_k$ uniformly at random in $[0,1]$. Under \eqref{eq:H}, the convergence in distribution
\[\left(\frac{9}{4 \sigma_p^2} \frac{1}{N_\n}\right)^{1/4} \left(L_n(N_\n U_1), \dots, L_n(N_\n U_k)\right)
\cvloi
\left(Z_{U_1}, \dots, Z_{U_k}\right)\]
holds jointly with \eqref{eq:cv_hauteur}, where the process $Z$ is independent of $(U_1, \dots, U_k)$.
\end{prop}

\begin{proof}
As for Lemma \ref{lem:lignees_multinomiales_k}, we focus on the case $k=2$ and comment on the general case at the end. Let $u_n$ and $v_n$ be independent uniform random vertices of $T_n$ and $w_n$ be their most recent common ancestor, let further $\hat{u}_n$ and $\hat{v}_n$ be the children of $w_n$ which are respectively an ancestor of $u_n$ and $v_n$. We write:
\[l_n(u_n) = l_n(w_n) + (l_n(\hat{u}_n)-l_n(w_n)) + (l_n(u_n) - l_n(\hat{u}_n)),\]
and we have a similar decomposition for $v_n$. The point is that, conditional on $T_n$, $u_n$ and $v_n$, the random variables $l_n(w_n)$, $l_n(u_n) - l_n(\hat{u}_n)$ and $l_n(v_n) - l_n(\hat{v}_n)$ are independent. Moreover, according to Proposition \ref{prop:deplacement_max_autour_site}, with high probability, $l_n(\hat{u}_n)-l_n(w_n)$ and $l_n(\hat{v}_n)-l_n(w_n)$ are both small compared to $N_\n^{1/4}$.

According to \eqref{eq:cv_hauteur}, we have
\[\left(\frac{\sigma_p^2}{4} \frac{1}{N_\n}\right)^{1/2} \left(|w_n|, |u_n| - |\hat{u}_n|, |v_n| - |\hat{v}_n|\right)
\cvloi
\left(m_\exc(U,V), \exc_U - m_\exc(U,V), \exc_V - m_\exc(U,V)\right),\]
where $U$ and $V$ are i.i.d uniform random variables on $[0,1]$ independent of $\exc$. We shall prove that, jointly with \eqref{eq:cv_hauteur},
\begin{equation}\label{eq:cv_labels_trois_parties}
\sqrt{\frac{3}{\sigma_p^2}} \left(\frac{l_n(w_n)}{\sqrt{|w_n|}}, \frac{l_n(u_n) - l_n(\hat{u}_n)}{\sqrt{|u_n| - |\hat{u}_n|}}, \frac{l_n(v_n) - l_n(\hat{v}_n)}{\sqrt{|v_n| - |\hat{v}_n|}}\right)
\cvloi \left(G_1, G_2, G_3\right),
\end{equation}
where $G_1$, $G_2$, $G_3$ are i.i.d. standard Gaussian random variables. Proposition \ref{prop:deplacement_max_autour_site} and \eqref{eq:cv_labels_trois_parties} then imply that, jointly with \eqref{eq:cv_hauteur}, the pair
\[\left(\left(\frac{9}{4 \sigma_p^2} \frac{1}{N_\n}\right)^{1/4} (l_n(u_n), l_n(v_n))\right)_{n \ge 1}\]
converges in distribution towards
\[\left(\sqrt{m_\exc(U,V)} G_1 + \sqrt{\exc_U - m_\exc(U,V)} G_2, \sqrt{m_\exc(U,V)} G_1 + \sqrt{\exc_V - m_\exc(U,V)} G_3\right)
= (Z_{U_1}, Z_{U_2}).\]
The proof of \eqref{eq:cv_labels_trois_parties} is \emph{mutatis mutandis} the same as that of Proposition \ref{prop:cv_label_unif}: consider the three branches $\llbracket\varnothing, w_n\rrbracket$, $\llbracket \hat{u}_n, u_n\rrbracket$ and $\llbracket \hat{v}_n, v_n\rrbracket$, we use Lemma \ref{lem:lignees_multinomiales_k} to compare the number of elements in each branch which have $i$ children and among which the $j$-th one belongs to the branch to independent multinomial distributions; then we may use the arguments of the proof of Proposition \ref{prop:cv_label_unif} to each branch independently which yields \eqref{eq:cv_labels_trois_parties}.

The general case $k \ge 2$ hides no difficulty. Sample i.i.d. uniform random vertices $u_{n,1}, \dots, u_{n,k}$ of $T_n$; appealing to Proposition \ref{prop:deplacement_max_autour_site}, we neglect the contribution of the branch-points of the reduced tree $T_n(u_{n,1}, \dots, u_{n,k})$ and we decompose the labels of each vertex $u_{n,i}$ as the sum of the increments over all the branches of the forest $F_n(u_{n,1}, \dots, u_{n,k})$; Lemma \ref{lem:lignees_multinomiales_k} then yields the generalisation of \eqref{eq:cv_labels_trois_parties}.
\end{proof}

\subsection{Concentration results for discrete excursions}
\label{subsec:concentration}

In this subsection, we shall prove two concentration inequalities for the {\L}ukasiewicz path of $T_n$. The first one shall be used to derive the tightness of the label process in the next subsection, and the second one in Section \ref{sec:carte_brownienne} in the proof of Theorem \ref{thm:cv_carte}.

\begin{prop}\label{prop:moments_marche_Luka}
Assume that \eqref{eq:H} holds and let $W_n$ be the {\L}ukasiewicz path of a tree sampled uniformly at random in $\tree(\n)$. There exists a constant $C > 0$ such that, uniformly for $t \ge 0$, $n \in \N$ and $0 \le j < k \le N_\n+1$ with $k-j \le N_\n/2$,
\[\Pr{W_n(j) - \min_{j \le i \le k} W_n(i) > t} \le \exp\left(-\frac{t^2}{C \cdot (k-j)}\right).\]
Consequently, for every $r > 0$, if $C(r) = \Gamma(1+\frac{r}{2}) \cdot C^{r/2}$, then the bound
\[\Es{\left(W_n(j) - \min_{j \le i \le k} W_n(i)\right)^r} \le C(r) \cdot (k-j)^{r/2},\]
holds uniformly for $n \in \N$ and $0 \le j < k \le N_\n+1$ such that $k-j \le N_\n/2$.
\end{prop}

This result follows from Section 3 of Addario-Berry \cite{Addario_Berry:Tail_bounds_for_the_height_and_width_of_a_random_tree_with_a_given_degree_sequence}. Fix $\m = (m_0, m_1, m_2, \dots)$ a sequence of non-negative integers with finite sum satisfying
\[M = \sum_{i \ge 0} m_i,
\qquad
\sum_{i \ge 0} (i-1) m_i = -1
\qquad\text{and}\qquad
\varsigma^2 = \sum_{i \ge 0} (i-1)^2 m_i,\]
and define
\[\mathbf{B}(\m) = \left\{x = (x_1, \dots, x_M): \#\{j : x_j = i-1\} = m_i \text{ for every } i \ge 0\right\}.\]
Given $x \in \mathbf{B}(\m)$, we consider the walk $S_x$ defined by $S_x(0) = 0$ and $S_x(k) = x_1 + \dots + x_k$ for $1 \le k \le M$. A careful reading of \cite[Section 3]{Addario_Berry:Tail_bounds_for_the_height_and_width_of_a_random_tree_with_a_given_degree_sequence} which focuses on the case $k=\lfloor M/2\rfloor$, and which relies on a concentration inequality similar to Lemma \ref{lem:borne_ecart_max_pont} applied to the martingale $(S_x(k)+1)/(M-k)$, yields the following result.

\begin{lem}[Addario-Berry \cite{Addario_Berry:Tail_bounds_for_the_height_and_width_of_a_random_tree_with_a_given_degree_sequence}]
\label{lem:Addario_Berry_borne_sous_gaussienne}
If $x$ is sampled uniformly at random in $\mathbf{B}(\m)$, then
\[\Pr{- \min_{0 \le i \le k} S_x(i) \ge t}
\le \exp\left(-\frac{t^2}{(16 \frac{\varsigma^2}{M} + \frac{8}{3} (1 - \frac{1}{M})) k}\right)\]
for every $1 \le k \le \lfloor M/2\rfloor$ and every $t \ge 0$.
\end{lem}

Observe that $S_x(M) = -1$ for every $x \in \mathbf{B}(\m)$; we define further
\[\mathbf{E}(\m) = \left\{x \in \mathbf{B}(\m): S_x(k) \ge 0 \text{ for every } 1 \le k \le M-1\right\}.\]
The sets $\mathbf{E}(\m)$ and $\tree(\m)$ are in one-to-one correspondence: each path $S_x$ with $x$ in $\mathbf{E}(\m)$ is the {\L}ukasiewicz path of a tree in $\tree(\m)$. For $x \in \mathbf{B}(\m)$ and $j \in \{1, \dots, M\}$, denote by $x^{(j)} \in \mathbf{B}(\m)$ the $j$-th cyclic shift of $x$ defined by
\[x^{(j)}_k = x_{k+j \text{ mod } M},
\qquad
1 \le k \le M.\]
It is well-known that, given $x \in \mathbf{B}(\m)$, we have $x^{(j)} \in \mathbf{E}(\m)$ if and only if $j$ is the least time at which the walk $S_x$ achieves its minimum overall value:
\begin{equation}\label{eq:argim_pont_echangeable}
j = \inf\left\{1 \le k \le M : S_x(k) = \inf_{1 \le i \le M} S_x(i)\right\}.
\end{equation}
Given $x \in \mathbf{B}(\m)$, we let $x^\ast$ be the unique cyclic shift of $x$ in $\mathbf{E}(\m)$. It is a standard fact that if $x$ has the uniform distribution in $\mathbf{B}(\m)$, then the time $j$ satisfying \eqref{eq:argim_pont_echangeable} has the uniform distribution on $\{1, \dots, M\}$ and furthermore $x^\ast = x^{(j)}$ is uniformly distributed in $\mathbf{E}(\m)$ and is independent of $j$.

\begin{proof}[Proof of Proposition \ref{prop:moments_marche_Luka}]
According to the previous remark, we know that $W_n$ is distributed as $S_{x^\ast}$ where $x$ has the uniform distribution in $\mathbf{B}(\n)$. With the previous notation, $M=N_\n+1$ and
\[\varsigma^2
= (N_\n+1) \sigma^2_\n + \frac{N_\n^2}{N_\n+1} - N_\n + 1
= (N_\n+1) \sigma^2_\n + \frac{1}{N_\n+1}.\]
We then apply Lemma \ref{lem:Addario_Berry_borne_sous_gaussienne} to $S_{x^\ast}$: for every $t \ge 1$, for every $1 \le k-j \le \lfloor N_\n/2\rfloor$,
\begin{align*}
\Pr{S_{x^\ast}(j) - \min_{j \le i \le k} S_{x^\ast}(i) \ge t}
&= \Pr{-\min_{0 \le i \le k-j} S_x(i) \ge t}
\\
&\le \exp\left(-\frac{t^2}{(16 (\sigma^2_\n + \frac{1}{N_\n+1}) + \frac{8}{3}(1 - \frac{1}{N_\n+1}) (k-j)}\right),
\end{align*}
which corresponds to the first claim, with $C = \sup_{n \ge 1} \{16 (\sigma^2_\n + \frac{1}{N_\n+1}) + \frac{8}{3}(1 - \frac{1}{N_\n+1})\} < \infty$; the second claim follows by integrating this tail bound applied to $t^{1/r}$.
\end{proof}

We next show that the vertices of $T_n$ with a given offspring are in some sense uniformly distributed for large $n$. If $T \in \tree$ is a tree and $u_0, \dots, u_N$ are its vertices listed in lexicographical order, then for every set $A \subset \Z_+$ and every integer $1 \le i \le N+1$, we let
\[\Lambda_{T,i}(A) = \#\left\{0 \le j \le i-1 : k_{u_j} \in A\right\}\]
be the number of vertices of $T$ amongst the first $i$ which have a number of children in $A$. The next result shows that this quantity grows roughly linearly with $i$.

\begin{prop}\label{prop:repartition_feuilles}
Assume that \eqref{eq:H} holds and sample $T_n$ uniformly at random in $\tree(\n)$ for every $n \ge 1$. Then for every $A \subset \Z_+$,
\[\Pr{\max_{1 \le i \le N_\n+1} \left|\Lambda_{T_n,i}(A) - p_\n(A) i\right| > N_\n^{3/4}} \cv 0.\]
\end{prop}

\begin{proof}
For every $y \in \mathbf{B}(\n)$, every $A \subset \Z_+$ and every $1 \le i \le N_\n+1$, set
\[\lambda_{y,i}(A) = \#\{1 \le k \le i : y_k+1 \in A\}.\]
Note that $\lambda_{y, N_\n+1}(A) = (N_\n+1) p_\n(A)$. As previously discussed, the {\L}ukasiewicz path of $T_n$ has the law of $S_x$ where $x$ is uniformly distributed in $\mathbf{E}(\n)$, so
\[\Pr{\max_{1 \le i \le N_\n+1} \left|\Lambda_{T_n,i}(A) - p_\n(A) i\right| > N_\n^{3/4}} 
= \Pr{\max_{1 \le i \le N_\n} \left|\lambda_{x, i}(A) - p_\n(A) i\right| > N_\n^{3/4}}.\]

Let us first consider $y$ uniformly distributed in $\mathbf{B}(\n)$. For each $1 \le i \le N_\n+1$ fixed, $\lambda_{y,i}(A) = \sum_{k=1}^i \ind{y_k+1 \in A}$ is the sum of $i$ dependent Bernoulli random variables, which arise from a sampling without replacement in an urn with initial configuration of $\sum_{i \in A} n_i$ ``good'' balls and $N_\n+1-\sum_{i \in A} n_i$ ``bad'' balls. It is well-known that the expected value of any continuous convex function of $\lambda_{y,i}(A)$ is bounded above by the corresponding quantity for the sum of $i$ i.i.d. Bernoulli random variables with parameter $p_\n(A)$, which arise from sampling with replacement, see e.g. Hoeffding's seminal paper \cite[Theorem 4]{Hoeffding:Probability_inequalities_for_sums_of_bounded_random_variables}. In particular, the Chernoff bound for binomial random variables still holds and yields
\begin{align*}
\Pr{\max_{1 \le i \le N_\n} \left|\lambda_{y, i}(A) - p_\n(A) i\right| > N_\n^{3/4}}
&\le N_\n \max_{1 \le i \le N_\n} \Pr{\left|\lambda_{y, i}(A) - p_\n(A) i\right| > N_\n^{3/4}}
\\
&\le 2 N_\n \max_{1 \le i \le N_\n} \exp\left(- 2 N_\n^{3/2}/i\right)
\\
&= o(N_\n^{-1}).
\end{align*}

Next, let $j$ be as in \eqref{eq:argim_pont_echangeable} and recall that $j$ is uniformly distributed in $\{1, \dots, N_\n+1\}$ and that $x = y^\ast = y^{(j)}$ is uniformly distributed in $\mathbf{E}(\n)$ and independent of $j$. If $j = N_\n+1$, then $x=y$ and our claim follows from the above bound. We then implicitly condition $j$ to be less than $N_\n+1$, in which case it has the uniform distribution in $\{1, \dots, N_\n\}$ and it is independent of $x$. Observe that $N_\n+1-j$ also has the uniform distribution in $\{1, \dots, N_\n\}$ and is independent of $x$, so
\[\Pr{\max_{1 \le i \le N_\n} \left|\lambda_{x, i}(A) - p_\n(A) i\right| > N_\n^{3/4}}
\le N_\n \Pr{\left|\lambda_{x, N_\n+1-j}(A) - p_\n(A) (N_\n+1-j)\right| > N_\n^{3/4}}.\]
Furthermore, in our coupling, $\lambda_{x, N_\n+1-j}(A) = \#\{1 \le k \le N_\n+1-j : x_k+1 \in A\}$ is also equal to $\#\{1 \le k \le N_\n+1-j : y_{N_\n+2-k}+1 \in A\}$. By time-reversal, we have the identity
\[\left((y_{N_\n+2-k} ; 1 \le k \le N_\n+1) ; N_\n+1-j\right)
\eqloi
\left((y_k ; 1 \le k \le N_\n+1) ; j'\right),\]
where $j' = \sup\{0 \le k \le N_\n : S_y(k) = \max_{1 \le l \le N_\n+1} S_x(l)\}$. We conclude that
\[\Pr{\max_{1 \le i \le N_\n+1} \left|\Lambda_{T_n, i}(A) - p_\n(A) i\right| > N_\n^{3/4}}
\le N_\n \Pr{\left|\lambda_{y, j'}(A) - p_\n(A) j'\right| > N_\n^{3/4}} + \Pr{j=N_\n+1},\]
which converges to $0$ as $n \to \infty$.
\end{proof}

\subsection{Tightness of the label process}
\label{subsec:tension_labels}

Let us prove the tightness of the label process; jointly with Proposition \ref{prop:cv_label_unif_multi}, this will end the proof of Theorem \ref{thm:cv_fonctions_serpent}.

\begin{prop}\label{prop:tension_labels}
For every $n \ge 1$, sample $(T_n, l_n)$ uniformly at random in $\bartree(\n)$. Under \eqref{eq:H}, the sequence
\[\left(N_\n^{-1/4} L_n(N_\n t) ; t \in [0,1]\right)_{n \ge 1}\]
is tight in $\mathscr{C}([0,1], \R)$.
\end{prop}

In the remainder of this section, we shall use the notation $C(q)$ for a positive constant which depends only on a real number $q$ and, implicitly, on the sequences $\n$, and which will often differ from one line to another.

We shall prove that, for some sequence of events $\mathcal{E}_n$ satisfying $\P(\mathcal{E}_n) \to 1$ as $n \to \infty$ (those from Corollary \ref{cor:bon_evenement_tension_labels}), for every $q > 4$, for every $\beta \in (0, q/4-1)$, for every $n$ large enough, for every $i,j \in \{0, \dots, N_\n\}$,
\begin{equation}\label{eq:tension_labels_Kolmogorov}
\Esc{\left|L_n(i) - L_n(j)\right|^q}{\mathcal{E}_n} \le C(q) \cdot N_\n^{q/4} \cdot \left|\frac{i-j}{N_\n}\right|^{1+\beta}.
\end{equation}
Set $L_{(n)}(t) = N_\n^{-1/4} L_n(N_\n t)$ for $n \in \N$ and $t \in [0,1]$, then the previous display reads
\[\Esc{\left|L_{(n)}(s) - L_{(n)}(t)\right|^q}{\mathcal{E}_n} \le C(q) \cdot |s-t|^{1+\beta},\]
whenever $s, t \in [0, 1]$ are such that $N_\n s$ and $N_\n t$ are both integers. Since $L_{(n)}$ is defined by linear interpolation between such times, this bound then holds for every $s, t \in [0,1]$ (possibly with a different constant $C(q)$). Since $q$ can be chosen arbitrarily large, the standard Kolmogorov criterion then implies the following bound for the H\"{o}lder norm of $L_{(n)}$: for every $\alpha \in (0,1/4)$,
\[\lim_{K \to \infty} \limsup_{n \to \infty} \Prc{\sup_{0 \le s \ne t \le 1} \frac{|L_{(n)}(s) - L_{(n)}(t)|}{|s-t|^\alpha} > K}{\mathcal{E}_n} = 0;\]
since $\P(\mathcal{E}_n) \to 1$ as $n \to \infty$, we obtain
\[\lim_{K \to \infty} \limsup_{n \to \infty} \Pr{\sup_{0 \le s \ne t \le 1} \frac{|L_{(n)}(s) - L_{(n)}(t)|}{|s-t|^\alpha} > K} = 0,\]
and the sequence $(L_{(n)} ; n \ge 1)$ is tight in $\mathscr{C}([0,1], \R)$.

The proof of \eqref{eq:tension_labels_Kolmogorov} relies on the coding of $T_n$ by its {\L}ukasiewicz path. The next lemma, whose proof is left as an exercise, gathers some deterministic results that we shall need (we refer to e.g. Le Gall \cite{Le_Gall:Random_trees_and_applications} for a thorough discussion of such results). In order to simplify the notation, we identify for the remainder of this section the vertices of a one-type tree with their index in the lexicographic order: if $u$ and $u'$ are the $i$-th and $i'$-th vertices of $T_n$, we write $u \le K$ if $i \le K$, $W_n(u)$ for $W_n(i)$ and $|u-u'|$ for $|i-i'|$, the lexicographic distance between $u$ and $u'$. Recall also that $uj$ is the $j$-th child of a vertex $u$.

\begin{lem}\label{lem:codage_marche_Luka}
Let $T$ be a one-type plane tree and $W$ be its {\L}ukasiewicz path. Fix a vertex $u \in T$, then
\[W(u k_u) = W(u),
\qquad
W(uj') = \inf_{[uj,uj']} W
\qquad\text{and}\qquad
j' - j = W(uj) - W(uj')\]
for every $1 \le j \le j' \le k_u$.
\end{lem}

In the course of the proof of \eqref{eq:tension_labels_Kolmogorov}, we shall need the following two ingredients. First, a consequence of the so-called Marcinkiewicz--Zygmund inequality, see e.g. Gut \cite[Theorem 8.1]{Gut:Probability_a_graduate_course}: fix $q \ge 2$ and consider independent and centred random variables $Y_1, \dots, Y_m$ which admit a finite $q$-th moment, then there exists $C(q) \in (0, \infty)$ such that
\[\frac{1}{C(q)} \cdot \Es{\left(\sum_{i=1}^m \left|Y_i\right|^2\right)^{q/2}}
\le \Es{\left|\sum_{i=1}^m Y_i\right|^q}
\le C(q) \cdot \Es{\left(\sum_{i=1}^m \left|Y_i\right|^2\right)^{q/2}}.\]
Consider the right-most term, and raise it temporarily to the power $2/q$ in order to apply the triangle inequality for the $L^{q/2}$-norm, the second inequality thus yields the following bound:
\begin{equation}\label{eq:Marcinkiewicz_Zygmund}
\Es{\left|\sum_{i=1}^m Y_i\right|^q}
\le C(q) \cdot \left(\sum_{i=1}^m \Es{\left|Y_i\right|^q}^{2/q}\right)^{q/2}.
\end{equation}

Second, for every $r \ge 1$, consider $X^{(r)}$ a uniform random bridge in $\mathcal{B}^+_r$, defined in \eqref{eq:def_pont_sans_saut_negatif}; Le Gall \& Miermont \cite[Lemma 1]{Le_Gall-Miermont:Scaling_limits_of_random_planar_maps_with_large_faces} have shown that for every $q \ge 2$ and every $i,j \in \{0, \dots, r\}$,
\begin{equation}\label{eq:borne_moments_pont_uniforme}
\Es{\left|X^{(r)}_i - X^{(r)}_j\right|^q} \le C(q) \cdot |i-j|^{q/2}.
\end{equation}

\begin{proof}
[Proof of Proposition \ref{prop:tension_labels}]

Recall that we identify the vertices of $T_n$ with their index in the lexicographic order. Fix $q > 4$, $\beta \in (0, q/4-1)$, $n$ large enough so that $\mathcal{E}_n$ defined in Corollary \ref{cor:bon_evenement_tension_labels} has probability larger than $1/2$, and two integers $0 \le u < v \le N_\n+1$ with $v-u \le \lfloor N_\n / 2\rfloor$; we aim at showing
\[\Esc{|l_n(u) - l_n(v)|^q}{\mathcal{E}_n} \le C(q) \cdot N_\n^{q/4} \cdot \left|\frac{u-v}{N_\n}\right|^{1+\beta}.\]
Let $u \wedge v$, be the most recent common ancestor of $u$ and $v$ in $T_n$ and further $\hat{u}$ and $\hat{v}$ be the children of $u \wedge v$ which are respectively ancestor of $u$ and $v$. We stress that $u$ and $v$ are deterministic times, whereas $u \wedge v$, $\hat{u}$ and $\hat{v}$ are random and measurable with respect to $T_n$. We write:
\[l_n(u) - l_n(v) = \left(\sum_{w \in \rrbracket \hat{u}, u\rrbracket} l_n(w) - l_n(pr(w))\right) + (l_n(\hat{u}) - l_n(\hat{v})) + \left(\sum_{w \in \rrbracket \hat{v}, v\rrbracket} l_n(pr(w)) - l_n(w)\right).\]

Recall the notation $1 \le \chi_{\hat{u}} \le \chi_{\hat{v}} \le k_{u \wedge v}$ for the relative position of $\hat{u}$ and $\hat{v}$ among the children of $u \wedge v$. By construction of the labels on $T_n$, the bound \eqref{eq:borne_moments_pont_uniforme} reads in our context:
\[\Esc{\left|l_n(\hat{u}) - l_n(\hat{v})\right|^q}{T_n} \le C(q) \cdot (\chi_{\hat{v}} - \chi_{\hat{u}})^{q/2}.\]
Next, fix $w \in \rrbracket \hat{u}, u\rrbracket$, since $l_n(pr(w)) = l_n(pr(w) k_{pr(w)})$, as previously, the bound \eqref{eq:borne_moments_pont_uniforme} gives:
\[\Esc{|l_n(w) - l_n(pr(w))|^q}{T_n}
\le C(q) \cdot (k_{pr(w)} - \chi_w)^{q/2}.\]
Similarly, for every $w \in \rrbracket \hat{v}, v\rrbracket$, we have
\[\Esc{|l_n(pr(w)) - l_n(w)|^q}{T_n}
\le C(q) \cdot \chi_w^{q/2}.\]
According to the inequality \eqref{eq:Marcinkiewicz_Zygmund}, we thus have
\begin{align}\label{eq:tension_branches_gauche_droite}
\Esc{|l_n(u) - l_n(v)|^q}{T_n}
&\le C(q) \cdot 
\left(
\sum_{w \in \rrbracket \hat{u}, u\rrbracket} (k_{pr(w)} - \chi_w)
+ (\chi_{\hat{v}} - \chi_{\hat{u}})
+ \sum_{w \in \rrbracket \hat{v}, v\rrbracket} \chi_w
\right)^{q/2}\nonumber
\\
&\le C(q) \cdot 
\left(
\left(\sum_{w \in \rrbracket \hat{u}, u\rrbracket} (k_{pr(w)} - \chi_w) + (\chi_{\hat{v}} - \chi_{\hat{u}})\right)^{q/2}
+ \left(\sum_{w \in \rrbracket \hat{v}, v\rrbracket} \chi_w\right)^{q/2}
\right).
\end{align}

Let us first consider the first term in \eqref{eq:tension_branches_gauche_droite}. Appealing to Lemma \ref{lem:codage_marche_Luka}, we have
\[\chi_{\hat{v}} - \chi_{\hat{u}}
= W_n(\hat{u}) - W_n(\hat{v}),\]
and similarly, for every $w \in \rrbracket \hat{u}, u\rrbracket$,
\[k_{pr(w)} - \chi_w
= W_n(w) - W_n(pr(w) k_{pr(w)})
= W_n(w k_w) - W_n(pr(w) k_{pr(w)}),\]
so
\[\sum_{w \in \rrbracket \hat{u}, u\rrbracket} (k_{pr(w)} - \chi_w) + (\chi_{\hat{v}} - \chi_{\hat{u}})
= W_n(u) - W_n(\hat{v})
= W_n(u) - \inf_{[u, v]} W_n.\]
Proposition \ref{prop:moments_marche_Luka} then yields
\[\Esc{\left(\sum_{w \in \rrbracket \hat{u}, u\rrbracket} (k_{pr(w)} - \chi_w) + (\chi_{\hat{v}} - \chi_{\hat{u}})\right)^{q/2}}{\mathcal{E}_n}
\le C(q) \cdot |u - v|^{q/4}
\le C(q) \cdot N_\n^{q/4} \cdot \left|\frac{u-v}{N_\n}\right|^{1+\beta}.\]

We next focus on the second term in \eqref{eq:tension_branches_gauche_droite}. We would like to proceed symmetrically but there is a technical issue: on the branch $\rrbracket \hat{u}, u\rrbracket$, we strongly used the fact that $l_n(wk_w) = l_n(w)$ and this does no hold on $\rrbracket \hat{v}, v\rrbracket$: we do not have $l_n(w1)=l_n(w)$ in general. Let $T_n^-$ be the ``mirror image'' of $T_n$, i.e. the tree obtained from $T_n$ by flipping the order of the children of every vertex; let us write $w^- \in T_n^-$ for the mirror image of a vertex $w \in T_n$; make the following observations:
\begin{itemize}
\item $T_n^-$ has the same law as $T_n$, so in particular, its {\L}ukasiewicz path has the same law as that of $T_n$;
\item for every $w \in \rrbracket \hat{v}, v\rrbracket$, the quantity $\chi_w-1$ in $T_n$ corresponds to the quantity $k_{pr(w^-)} - \chi_{w^-}$ in $T_n^-$;
\item the lexicographical distance between the last descendant in $T_n^-$ of respectively $\hat{v}^-$ and $v^-$ is smaller than the lexicographical distance between $\hat{v}$ and $v$ in $T_n$ (the elements of $\rrbracket \hat{v}, v\rrbracket = \rrbracket \hat{v}^-, v^-\rrbracket$ are missing). 
\end{itemize}
With theses observations, the previous argument used to control the branch $\rrbracket \hat{u}, u\rrbracket$ shows that
\[\Esc{\left(\sum_{w \in \rrbracket \hat{v}, v\rrbracket} (\chi_w-1)\right)^{q/2}}{\mathcal{E}_n}
\le C(q) \cdot |u-v|^{q/4}
\le C(q) \cdot N_\n^{q/4} \cdot \left|\frac{u-v}{N_\n}\right|^{1+\beta}.\]
Since $\chi_w \le 2 (\chi_w-1)$ whenever $\chi_w \ge 2$, it only remains to show that 
\[\Esc{\#\{w \in \rrbracket \hat{v}, v\rrbracket : \chi_w = 1\}^{q/2}}{\mathcal{E}_n} \le C(q) \cdot N_\n^{q/4} \cdot \left|\frac{u-v}{N_\n}\right|^{1+\beta}.\]
Let $C$ and $h_\n$ be as in Corollary \ref{cor:bon_evenement_tension_labels}. On the one hand, since $h_\n$ is small compared to any positive power of $N_\n$, we have for $n$ large enough,
\[\Es{\#\{w \in \rrbracket \hat{v}, v\rrbracket : \chi_w = 1\}^{q/2} \ind{\#\rrbracket \hat{v}, v\rrbracket \le h_\n}}
\le h_\n^{q/2}
\le N_\n^{q/4} \cdot \left|\frac{u-v}{N_\n}\right|^{1+\beta}.\]
On the other hand, if $\#\rrbracket \hat{v}, v\rrbracket > h_\n$, then on the event $\mathcal{E}_n$, we know that 
\[\#\{w \in \rrbracket \hat{v}, v\rrbracket : \chi_w = 1\}
\le C \cdot \#\{w \in \rrbracket \hat{v}, v\rrbracket : \chi_w \ge 2\}
\le C \sum_{w \in \rrbracket \hat{v}, v\rrbracket} (\chi_w-1).\]
We then conclude from the previous bound.
\end{proof}

\begin{rem}\label{rem:moments_records_pont_echangeable}
It is possible that the following stronger bound than \eqref{eq:tension_labels_Kolmogorov} holds: for every $q > 4$ and every $0 \le u < v \le N_\n+1$,
\begin{equation}\label{eq:tension_labels_Kolmogorov_optimal}
\Es{\left|L_n(u) - L_n(v)\right|^q} \le C(q) \cdot |u-v|^{q/4}.
\end{equation}
Indeed, the only missing point in the previous proof is the last bound on the moments of $\#\{w \in \rrbracket \hat{v}, v\rrbracket : \chi_w = 1 \text{ and } k_{pr(w)} \ge 2\}$.\footnote{Note that we did not include the condition $k_{pr(w)} \ge 2$ in the previous proof but the increment of label is zero if $k_{pr(w)} = 1$.} Observe that
\begin{align*}
\#\{w \in \rrbracket \hat{v}, v\rrbracket : \chi_w = 1 \text{ and } k_{pr(w)} \ge 2\}
&\le \#\left\{w \in [u, v[ : W_n(w) < \inf_{]w, v]} W_n\right\}
\\
&\mathop{=}^{(d)} \#\left\{w \in ]0, v-u] : S_n(w) > \sup_{[0, w[} S_n\right\}
\\
&\le \sup_{0 \le w \le v-u} S_n(w),
\end{align*}
where $S_n$ is a uniform random bridge in $\mathbf{B}(\n)$, as defined in Section \ref{subsec:concentration}; it is obtained by first taking the $v$-th cyclic shift of $W_n$ and then going backward in time and space.

Under the stronger assumption that $\Delta_\n$ is uniformly bounded (which is the case for e.g. uniform random $2\kappa$-angulations), Proposition \ref{prop:moments_marche_Luka} shows that for every $r > 0$,
\[\Es{\left(\sup_{0 \le w \le v-u} S_n(w)\right)^r} \le C(r) \cdot |u-v|^{r/2},\]
uniformly for $n \in \N$ and $0 \le u < v \le N_\n+1$ such that $|u-v| \le \lfloor N_\n/2\rfloor$, which yields \eqref{eq:tension_labels_Kolmogorov_optimal}.

On another model, Miermont \cite[Proof of Proposition 8]{Miermont:Invariance_principles_for_spatial_multitype_Galton_Watson_trees}, obtained the bound
\[\Es{\left(\#\left\{w \in ]0, v-u] : S(w) = \sup_{[0, w]} S\right\}\right)^r} \le C(r) \cdot |u-v|^{r/2},\]
where $S$ is a centred random walk with finite variance. The argument used in the proof of Lemma \ref{lem:borne_ecart_max_pont} enables us to extend it to such a walk conditioned to be at $-1$ at time $N_\n+1$. This case corresponds to Boltzmann random maps (with generic critical weight sequence) studied in Section \ref{sec:Boltzmann}, for which \eqref{eq:tension_labels_Kolmogorov_optimal} therefore holds.
\end{rem}

\section{Convergence of random maps}
\label{sec:carte_brownienne}

In this short section we deduce Theorem \ref{thm:cv_carte} from Theorem \ref{thm:cv_fonctions_serpent}, following the argument of Le Gall \cite[Section 8.3]{Le_Gall:Uniqueness_and_universality_of_the_Brownian_map} and \cite[Section 3]{Le_Gall:The_topological_structure_of_scaling_limits_of_large_planar_maps}. First, observe that every map in $\Map(\n)$ has $n_0+1$ vertices so, if $\map_n$ has the uniform distribution in $\Map(\n)$ and $\map^\star_n$ is a pointed map obtained by distinguishing a vertex of $\map_n$ uniformly at random, then $\map^\star_n$ has the uniform distribution in $\barMap(\n)$. It is therefore sufficient to prove Theorem \ref{thm:cv_carte} with $\map_n$ replaced by $\map^\star_n$.

Let $\map^\star_n$ be a (deterministic) pointed and rooted planar map in $\barMap(\n)$ and denote by $\star$ its origin; let $(\T_n, \ell_n)$ be its associated two-type labelled tree via the $\BDG$ bijection and let $(c^\circ_0, \dots, c^\circ_{N_\n})$ be the white contour sequence of the latter. Recall that the vertices $c^\circ_i$ are identified to the vertices of $\map_n$ different from $\star$. For every $i,j \in \{0, \dots, N_\n\}$, we set
\[d_n(i,j) = \dgr(c^\circ_i, c^\circ_j),\]
where $\dgr$ is the graph distance of $\map_n$. We then extend $d_n$ to a continuous function on $[0,N_\n]^2$ by ``bilinear interpolation'' on each square of the form $[i,i+1] \times [j,j+1]$ as in \cite[Section 2.5]{Le_Gall:Uniqueness_and_universality_of_the_Brownian_map}. Recall the convention $c^\circ_{N_\n+i}=c^\circ_i$ for every $0 \le i \le N_\n$ and the interpretation, at the very end of Section \ref{subsec:BDG}, of the labels as distances from $\star$ in $\map_n$: for every $0 \le i \le N_\n$,
\begin{equation}\label{eq:distance_origine_carte}
\dgr(\star, c^\circ_i) = \mathcal{L}^\circ_n(i) - \min_{0 \le j \le N_\n} \mathcal{L}^\circ_n(j) + 1.
\end{equation}
Then, using the triangle inequality at a point where a geodesic from $c^\circ_i$ to $\star$ and a geodesic from $c^\circ_j$ to $\star$ merge, Le Gall \cite[Equation 4]{Le_Gall:Uniqueness_and_universality_of_the_Brownian_map} obtains the bound
\begin{equation}\label{eq:bornes_distances_carte}
d_n(i,j) \le \mathcal{L}^\circ_n(i) + \mathcal{L}^\circ_n(j) - 2 \max\left\{\min_{i \le k \le j} \mathcal{L}^\circ_n(k) ; \min_{j \le k \le N_\n+i} \mathcal{L}^\circ_n(k)\right\} + 2.
\end{equation}
See also Lemma 3.1 in \cite{Le_Gall:The_topological_structure_of_scaling_limits_of_large_planar_maps} for a detailed proof in a slightly different context.

Define for every $t \in [0,1]$:
\[\mathcal{C}_{(n)}(t)= \left(\frac{\sigma_p^2}{16p_0^2} \frac{1}{N_\n}\right)^{1/2} \mathcal{C}_n(2N_\n t),
\qquad\text{and}\qquad
\mathcal{L}^\circ_{(n)}(t) = \left(\frac{9}{4 \sigma_p^2} \frac{1}{N_\n}\right)^{1/4} \mathcal{L}^\circ_n(N_\n t),\]
and for every $s,t \in [0,1]$:
\begin{align*}
d_{(n)}(s, t) &= \left(\frac{9}{4 \sigma_p^2} \frac{1}{N_\n}\right)^{1/4} d_n(N_\n s, N_\n t),
\\
D_{\mathcal{L}^\circ_{(n)}}(s, t)
&= \mathcal{L}^\circ_{(n)}(s) + \mathcal{L}^\circ_{(n)}(t) - 2 \max\left\{\check{\mathcal{L}}^\circ_{(n)}(s); \check{\mathcal{L}}^\circ_{(n)}(t)\right\},
\end{align*}
where $\check{\mathcal{L}}^\circ_{(n)}$ is defined in a similar way as $\check{Z}$ in Section \ref{subsec:serpent_brownien}.

\begin{prop}\label{prop:convergence_distances_carte}
Let $(\T_n, \ell_n)$ have the uniform distribution in $\bartree_{\circ, \bullet}(\n)$ for every $n \ge 1$. Under \eqref{eq:H}, the convergence in distribution of continuous paths
\[\left(\mathcal{C}_{(n)}(t), \mathcal{L}^\circ_{(n)}(t), d_{(n)}(s, t)\right)_{s,t \in [0,1]}
\cvloi
(\exc_t, Z_t, \mathscr{D}(s,t))_{s,t \in [0,1]},\]
holds, where $\mathscr{D}$ is defined in Section \ref{subsec:serpent_brownien}.
\end{prop}

\begin{proof}
The convergence \eqref{eq:cv_contour_labels}, jointly with Remark \ref{rem:contour_contour_blanc} yields the convergence in distribution
\[\left(\mathcal{C}_{(n)}(t), \mathcal{L}^\circ_{(n)}(t), D_{\mathcal{L}^\circ_{(n)}}(s, t)\right)_{s,t \in [0,1]}
\cvloi
(\exc_t, Z_t, D_Z(s,t))_{s,t \in [0,1]}.\]
The bound \eqref{eq:bornes_distances_carte} implies further the tightness of $(d_{(n)} ; n \ge 1)$, see Proposition 3.2 in \cite{Le_Gall:The_topological_structure_of_scaling_limits_of_large_planar_maps} for a proof in a similar context. Therefore, from every sequence of integers converging to $\infty$, we can extract a subsequence along which we have
\begin{equation}\label{eq:convergence_distances_sous_suite}
\left(\mathcal{C}_{(n)}(t), \mathcal{L}^\circ_{(n)}(t), d_{(n)}(s, t)\right)_{s,t \in [0,1]}
\cvloi
(\exc_t, Z_t, D(s,t))_{s,t \in [0,1]},
\end{equation}
where $(D(s,t) ; 0 \le s,t \le 1)$ depends a priori on the subsequence. We claim that
\[D = \mathscr{D} \qquad\text{almost surely}.\]
From the bound \eqref{eq:bornes_distances_carte}, $D$ is bounded above by $D_Z$, also (see Proposition 3.3 in \cite{Le_Gall:The_topological_structure_of_scaling_limits_of_large_planar_maps}), one can check that $D$ is a pseudo-metric on $[0,1]$ which satisfies $D(s,t) = 0$ as soon as $d_\exc(s,t) = 0$. It thus follows from the maximality property discussed in section \ref{subsec:serpent_brownien} that $D \le \mathscr{D}$ almost surely. Our aim is to show the following: let $X, Y$ be i.i.d. uniform random variables on $[0,1]$ such that the pair $(X, Y)$ is independent of everything else, then
\begin{equation}\label{eq:identite_distances_carte_brownienne_repointee}
D(X,Y) \eqloi D(s_\star, Y) \enskip = \enskip Z_Y - Z_{s_\star},
\end{equation}
where $s_\star$ is the (a.s. unique \cite{Le_Gall-Weill:Conditioned_Brownian_trees}) point at which $Z$ attains its minimum. The second equality is a continuous analog of \eqref{eq:distance_origine_carte} which can be obtained from the latter by letting $n \to \infty$ along the same subsequence as in \eqref{eq:convergence_distances_sous_suite}. Le Gall \cite[Corollary 7.3]{Le_Gall:Uniqueness_and_universality_of_the_Brownian_map} has proved that \eqref{eq:identite_distances_carte_brownienne_repointee} holds true when $D$ is replaced by $\mathscr{D}$. In particular, if \eqref{eq:identite_distances_carte_brownienne_repointee} holds, then $D(X, Y)$ is distributed as $\mathscr{D}(X,Y)$. Since we know that $D \le \mathscr{D}$ almost surely, this implies $D(X, Y) = \mathscr{D}(X,Y)$ almost surely which, by a density argument, implies $D = \mathscr{D}$ almost surely.

Let us prove \eqref{eq:identite_distances_carte_brownienne_repointee}. We adapt the argument of Bettinelli \& Miermont \cite[Lemma 32]{Bettinelli-Miermont:Compact_Brownian_surfaces_I_Brownian_disks}. Recall that the white contour sequence of $\T_n$ is denoted by $(c^\circ_0, \dots, c^\circ_{N_\n})$ and let $v_1, \dots, v_{n_0}$ be its white vertices listed in the order of their last visit in the contour sequence; for example the root is $v_{n_0}$. For $1 \le i \le n_0$, let $g(i) \in \{1, \dots, N_\n\}$ be the index such that $c^\circ_{g(i)}$ is the last visit of $v_i$. Observe that $(c^\circ_{g(1)}, \dots, c^\circ_{g(n_0)}) = (v_1, \dots, v_{n_0})$ is an enumeration of the white vertices of $\T_n$ without redundancies. We then set $g(0) = 0$ and extend $g$ linearly to a continuous function on $[0,n_0]$. Let us prove that
\begin{equation}\label{eq:approximation_sites_aretes_carte}
\left(\frac{g(n_0 t)}{N_\n} ; t \in [0,1]\right) \cvproba (t ; t \in [0,1]).
\end{equation}
Let $\Lambda(0)=0$ and for every $1 \le j \le N_\n$, let
\[\Lambda(j) = \#\left\{1 \le i \le n_0 : v_i \in \{c^\circ_0, \dots, c^\circ_j\} \text{ and } v_i \notin \{c^\circ_{j+1}, \dots, c^\circ_{N_\n}\}\right\},\]
denote the number of vertices fully explored at time $j$ in the white contour exploration. Then \eqref{eq:approximation_sites_aretes_carte} is equivalent to
\[\left(\frac{\Lambda(N_\n t)}{n_0} ; t \in [0,1]\right) \cvproba (t ; t \in [0,1]).\]
Let $T_n$ be the image of $\T_n$ by the $\JS$ bijection; it can be checked along the same line as the proof of Lemma \ref{lem:processus_bijection_JS} that for every $1 \le j \le N_\n$, $\Lambda(j)$ denotes the number $\Lambda_{T_n, j}(0)$ of leaves among the first $j$ vertices of $T_n$ in lexicographical order. The above convergence of $\Lambda$ thus follows from Proposition \ref{prop:repartition_feuilles}.

Fix $X, Y$ i.i.d. uniform random variables on $[0,1]$ such that the pair $(X, Y)$ is independent of everything else, and set $x = c^\circ_{g(\lceil n_0 X\rceil)}$ and $y = c^\circ_{g(\lceil n_0 Y\rceil)}$. Note that $x$ and $y$ are uniform random white vertices of $\T_n$, they can therefore be coupled with two independent uniform random vertices $x'$ and $y'$ of $\map^\star_n$ in such a way that the conditional probability given $\map^\star_n$ that $(x,y) \ne (x', y')$ is at most $2(n_0+1)^{-1} \to 0$ as $n \to \infty$; we implicitly assume in the sequel that $(x,y) = (x', y')$. Since $\star$ is also a uniform random vertex of $\map^\star_n$, we obtain that
\begin{equation}\label{eq:identite_distances_carte_discrete_repointee}
\dgr(x,y) \eqloi \dgr(\star, y).
\end{equation}
By definition,
\[\dgr(x,y) = d_n(g(\lceil n_0 X\rceil), g(\lceil n_0 Y\rceil)),\]
and, according to \eqref{eq:distance_origine_carte},
\[\dgr(\star, y) = \mathcal{L}^\circ_n(g(\lceil n_0 Y\rceil)) - \min_{0 \le j \le N_\n} \mathcal{L}^\circ_n(j) + 1.\]
We obtain \eqref{eq:identite_distances_carte_brownienne_repointee} by letting $n \to \infty$ in \eqref{eq:identite_distances_carte_discrete_repointee} along the same subsequence as in \eqref{eq:convergence_distances_sous_suite}, appealing also to \eqref{eq:approximation_sites_aretes_carte}.
\end{proof}

The proof of Theorem \ref{thm:cv_carte} is then routine.

\begin{proof}[Proof of Theorem \ref{thm:cv_carte}]
We aim at showing the convergence of metric spaces
\begin{equation}\label{eq:convergence_cartes}
\left(\map_n^\star, \left(\frac{9}{4 \sigma_p^2} \frac{1}{N_\n}\right)^{1/4} \dgr\right)
\cvloi
(\mathscr{M}, \mathscr{D}),
\end{equation}
for the Gromov--Hausdorff topology. Recall (see e.g. \cite[Chapter 7.3]{Burago-Burago-Ivanov:A_course_in_metric_geometry}) that a \emph{correspondence} between two metric spaces $(X, d_X)$ and $(Y, d_Y)$ is a set $R \subset X \times Y$ such that for every $x \in X$, there exists $y \in Y$ such that $(x,y) \in R$ and vice-versa. The \emph{distortion} of $R$ is defined as
\[\mathrm{dis}(R) = \sup\left\{\left|d_X(x,x') - d_Y(y,y')\right| ; (x,y), (x', y') \in R\right\}.\]
Finally, the Gromov--Hausdorff distance between $(X, d_X)$ and $(Y, d_Y)$ is given by (\cite[Theorem 7.3.25]{Burago-Burago-Ivanov:A_course_in_metric_geometry})
\[\frac{1}{2} \cdot \inf_R \mathrm{dis}(R),\]
where the infimum is taken over all correspondences $R$ between $(X, d_X)$ and $(Y, d_Y)$.

The proof is deterministic: we show that the convergence \eqref{eq:convergence_cartes} holds whenever that in Proposition \ref{prop:convergence_distances_carte} does. Indeed, let $(\map_n^\star\setminus\{\star\}, \dgr)$ be the metric space given by the vertices of $\map_n^\star$ different from $\star$ and their graph distance in $\map_n^\star$ and observe that the Gromov--Hausdorff distance between $(\map_n^\star, \dgr)$ and $(\map_n^\star\setminus\{\star\}, \dgr)$ is bounded by one. Recall that the vertices of $\map_n^\star$ different from $\star$ are in bijection with the white vertices of its associated two-type tree $\T_n$, which are given (with redundancies) by the white contour sequence $(c^\circ_0, \dots, c^\circ_{N_\n})$. Let $\Pi$ be the canonical projection $\CRT_\exc \to \mathscr{M} = \CRT_\exc / \approx$, then the set
\[\mathcal{R}_n = \left\{\left(c^\circ_{\lfloor N_\n t\rfloor}, \Pi(\pi_\exc(t))\right) ; t \in [0,1]\right\}.\]
is a correspondence between $(\map_n^\star\setminus\{\star\}, (\frac{9}{4 \sigma_p^2} \frac{1}{N_\n})^{1/4} \dgr)$ and $(\mathscr{M}, \mathscr{D})$ and its distortion is given by
\[\sup_{s,t \in [0,1]} \left|d_{(n)}(\lfloor N_\n s\rfloor/N_\n, \lfloor N_\n t\rfloor/N_\n) - \mathscr{D}(s,t)\right|,\]
which tends to $0$ whenever the convergence in Proposition \ref{prop:convergence_distances_carte} holds. This concludes the proof.
\end{proof}

\section{Boltzmann random maps}
\label{sec:Boltzmann}

In this last section, we state and prove the results alluded in Section \ref{subsec:intro_Boltzmann} on Boltzmann random maps. Let us make a preliminary remark: we shall divide by real numbers which depend on an integer $n$, and consider conditional probabilities with respect to events which depend on $n$; we shall therefore, if necessary, implicitly restrict ourselves to those values of $n$ for which such quantities are well-defined and statements such as ``as $n \to \infty$'' should be understood along the appropriate sequence of integers. Let us fix a sequence of non-negative real numbers $\q = (q_i ; i \ge 0)$ which, in order to avoid trivialities, satisfies $q_i > 0$ for at least one $i \ge 2$.

\subsection{Rooted and pointed Boltzmann maps}

Let $\barMap$ be the set of all rooted and pointed bipartite maps, that we shall view as pairs $(\map, \star)$, where $\map \in \Map$ is a rooted bipartite map, and $\star$ is a vertex of $\map$. We adapt the distributions described in Section \ref{subsec:intro_Boltzmann} to such maps by setting
\[W^{\q, \star}((\map, \star)) = W^\q(\map) = \prod_{f \in \mathrm{Faces}(\map)} q_{\mathrm{deg}(f)/2},
\qquad (\map, \star) \in \barMap,\]
where $\mathrm{Faces}(\map)$ is the set of faces of $\map$ and $\mathrm{deg}(f)$ is the degree of such a face $f$. We set $Z_\q^\star = W^{\q, \star}(\barMap)$.

\begin{defi}
The sequence $\q$ is called \emph{admissible} when $Z_\q^\star$ is finite.\footnote{In Section \ref{subsec:intro_Boltzmann}, we considered unpointed maps and denoted the total mass by $Z_\q$. Clearly, if $Z_\q^\star$ is finite, then so is $Z_\q$. It can be shown that the converse implication holds, see e.g. \cite{Bernardi-Curien-Miermont:A_Boltzmann_approach_to_percolation_on_random_triangulations}, so the notion of admissibility is the same for pointed and unpointed maps.}
\end{defi}

If $\q$ is admissible, we set
\[\P^{\q, \star}(\cdot) = \frac{1}{Z_\q^\star} W^{\q, \star}(\cdot).\]
For every integer $n \ge 2$, let $\barMap_{E=n}$, $\barMap_{V=n}$ and $\barMap_{F=n}$ be the subsets of $\barMap$ of those maps with respectively $n-1$ edges, $n+1$ vertices (these shifts by one will simplify the statements) and $n$ faces. More generally, for every $A \subset \N$, let $\barMap_{F,A=n}$ be the subset of $\barMap$ of those maps with $n$ faces whose degree belongs to $2A$ (and possibly other faces, but with a degree in $2\N \setminus 2A$). For every $S = \{E, V, F\} \cup \bigcup_{A \subset \N} \{F,A\}$ and every $n \ge 2$, we define
\[\P^{\q, \star}_{S=n}((\map, \star)) = \P^{\q, \star}((\map, \star) \mid (\map, \star) \in \barMap_{S=n}),
\qquad (\map, \star) \in \barMap_{S=n},\]
the law of a rooted and pointed Boltzmann map conditioned to have size $n$.

Given the sequence $\q$, set
\begin{equation}\label{eq:poids_Galton_Watson_cartes_Boltzmann}
\overline{q}_0 = 1
\qquad\text{and}\qquad
\overline{q}_k = \binom{2k-1}{k-1} q_k
\quad\text{for}\quad k \ge 1,
\end{equation}
and define the power series
\begin{equation}\label{eq:series_entiere_cartes_Boltzmann}
g_\q(x) = \sum_{k \ge 0} x^k \overline{q}_k, \qquad x \ge 0.
\end{equation}
Denote by $R_\q$ its radius of convergence, note that $g_\q$ is convex, strictly increasing and continuous on $[0, R_\q]$ and $g_\q(0)=1$. In particular, it has at most two fixed points, necessarily in $(1, R_\q]$; in fact, we have the following exclusive four cases:
\begin{enumerate}
\item There are no fixed points.
\item There are two fixed points $1 < x_1 < x_2 \le R_\q$, moreover $g_\q'(x_1) < 1$ and $g_\q'(x_2) > 1$.
\item There is a unique fixed point $1 < x \le R_\q$, with $g_\q'(x) < 1$.
\item There is a unique fixed point $1 < x \le R_\q$, with $g_\q'(x) = 1$.
\end{enumerate}

Marckert \& Miermont \cite{Marckert-Miermont:Invariance_principles_for_random_bipartite_planar_maps} have defined another power series $f_\q$, such that $g_\q(x) = 1 + x f_\q(x)$ for every $x \ge 0$. Proposition 1 in \cite{Marckert-Miermont:Invariance_principles_for_random_bipartite_planar_maps} reads as follows with our notation.

\begin{prop}[Marckert \& Miermont \cite{Marckert-Miermont:Invariance_principles_for_random_bipartite_planar_maps}]
\label{prop:condition_admissibilite}
The sequence $\q$ is admissible if and only if $g_\q$ has at least one fixed point. In this case, $Z_\q^\star$ is the fixed point satisfying $g_\q'(Z_\q^\star) \le 1$.
\end{prop}

The proof in \cite{Marckert-Miermont:Invariance_principles_for_random_bipartite_planar_maps} is based on the $\BDG$ bijection, we shall present a short adaption in Section \ref{subsec:convergence_Boltzmann} using the composition of the $\BDG$ and the $\JS$ bijections. 
Following \cite{Marckert-Miermont:Invariance_principles_for_random_bipartite_planar_maps} let us introduce more terminology. 

\begin{defi}
An admissible sequence $\q$ is called \emph{critical} when $Z_\q^\star$ is the unique fixed point of $g_\q$ and satisfies moreover $g_\q'(Z_\q^\star) = 1$. It is called \emph{generic critical} when it is admissible, critical, and $g_\q''(Z_\q^\star) < \infty$, and \emph{regular critical} when moreover $Z_\q^\star < R_\q$.
\end{defi}

Note that an admissible sequence $\q$ induces a probability measure on $\Z_+$ with mean smaller than or equal to one:
\begin{equation}\label{eq:loi_GW_carte_Boltzmann}
p_\q(k) = (Z_\q^\star)^{k-1} \binom{2k-1}{k-1} q_k,
\qquad k \ge 0.
\end{equation}
Indeed,
\[\sum_{k \ge 0} p_\q(k) = \frac{g_\q(Z_\q^\star)}{Z_\q^\star} = 1,
\qquad\text{and}\qquad
\sum_{k \ge 0} k p_\q(k) = g_\q'(Z_\q^\star) \le 1.\]
This distribution has mean $1$ if and only if $\q$ is critical, and in this case, its variance is \begin{equation}\label{eq:constante_cartes_Boltzmann}
\Sigma_\q^2 = \left(\sum_{k \ge 0} k^2 p_\q(k)\right) - 1
= \left.\left(\frac{\d}{\d x} x g_\q'(x)\right)\right|_{x=Z_\q^\star} - 1
= Z_\q^\star g_\q''(Z_\q^\star),
\end{equation}
which is finite if and only if $\q$ is generic critical. In terms of the function $f_\q$ from \cite{Marckert-Miermont:Invariance_principles_for_random_bipartite_planar_maps}, we have $\Sigma_\q^2 = (2 + (Z_\q^\star)^3 f_\q''(Z_\q^\star))/Z_\q^\star$. The argument of \cite[Proposition 7]{Marckert-Miermont:Invariance_principles_for_random_bipartite_planar_maps} show that if $\q$ is regular critical, then $p_\q$ admits small exponential moments.

\begin{thm}\label{thm:cartes_Boltzmann}
Suppose $\q$ is generic critical, define $p_\q$ by \eqref{eq:loi_GW_carte_Boltzmann} and $\Sigma_\q^2$ by \eqref{eq:constante_cartes_Boltzmann} and for every subset $A \subset \N$, define
\[C^\q_E = 1,
\qquad
C^\q_V = p_\q(0) = \frac{1}{Z_\q^\star},
\qquad
C^\q_F = 1-p_\q(0) = 1-\frac{1}{Z_\q^\star},
\qquad
C^\q_{F,A} = p_\q(A).\]
Fix $S \in \{E, V, F\} \cup \bigcup_{A \subset \N} \{F,A\}$ and for every $n \ge 2$, sample $\map_n$ from $\P^\q_{S=n}$, then the convergence in distribution
\[\left(\map_n, \left(\frac{9}{4} \frac{C^\q_S}{\Sigma_\q^2} \frac{1}{n}\right)^{1/4} \dgr\right)
\cvloi
(\mathscr{M}, \mathscr{D}),\]
holds in the sense of Gromov--Hausdorff.
\end{thm}

Note that the Boltzmann laws in this statement are \emph{not} the pointed versions. We shall prove first that it holds under the pointed version $\P^{\q, \star}_{S=n}$, relying on the composition of the $\BDG$ and $\JS$ bijections to check that \eqref{eq:H} is fulfilled with the probability $p_\q$ given by \eqref{eq:loi_GW_carte_Boltzmann}. Then we will show that $\P^{\q, \star}_{S=n}$ and $\P^\q_{S=n}$ are close as $n \to \infty$; the argument of the latter will closely follow that of Bettinelli \& Miermont \cite[Section 7.2]{Bettinelli-Miermont:Compact_Brownian_surfaces_I_Brownian_disks}, see also Abraham \cite[Section 6]{Abraham:Rescaled_bipartite_planar_maps_converge_to_the_Brownian_map}, and Bettinelli, Jacob \& Miermont \cite[Section 3]{Bettinelli-Jacob-Miermont:The_scaling_limit_of_uniform_random_plane_maps_via_the_Ambjorn_Budd_bijection}.

\begin{rem}
Le Gall \cite[Theorem 9.1]{Le_Gall:Uniqueness_and_universality_of_the_Brownian_map} obtained this result in the case $S=V$, when $\q$ is supposed to be regular critical, not only generic critical. Bettinelli \& Miermont \cite[Theorem 5]{Bettinelli-Miermont:Compact_Brownian_surfaces_I_Brownian_disks} also obtained similar convergences in the three cases $S = E, V, F$ for Boltzmann maps \emph{with a boundary}, associated with regular critical weights. Theorem \ref{thm:cartes_Boltzmann} completes (and improves since we only assume $\q$ to be generic critical) their Remark 2.
\end{rem}

Note that $\Map_{E=n}$ is finite for every $n \ge 2$ so the Boltzmann distribution $\P^\q_{E=n}$ makes sense even if $Z_\q = \infty$. The proof of Theorem \ref{thm:cartes_Boltzmann} shows that we do not need $\q$ to be admissible in this case.

\begin{thm}\label{thm:cartes_Boltzmann_n_aretes}
Suppose there exists $x > 0$ (necessarily unique) such that
\[g_\q(x) < \infty,
\qquad
x g_\q'(x) = g_\q(x),
\qquad\text{and}\qquad
x g_\q''(x) < \infty.\]
Then if $\map_n$ is sampled from $\P^\q_{E=n}$ for every $n \ge 2$, the convergence in distribution
\[\left(\map_n, \left(\frac{9}{4} \frac{g_\q(x)}{x^2 g_\q''(x)} \frac{1}{n}\right)^{1/4} \dgr\right)
\cvloi
(\mathscr{M}, \mathscr{D}),\]
holds in the sense of Gromov--Hausdorff.
\end{thm}

If $\q$ is generic critical, then the assumptions are fulfilled by $x = Z_\q^\star$: we have $g_\q(Z_\q^\star) = Z_\q^\star$ so $x g_\q'(x) = g_\q(x)$ is equivalent to $g_\q'(Z_\q^\star) = 1$ and then
\[\frac{g_\q(x)}{x^2 g_\q''(x)} = \frac{1}{Z_\q^\star g_\q''(Z_\q^\star)} = \frac{1}{\Sigma_\q^2} = \frac{C^\q_E}{\Sigma_\q^2},\]
so Theorem \ref{thm:cartes_Boltzmann_n_aretes} recovers Theorem \ref{thm:cartes_Boltzmann}.

As an application of Theorem \ref{thm:cartes_Boltzmann_n_aretes}, consider the case $q_k = 1$ for every $k \ge 1$, then $\P^\q_{E=n}$ is the uniform distribution in $\Mapb_{E=n}$. In this case, $g_\q$ has a radius of convergence equal to $1/4$ and is given by
\[g_\q(x)
= 1 + \sum_{k \ge 1} x^k \binom{2k-1}{k-1}
= \frac{1+\sqrt{1-4x}}{2 \sqrt{1-4x}},
\qquad 0 < x  < 1/4.\]
Furthermore,
\[x g_\q'(x) = g_\q(x) \quad\text{if and only if}\quad x = \frac{3}{16},
\qquad\text{and then}\qquad
\frac{g_\q(3/16)}{(3/16)^2 g_\q''(3/16)} = \frac{9}{2},\]
so Theorem \ref{thm:cartes_Boltzmann_n_aretes} yields Corollary \ref{cor:cv_cartes_biparties}.

The proofs of Theorems \ref{thm:cartes_Boltzmann} and \ref{thm:cartes_Boltzmann_n_aretes} use the notion of \emph{simply generated trees} that we next recall.

\subsection{Simply generated trees}
\label{subsec:arbres_SG}

Let us define a measure on the set of finite one-type tree $\tree$ by
\[\Theta^\q(T) = \prod_{u \in T} w(k_u), \qquad T \in \tree.\]
Let $\Upsilon_\q = \Theta^\q(\tree)$, if the latter is finite, we define a probability measure on $\tree$ by
\[\SG^\q(\cdot)= \frac{1}{\Upsilon_\q} \Theta^\q(\cdot).\]
A random tree sampled according to $\SG^\q$ is called a \emph{simply generated tree}. Such distributions have been introduced by Meir \& Moon \cite{Meir_Moon-On_the_altitude_of_nodes_in_random_trees} and studied in great detail by Janson \cite{Janson:Simply_generated_trees_conditioned_Galton_Watson_trees_random_allocations_and_condensation} on the set of trees with a given number of vertices. A particular case is when the weight sequence $\q$ is a probability measure on $\Z_+$ with mean less than or equal to one: in this case, $\Upsilon_\q = 1$ and $\SG^\q = \Theta^\q$ is the law of a \emph{subcritical Galton--Watson tree} with offspring distribution $\q$; we denote it by $\GW^\q$. When the expectation of $\q$ is exactly equal to one, we say that $\q$ (as well as any random tree sampled from $\GW^\q$) is \emph{critical}.

Note that we may define simply generated trees with $n$ vertices even if $\Upsilon_\q$ is infinite by rescaling the measure $\Theta^\q$ restricted to this finite set by its total mass.

\begin{lem}\label{lem:arbres_SG_et_GW}
Let us denote by $\#T$ the number of vertices of a tree $T \in \tree$.
\begin{enumerate}[ref={\thelem(\roman*)}]
\item\label{lem:arbres_SG_et_GW_1}
 Fix $c > 0$ and set $\tilde{q}_k = c^{k-1} q_k$ for every $k \ge 0$. Then $\Upsilon_{\tilde{\q}} < \infty$ if and only if $\Upsilon_\q < \infty$ and in this case, the laws $\SG^{\tilde{\q}}$ and $\SG^\q$ coincide.

\item\label{lem:arbres_SG_et_GW_2} Fix $a,b > 0$ and set $\hat{q}_k = a b^k q_k$ for every $k \ge 0$. Then the conditional laws $\SG^{\hat{\q}}(\, \cdot \mid \#T = n)$ and $\SG^\q(\, \cdot \mid \#T = n)$ coincide for all $n \ge 1$.
\end{enumerate}
\end{lem}

\begin{proof}
Note that for every tree $T\in \tree$, one has $\sum_{u \in T} k_u = \#T-1$ and so $\sum_{u \in T} (k_u - 1) = -1$; it follows that
\[\Theta^{\tilde{\q}}(T) = \prod_{u \in T} c^{k_u-1} q_{k_u} = c^{-1} \Theta^\q(T),\]
so $\Upsilon_{\tilde{\q}} = c^{-1} \Upsilon_\q$ and the first claim follows. Similarly,
\[\Theta^{\hat{\q}}(T) = \prod_{u \in T} a b^{k_u} q_{k_u} = a^{\#T} b^{\#T-1} \Theta^\q(T),\]
so $\Theta^{\hat{\q}}(\{T \in \tree : \#T=n\}) = a^n b^{n-1} \Theta^\q(\{T \in \tree : \#T=n\})$ and the second claim follows.
\end{proof}

We shall use Lemma \ref{lem:arbres_SG_et_GW} with sequences $\tilde{\q}$ or $\hat{\q}$ which are probability measures with mean $1$ so, in the first case, $\SG^{\tilde{\q}} = \GW^{\tilde{\q}}$ is the law of a critical Galton--Watson tree, and in the second case, $\SG^{\hat{\q}}(\, \cdot \mid \#T = n) = \GW^{\hat{\q}}(\, \cdot \mid \#T = n)$ is the law of such a tree conditioned to have $n$ vertices.

We close this section with two results on size-conditioned critical Galton--Watson; the proofs are deferred to Section \ref{sec:appendice_GW}. We first claim that the empirical degree sequence of a Galton--Watson tree conditioned to be large satisfies \eqref{eq:H}. For a plane tree $T$ and an integer $i \ge 0$, let us denote by $n_T(i) = \#\{u \in T : k_u = i\}$ the number of vertices of $T$ with $i$ children. For any subset $A \subset \Z_+$, set $n_T(A) = \sum_{i \in A} n_i(T)$; note that $n_T(\Z_+)$ is the total number of vertices of $T$, $n_T(0)$ is its number of leaves and $n_T(\N)$ its number of internal vertices. Consider the empirical offspring distribution of $T$ and its variance, given by
\[p_T(i) = \frac{n_T(i)}{n_T(\Z_+)}
\quad\text{for}\quad i \ge 0
\qquad\text{and}\qquad
\sigma^2_T = \sum_{i \ge 0} i^2 p_T(i) - \left(\frac{n_T(\Z_+)-1}{n_T(\Z_+)}\right)^2,\]
and finally set $\Delta_T = \max\{i \ge 0 : n_T(i) > 0\}$.

\begin{prop}\label{prop:Galton_Watson_hypothese_H_general}
Let $\mu$ be a critical distribution in $\Z_+$ with variance $\sigma^2 \in (0,\infty)$ and fix $A \subset \Z_+$; under $\GW^\mu(\, \cdot \mid n_T(A) = n)$, the convergence
\[\left(p_T, \sigma^2_T, n_T(\Z_+)^{-1/2} \Delta_T\right)
\cvproba
(\mu, \sigma^2, 0),\]
holds in probability.
\end{prop}

This result was obtained by Broutin \& Marckert \cite[Lemma 11]{Broutin-Marckert:Asymptotics_of_trees_with_a_prescribed_degree_sequence_and_applications} in the case $A=\Z_+$. Their proof extends to the general case using arguments due to Kortchemski \cite{Kortchemski:Invariance_principles_for_Galton_Watson_trees_conditioned_on_the_number_of_leaves}.

Finally, we claim that the inverse of the number of leaves, normalised to have expectation $1$, converges to $1$ in $L^1$.

\begin{lem}\label{lem:biais_sites_GW}
Let $\mu$ be a critical distribution in $\Z_+$ with variance $\sigma^2 \in (0,\infty)$. For every $A \subset \Z_+$, we have
\[\lim_{n \to \infty} \GW^\mu\left[\left|\frac{1}{n_T(0)} \frac{1}{\GW^\mu[\frac{1}{n_T(0)} \mid n_T(A) = n]} - 1\right| \;\middle|\; n_T(A) = n\right] = 0.\]
\end{lem}

\subsection{Convergence of Boltzmann random maps}
\label{subsec:convergence_Boltzmann}

We first prove the convergence of rooted and pointed Boltzmann maps, using the $\BDG$ and the $\JS$ bijections, and next compare the pointed and non pointed Boltzmann laws to deduce Theorems \ref{thm:cartes_Boltzmann} and \ref{thm:cartes_Boltzmann_n_aretes}.

\begin{prop}\label{prop:cartes_Boltzmann_pointees}
Theorems \ref{thm:cartes_Boltzmann} and \ref{thm:cartes_Boltzmann_n_aretes} hold under their respective assumptions when the measures $\P^\q_{S=n}$ are replaced by their pointed version $\P^{\q, \star}_{S=n}$.
\end{prop}

The main idea is to observe that for every $n \ge 2$ and $S \in \{E, V, F\} \cup \bigcup_{A \subset \N} \{F,A\}$, the composition of the $\BDG$ and the $\JS$ bijections maps the set $\barMap_{S=n}$ onto the subset of $\tree$ of those trees $T$ satisfying $n_T(B_S) = n$, where for every $A \subset \N$,
\begin{equation}\label{eq:conditionnement_arbres_cartes}
B_E = \Z_+,
\qquad
B_V = \{0\},
\qquad
B_F = \N
\qquad\text{and}\qquad
B_{F,A} = A.
\end{equation}

\begin{proof}
Fix a rooted and pointed map $(\map, \star) \in \barMap$ and let $(T, l)$ be its associated labelled one-type tree after the $\BDG$ and then the $\JS$ bijections. Recall that the faces of $\map$ are in bijection with the internal vertices of $T$, whereas the vertices of $\map$ different from $\star$ are in bijection with the leaves of $T$; in particular, with the notation of the previous subsection, for every $i \ge 1$, the number of faces of $\map$ of degree $2i$ is given by $n_T(i)$, and its number of vertices minus one by $n_T(0)$. Thereby,
\[W^{\q, \star}((\map, \star))
= \prod_{f \in \mathrm{Faces}(\map)} q_{\mathrm{deg}(f)/2}
= \prod_{u \in T : k_u \ge 1} q_{k_u}.\]
Recall also from \eqref{eq:nombre_etiquetages_arbre} the number of possible labellings of a given plane tree. The measure $W^{\q, \star}$ on $\barMap$ thus induces a measure on $\tree$, where each $T \in \tree$ is given the weight
\[\prod_{u \in T : k_u \ge 1} \binom{2k_u-1}{k_u-1} q_{k_u}
= \Theta^{\overline{\q}}(T),\]
where $\overline{\q}$ is given by \eqref{eq:poids_Galton_Watson_cartes_Boltzmann}. This shows that if $(\map, \star)$ has the law $\P^{\q, \star}$ and $(T, l)$ its associated labelled one-type tree after the $\BDG$ and then the $\JS$ bijections, then $T$ has the law $\SG^{\overline{\q}}$. Similarly, for every $n \ge 2$ and $S \in \{E, V, F\} \cup \bigcup_{A \subset \N} \{F,A\}$, if $(\map, \star)$ has the law $\P^{\q, \star}_{S=n}$, then $T$ has the law $\SG^{\overline{\q}}(\, \cdot \mid n_T(B_S) = n)$, where $B_S$ is given by \eqref{eq:conditionnement_arbres_cartes}. Furthermore, in both cases, conditional on the tree $T$, the labelling $l$ is uniformly distributed amongst all possibilities.

Let us now prove that Theorem \ref{thm:cartes_Boltzmann_n_aretes} holds for the pointed maps sampled from $\P^{\q, \star}_{E=n}$. Suppose that $x > 0$ is such that
\[g_\q(x) < \infty,
\qquad
x g_\q'(x) = g_\q(x),
\qquad\text{and}\qquad
x g_\q''(x) < \infty.\]
Define a probability measure on $\Z_+$ similar to \eqref{eq:loi_GW_carte_Boltzmann} where $Z_\q^\star$ is replaced by $x$:
\begin{equation}\label{eq:loi_GW_carte_Boltzmann_n_aretes}
\mu_\q(k) = \frac{x^k \overline{q}_k}{g_\q(x)},
\qquad k \ge 0.
\end{equation}
Note that
 $\mu_\q$ has expectation
\[\sum_{k \ge 0} k \mu_\q(k)
= \frac{x g_\q'(x)}{g_\q(x)}
= 1,\]
and variance
\[\sum_{k \ge 0} k^2 \mu_\q(k) - 1
= \frac{x g_\q'(x) + x^2 g_\q''(x)}{g_\q(x)} - 1
= \frac{x^2 g_\q''(x)}{g_\q(x)} \in (0, \infty).\]
According to Lemma \ref{lem:arbres_SG_et_GW_2}, the tree $T$ has the law $\GW^{\mu_\q}(\, \cdot \mid n_T(\Z_+) = n)$, Proposition \ref{prop:Galton_Watson_hypothese_H_general} and Skorohod's representation Theorem ensure then that, on some probability space, \eqref{eq:H} is fulfilled almost surely with $p = \mu_\q$ and we conclude from Theorem \ref{thm:cv_carte}.

The proof of the fact that Theorem \ref{thm:cartes_Boltzmann} holds for the pointed maps sampled from $\P^{\q, \star}_{S=n}$ is similar. If $\q$ is generic critical, then $Z_\q^\star$ satisfies the above assumptions on $x$ and furthermore $g_\q(Z_\q^\star) = Z_\q^\star$ so $\mu_\q$ is the probability $p_\q$ given by \eqref{eq:loi_GW_carte_Boltzmann}:
\[\mu_\q(k) = p_\q(k) = (Z_\q^\star)^{k-1} \overline{q}_k,
\qquad k \ge 0.\]
According to Lemma \ref{lem:arbres_SG_et_GW_1}, the tree $T$ has the law $\GW^{p_\q}(\, \cdot \mid n_T(B_S) = n)$. Again, Proposition \ref{prop:Galton_Watson_hypothese_H_general} ensures then that \eqref{eq:H} is fulfilled with $p = p_\q$ and the claim follows.
\end{proof}

We have seen all the ingredients to prove Proposition \ref{prop:condition_admissibilite}. The proof is inspired from \cite{Marckert-Miermont:Invariance_principles_for_random_bipartite_planar_maps}.

\begin{proof}[Proof of Proposition \ref{prop:condition_admissibilite}]
Let $\overline{\q}$ be given by \eqref{eq:poids_Galton_Watson_cartes_Boltzmann}. According to the previous proof, we have
\[Z^\star_\q = \sum_{(\map, \star) \in \barMap} W^{\q, \star}((\map, \star)) = \sum_{T \in \tree} \Theta^{\overline{\q}}(T) = \Upsilon_{\overline{\q}},\]
Suppose that this quantity is finite, we next decompose the second sum according to the degree of the root of $T$. If the latter is $k$, then $T$ is made of $k$ trees, say $T_1, \dots, T_k$, attached to a common root; this leads to the following equation:
\[\sum_{T \in \tree} \Theta^{\overline{\q}}(T)
= \sum_{k \ge 0} \overline{q}_k \sum_{T_1, \dots, T_k \in \tree} \prod_{i=1}^k \Theta^{\overline{\q}}(T_i)
= \sum_{k \ge 0} \overline{q}_k \left(\sum_{T \in \tree} \Theta^{\overline{\q}}(T)\right)^k,\]
in other words $Z^\star_\q = g_\q(Z^\star_\q)$. Let us prove furthermore that $g_\q'(Z^\star_\q) \le 1$. Since $Z^\star_\q = g_\q(Z^\star_\q)$, the sequence $p_\q$ defined by $p_\q(k) = (Z^\star_\q)^{k-1} \overline{q}_k$ for every $k \ge 0$ is a probability and $g_\q'(Z^\star_\q)$ is its mean. According to Lemma \ref{lem:arbres_SG_et_GW_1}, the law $\SG^{\overline{\q}}$ coincides with $\SG^{p_\q}$ so
\[\sum_{T \in \tree} \SG^{p_\q}(T)
= \frac{1}{\Upsilon_{\overline{\q}}} \sum_{T \in \tree} \Theta^{\overline{\q}}(T)
= 1.\]
We conclude that $\SG^{p_\q} = \GW^{p_\q}$ is the law of a sub-critical Galton--Watson tree with offspring distribution $p_\q$, which has therefore mean $g_\q'(Z^\star_\q) \le 1$.

Conversely, suppose that $g_\q$ has at least one fixed point and let us prove that $Z^\star_\q$ is finite. Recall that one of the fixed points, say, $x > 0$, must satisfy $g_\q'(x) \le 1$; we set $\mu_\q(k) = x^{k-1} \overline{q}_k$ for every $k \ge 0$, the previous calculations show that $\mu_\q$ is a probability measure with mean $g_\q'(x) \le 1$. According to (the proof of) Lemma \ref{lem:arbres_SG_et_GW_1}, we have
\[\frac{1}{x} Z^\star_\q
= \frac{1}{x} \sum_{(\map, \star) \in \barMap} W^{\q, \star}((\map, \star))
= \frac{1}{x} \sum_{T \in \tree} \Theta^{\overline{\q}}(T)
= \sum_{T \in \tree} \Theta^{\mu_\q}(T)
=1.\]
We conclude that $Z^\star_\q = x$ is indeed finite.
\end{proof}

Finally, we show that the pointed and non pointed Boltzmann laws are close to each other, following arguments from \cite{Abraham:Rescaled_bipartite_planar_maps_converge_to_the_Brownian_map, Bettinelli-Jacob-Miermont:The_scaling_limit_of_uniform_random_plane_maps_via_the_Ambjorn_Budd_bijection, Bettinelli-Miermont:Compact_Brownian_surfaces_I_Brownian_disks}. Theorems \ref{thm:cartes_Boltzmann} and \ref{thm:cartes_Boltzmann_n_aretes} follow from Propositions \ref{prop:cartes_Boltzmann_pointees} and \ref{prop:biais_cartes_Boltzmann_pointees}.

\begin{prop}\label{prop:biais_cartes_Boltzmann_pointees}
Fix $S \in \{E, V, F\} \cup \bigcup_{A \subset \N} \{F,A\}$ and let $\q$ satisfy the assumptions of Theorem \ref{thm:cartes_Boltzmann} or of Theorem \ref{thm:cartes_Boltzmann_n_aretes} if $S=E$. Let $\phi : \barMap \to \Map : (M, \star) \mapsto M$ and let $\phi_* \P^{\q, \star}_{S=n}$ be the push-forward measure induced on $\Map$ by $\P^{\q, \star}_{S=n}$, then
\[\left\|\P^\q_{S=n} - \phi_* \P^{\q, \star}_{S=n}\right\|_{TV} \cv 0,\]
where $\|\cdot\|_{TV}$ refers to the total variation norm.
\end{prop}

\begin{proof}
For each pointed map $(\map, \star) \in \barMap$, let $V(\map)$ be the number of vertices of $\map$. If $T$ is the one-type tree associated with $(\map, \star)$, then $V(\map) = n_T(0)-1$. Notice that $\P^{\q, \star}_{S=n}$ is absolutely continuous with respect to $\P^\q_{S=n}$: for every measurable and bounded function $f : \Map \to \R$, we have
\[\E^\q_{S=n}\left[f(\map)\right]
= \E^{\q, \star}_{S=n}\left[V(\map)^{-1}\right]^{-1} \E^{\q, \star}_{S=n}\left[V(\map)^{-1} f \circ \phi ((\map, \star))\right].\]
Let $p_\q$ be given by \eqref{eq:loi_GW_carte_Boltzmann} or \eqref{eq:loi_GW_carte_Boltzmann_n_aretes} in the case $S=E$ and let $B_S$ be given by \eqref{eq:conditionnement_arbres_cartes}. We have
\begin{align*}
\left\|\P^\q_{S=n} - \phi_* \P^{\q, \star}_{S=n}\right\|_{TV}
&= \frac{1}{2} \sup_{-1 \le f\le 1} \left|\E^\q_{S=n}\left[f(\map)\right] - \E^{\q, \star}_{S=n}\left[f \circ \phi ((\map, \star))\right]\right|
\\
&\le \frac{1}{2} \sup_{-1 \le f\le 1} \E^{\q, \star}_{S=n}\left[\left|\left(\E^{\q, \star}_{S=n}\left[V(\map)^{-1}\right]^{-1} V(\map)^{-1} - 1\right) f \circ \phi ((\map, \star))\right|\right]
\\
&\le \E^{\q, \star}_{S=n}\left[\left|\E^{\q, \star}_{S=n}\left[V(\map)^{-1}\right]^{-1} V(\map)^{-1} - 1\right|\right]
\\
&= \GW^{p_\q}\left[\big|\GW^{p_\q}[(n_T(0)-1)^{-1} \mid n_T(B_S)=n]^{-1} (n_T(0)-1)^{-1} - 1\big| \;\middle|\; n_T(B_S)=n\right].
\end{align*}
Lemma \ref{lem:biais_sites_GW} states that the last quantity above tends to zero as $n \to \infty$, which concludes the proof.
\end{proof}

\subsection{On Galton--Watson trees conditioned to be large}
\label{sec:appendice_GW}

It remains to prove Proposition \ref{prop:Galton_Watson_hypothese_H_general} and Lemma \ref{lem:biais_sites_GW}. The proof of the former result relies on the coding of a tree by its {\L}ukasiewicz path which, in the case of Galton--Watson trees is an excursion of a certain random walk. Our proofs use many results from \cite{Kortchemski:Invariance_principles_for_Galton_Watson_trees_conditioned_on_the_number_of_leaves} (see in particular sections 6 and 7 there), written explicitly for $A = \{0\}$ but which hold true in general, \emph{mutatis mutandis}, as explained in Section 8 there.

\begin{proof}[Proof of Proposition \ref{prop:Galton_Watson_hypothese_H_general}]
Fix $\varepsilon > 0$ and consider the event
\[E(\varepsilon) =
\left\{d\left(
\left(\frac{n_T(\cdot)}{n_T(\Z_+)}, \sum_{i \ge 0} (i-1)^2 \frac{n_T(i)}{n_T(\Z_+)}, \frac{\Delta_T}{n_T(\Z_+)^{1/2}}\right),
\left(\mu, \sigma^2, 0\right)
\right) > \varepsilon\right\},\]
where $d$ is a metric on the product space of probability measures on $\Z_+$ and $\R^2$, compatible with the product topology. We aim at showing
\[\GW^\mu(E(\varepsilon) \mid n_T(A) = n) \cv 0.\]

Let us denote by $(X_k; k \ge 1)$ a sequence of i.i.d. random variables with distribution $(\mu(i+1) ; i \ge -1)$ and $K_n(i) = \#\{1 \le k \le n : X_k = i-1\}$ for every $n \ge 1$ and $i \ge 0$. Consider the event
\[F(n,\varepsilon) =
\left\{d\left(
\left(\frac{K_n(\cdot)}{n}, \sum_{i \ge 0} (i-1)^2 \frac{K_n(i)}{n}, \frac{\max\{i \ge 0 : K_n(i) > 0\}}{n^{1/2}}\right),
\left(\mu, \sigma^2, 0\right)
\right) > \varepsilon\right\},\]
Broutin \& Marckert \cite{Broutin-Marckert:Asymptotics_of_trees_with_a_prescribed_degree_sequence_and_applications} have shown that
\[\P(F(n,\varepsilon)) \cv 0.\]

As in Section \ref{subsec:concentration}, given a path $x = (x_1, \dots, x_n) \in \Z^n$ such that $x_1 + \dots + x_n = -1$, we denote by $S_x(k) = x_1 + \dots + x_k$ for every $1 \le k \le n$ and by $x^* = (x_1^*, \dots, x_n^*)$ the unique cyclic shift of $x$ satisfying furthermore $S_{x^*}(k) \ge 0$ for every $1 \le k \le n-1$. Let $\zeta_r(A) = \inf\{k \ge 1 : K_k(A) = \lfloor r\rfloor\}$ for every $r \ge 1$. Kortchemski \cite[Proposition 6.5]{Kortchemski:Invariance_principles_for_Galton_Watson_trees_conditioned_on_the_number_of_leaves} shows that for every integer $n \ge 1$, 
the path $(S_{X^*}(k) ; 0 \le k \le \zeta_n(A))$ under $\P(\, \cdot \mid S_X(\zeta_n(A)) = -1)$ has the law of the {\L}ukasiewicz path of a tree $T$ under $\GW^\mu(\, \cdot \mid n_T(A) = n)$. 
Since $F(n,\varepsilon)$ is invariant under cyclic shift, it follows that
\[\GW^\mu(E(\varepsilon) \mid n_T(A) = n) = \P(F(\zeta_n(A),\varepsilon) \mid S_X(\zeta_n(A)) = -1).\]
Using a time-reversibility property of $(X_1, \dots, X_{\zeta_n(A)})$ under $\P(\, \cdot \mid S_X(\zeta_n(A)) = -1)$, see
\cite[Proposition 6.8]{Kortchemski:Invariance_principles_for_Galton_Watson_trees_conditioned_on_the_number_of_leaves}, it suffices to show that
\[\Prc{F(\zeta_{n/2}(A),\varepsilon)}{S_X(\zeta_n(A)) = -1} \cv 0.\]
As in the proof of \cite[Theorem 7.1]{Kortchemski:Invariance_principles_for_Galton_Watson_trees_conditioned_on_the_number_of_leaves}, for any $\alpha > 0$, the event $F(\zeta_{n/2}(A),\varepsilon)$ is included in the union of the following three events:
\begin{enumerate}
\item $F(\zeta_{n/2}(A),\varepsilon)
\cap \{|S_X(\zeta_{n/2}(A))| \le \alpha \sqrt{\sigma^2 n / (2\mu(A))}\}
\cap \{|\zeta_{n/2}(A) - \frac{n}{\mu(A)}| \le n^{3/4}\}$,
\item $\{|S_X(\zeta_{n/2}(A))| > \alpha \sqrt{\sigma^2 n / (2\mu(A))}\}$,
\item $\{|\zeta_{n/2}(A) - \frac{n}{\mu(A)}| > n^{3/4}\}$.
\end{enumerate}
By \cite[Lemmas 6.10 \& 6.11]{Kortchemski:Invariance_principles_for_Galton_Watson_trees_conditioned_on_the_number_of_leaves} (argument similar to the one we use in the proof of Lemma \ref{lem:borne_ecart_max_pont}, based on a local limit theorem), there exists a constant $C > 0$ independent of $\alpha$ such that for every $n$ large enough, the conditional probability $\P(\,\cdot \mid S_X(\zeta_n(A)) = -1)$ of the first event is bounded above by
\[C \cdot \Pr{F(\zeta_{n/2}(A),\varepsilon) \text{ and } \left|\zeta_{n/2}(A) - \frac{n}{\mu(A)}\right| \le n^{3/4}}.\]
Next, according to \cite[Equation 44]{Kortchemski:Invariance_principles_for_Galton_Watson_trees_conditioned_on_the_number_of_leaves},
\[\lim_{\alpha \to \infty} \lim_{n \to \infty} 
\Prc{\left|S_X(\zeta_{n/2}(A))\right| > \alpha \sqrt{\sigma^2 n / (2\mu(A))}}{S_X(\zeta_n(A)) = -1} = 0,\]
and, by \cite[Lemma 6.2(i)]{Kortchemski:Invariance_principles_for_Galton_Watson_trees_conditioned_on_the_number_of_leaves},
\[\lim_{n \to \infty} 
\Prc{\left|\zeta_{n/2}(A) - \frac{n}{\mu(A)}\right| > n^{3/4}}{S_X(\zeta_n(A)) = -1} = 0.\]
We conclude that there exists a constant $C > 0$ such that
\[\limsup_{n \to \infty} \Prc{F(\zeta_{n/2}(A),\varepsilon)}{S_X(\zeta_n(A)) = -1}
\le C \limsup_{n \to \infty} \Pr{F(\zeta_{n/2}(A),\varepsilon) \text{ and } \left|\zeta_{n/2}(A) - \frac{n}{\mu(A)}\right| \le n^{3/4}}.\]

On the event $|\zeta_{n/2}(A) - \frac{n}{\mu(A)}| \le n^{3/4}$, we have for every $i \ge 0$,
\[\frac{K_{n/\mu(A) - n^{3/4}}(i)}{n/\mu(A) + n^{3/4}} \le \frac{K_{\zeta_{n/2}(A)}(i)}{\zeta_{n/2}(A)} \le \frac{K_{n/\mu(A) + n^{3/4}}(i)}{n/\mu(A) - n^{3/4}},\]
and the claim from the fact that $\P(F(n,\varepsilon)) \to 0$ as $n \to \infty$.
\end{proof}

We next turn to the proof of Lemma \ref{lem:biais_sites_GW}. We shall need the following concentration result. For a sequence $(x_n; n \ge 1)$ of non-negative real numbers and $\delta > 0$, we write $x_n = \mathrm{oe}_\delta(n)$ if there exist $c_1, c_2 > 0$ such that $x_n \le c_1 \exp(-c_2 n^\delta)$ for every $n \ge 1$.

\begin{lem}\label{lem:concentration_feuilles_GW}
Let $\mu$ be a critical distribution in $\Z_+$ with variance $\sigma^2 \in (0,\infty)$ and fix $A \subset \Z_+$; there exists $\delta > 0$ such that
\[\GW^\mu\left(\left|\frac{n_T(0)}{n} - \frac{\mu(0)}{\mu(A)}\right| > \varepsilon \;\middle|\; n_T(A) = n\right) = \mathrm{oe}_\delta(n).\]
\end{lem}

\begin{proof}
We bound
\[\GW^\mu\left(\left|\frac{n_T(0)}{n} - \frac{\mu(0)}{\mu(A)}\right| > \varepsilon \;\middle|\; n_T(A) = n\right)
\le \frac{\GW^\mu\left(\left|\frac{n_T(0) \mu(A)}{n_T(A) \mu(0)} - 1\right| > \frac{\mu(A)}{\mu(0)} \varepsilon \;\middle|\; n_T(\Z_+) \ge n\right)}{\GW^\mu(n_T(A) = n)}.\]
According to \cite[Theorem 8.1]{Kortchemski:Invariance_principles_for_Galton_Watson_trees_conditioned_on_the_number_of_leaves}, there exists an explicit constant $C >0$ which depends only on $\mu$ and $A$ (see \cite[Theorem 3.1]{Kortchemski:Invariance_principles_for_Galton_Watson_trees_conditioned_on_the_number_of_leaves}) such that 
$\GW^\mu(n_T(A) = n) \sim C \cdot n^{-3/2}$ as $n \to \infty$. 
Moreover, from \cite[Corollary 2.6]{Kortchemski:Invariance_principles_for_Galton_Watson_trees_conditioned_on_the_number_of_leaves}, 
\[\GW^\mu\left(\left|\frac{n_T(0)}{\mu(0) n_T(\Z_+)} - 1\right| > n^{-1/4} \;\middle|\; n_T(\Z_+) \ge n\right)
= \mathrm{oe}_{1/2}(n).\]
Indeed, taking $t=1$ in \cite[Corollary 2.6]{Kortchemski:Invariance_principles_for_Galton_Watson_trees_conditioned_on_the_number_of_leaves}, we read $n_T(0) = \Lambda_T(\zeta(T))$. This result holds also when $0$ is replaced by $A$; it follows that
\[\GW^\mu\left(\left|\frac{n_T(0) \mu(A)}{n_T(A) \mu(0)} - 1\right| > \frac{\mu(A)}{\mu(0)} \varepsilon \;\middle|\; n_T(\Z_+) \ge n\right)
= \mathrm{oe}_{1/2}(n),\]
and the proof is complete.
\end{proof}

\begin{proof}[Proof of Lemma \ref{lem:biais_sites_GW}]
Fix $\varepsilon \in (0, 1)$ and observe that, since $n_T(0)^{-1} \le 1$,
\[\GW^\mu\left[\left|\frac{\mu(0) n}{\mu(A) n_T(0)} - 1\right| \;\middle|\; n_T(A) = n\right]
\le \varepsilon + \left(\frac{\mu(0) n}{\mu(A)} + 1\right) \GW^\mu\left(\left|\frac{\mu(0) n}{\mu(A) n_T(0)} - 1\right| > \varepsilon \;\middle|\; n_T(A) = n\right).\]
Next, the probability on the right-hand side is bounded above by
\[\GW^\mu\left(\frac{n_T(0)}{n} < \frac{1}{2} \frac{\mu(0)}{\mu(A)} \;\middle|\; n_T(A) = n\right)
+ \GW^\mu\left(\left|\frac{\mu(0)}{\mu(A)} - \frac{n_T(0)}{n}\right| > \frac{\varepsilon}{2} \frac{\mu(0)}{\mu(A)} \;\middle|\; n_T(A) = n\right),\]
which is $\mathrm{oe}_\delta(n)$ for some $\delta > 0$ according to Lemma \ref{lem:concentration_feuilles_GW}. This yields
\[\lim_{n \to \infty} \GW^\mu\left[\left|\frac{\mu(0) n}{\mu(A) n_T(0)} - 1\right| \;\middle|\; n_T(A) = n\right] = 0,
\quad\text{and so}\quad
\lim_{n \to \infty} \frac{\mu(0) n}{\mu(A)} \GW^\mu\left[\frac{1}{n_T(0)} \;\middle|\; n_T(A) = n\right] = 1.\]
The claim now follows from these two limits.
\end{proof}

\appendix

\section{Proof of the spinal decompositions}
\label{sec:appendice_epine}

In this section, we prove Lemma \ref{lem:lignee_multinomiale} and its extension Lemma \ref{lem:lignees_multinomiales_k}.

\subsection{The one-point decomposition}

\begin{proof}[Proof of Lemma \ref{lem:lignee_multinomiale}]
First, concerning the first good event, consider the ``mirror image'' $T_n^-$ of $T_n$, i.e. the tree obtained from $T_n$ by flipping the order of the children of every vertex. Denote by $W_n^-$ the {\L}ukasiewicz path of $T_n^-$. Observe that $T_n^-$ and $T_n$ have the same law therefore $W_n^-$ and $W_n$ as well. Furthermore, from Lemma \ref{lem:codage_marche_Luka}, we have for all $i \in \{0, \dots, N_\n\}$,
\[\LR(\mathbf{A}(u(i))) \le W_n(i) + W_n^-(i^-) + k_{u(i)},\]
where $i^-$ is the index in $T_n^-$ of the image of the $i$-th vertex of $T_n$. The convergence of $W_n$ and $H_n$ in \eqref{eq:Broutin_Marckert} then yields
\[\lim_{x \to \infty} \limsup_{n \ge 1} \Pr{\max_{u \in T_n} |u| \ge x N_\n^{1/2}}
= \lim_{x \to \infty} \limsup_{n \ge 1} \Pr{\max_{u \in T_n} \LR(\mathbf{A}(u)) \ge x N_\n^{1/2}}
= 0.\]
Regarding the second good event, let $U$ be uniformly distributed in $[0,1]$ and independent of $\exc$, then \eqref{eq:Broutin_Marckert} implies similarly that for every $x > 0$, we have
\[\limsup_{n \to \infty} \Pr{N_\n^{-1/2} |u_n| \le 1/x}
\le \Pr{2 \exc_U/\sigma_p \le 1/x},\]
which then converges to $0$ as $x \to \infty$.

Let us next turn to the comparison between $\mathbf{A}(u_n)$ conditioned on being in $\Good(n,x)$ and a multinomial sequence. Recall that we denote by $\chi_{u}$ the relative position of a vertex $u$ among its siblings. Define next for every vertex $u$ the \emph{content} of the branch $\llbracket\varnothing, u\llbracket$ as
\begin{equation}\label{eq:content}
\Cont(u) = \left(\left(k_{pr(v)}, \chi_v\right); v \in \rrbracket\varnothing, u\rrbracket\right),
\end{equation}
where the elements $v \in \rrbracket\varnothing, u\rrbracket$ are sorted in increasing order of their height. For any sequence $\m \in \Z_+^\N$, denote by $\Gamma(\m)$ the set of possible vectors $\Cont(u)$ when $\mathbf{A}(u)=\m$ and note that
\[\# \Gamma(\m) = \binom{|\m|}{(m_i ; i\ge 1)} \prod_{i \ge 1} i^{m_i}.\]
The removal of the branch $\llbracket\varnothing, u\llbracket$ from $T$ produces a plane forest of $\LR(\mathbf{A}(u))$ trees and there is a one-to-one correspondence between the pair $(T, u)$ on the one hand and this forest and $\Cont(u)$ on the other hand. For any sequence $\mathbf{q} = (q_i;i \ge 0)$ of non-negative integers with finite sum, let $\ensembles{F}(\mathbf{q})$ be the set of plane forests having exactly $q_i$ vertices with $i$ children for every $i \ge 0$; such a forest possesses $r= \sum_{i \ge 0} (1-i) q_i$ roots and it is well-known that
\[\# \ensembles{F}(\mathbf{q}) = \frac{r}{|\mathbf{q}|} \binom{|\mathbf{q}|}{(q_i ; i \ge 0)}.\]

Sample $T_n$ uniformly at random in $\tree(\n) = \ensembles{F}(\n)$ and $u_n$ uniformly at random in $T_n$, the previous  bijection readily implies that for any sequence $\m$ satisfying $m_0=0$ and $m_i \le n_i$ for every $i \ge 1$ and for any vector $C \in \Gamma(\m)$, we have
\[\Pr{\Cont(u_n) = C} = \frac{\#\ensembles{F}(\n-\m)}{(N_\n+1) \#\ensembles{F}(\n)},
\quad\text{and so}\quad
\Pr{\mathbf{A}(u_n) = \m} = \#\Gamma(\m) \cdot \frac{\#\ensembles{F}(\n-\m)}{(N_\n+1) \#\ensembles{F}(\n)}.\]
Consequently, if we set $h = |\m|$, we have
\begin{align*}
\Pr{\mathbf{A}(u_n) = \m}
&= \binom{h}{(m_i ; i\ge 1)} \prod_{i \ge 1} i^{m_i} \cdot \frac{\frac{\LR(\m)}{N_\n+1-h} \binom{N_\n+1-h}{(n_i-m_i ; i \ge 0)}}{(N_\n+1) \frac{1}{N_\n+1} \binom{N_\n+1}{(n_i ; i \ge 0)}}
\\
&= \frac{\LR(\m)}{N_\n+1-h} \cdot \frac{h!}{\prod_{i \ge 1} m_i!} \prod_{i \ge 1} \left(\frac{i n_i}{N_\n}\right)^{m_i} \cdot \prod_{i \ge 1} \frac{n_i!}{n_i^{m_i} (n_i-m_i)!} \cdot \frac{(N_\n+1-h)! N_\n^{h}}{(N_\n+1)!}.
\end{align*}
Note that
\[\Pr{\Xi_\n^{(h)} = \m} = \frac{h!}{\prod_{i \ge 1} m_i!} \prod_{i \ge 1} \left(\frac{i n_i}{N_\n}\right)^{m_i}.\]
Next, observe that $n_i! \le n_i^{m_i} (n_i-m_i)!$ for every $i \ge 1$; finally, using the inequality $(1-x)^{-1} \le \exp(2x)$ for $|x| \le 1/2$, we have as soon as $h \le N_\n/2$,
\[\frac{(N_\n+1-h)! N_\n^{h}}{(N_\n+1)!}
\le \prod_{i=0}^{h-1} \frac{1}{1 - i/(N_\n+1)}
\le \e^{h^2/N_\n}.\]
Putting things together, we obtain that if $h \le N_\n/2$, then
\[\Pr{\mathbf{A}(u_n) = \m}
\le \frac{\LR(\m)}{N_\n+1-h} \cdot \e^{h^2/N_\n} \cdot \Pr{\Xi_\n^{(h)} = \m}.\]
If $\m \in \Good(n,x)$, then $\LR(\m)$ and $h$ are both bounded above by $x N_\n^{1/2}$, so the proof is complete.
\end{proof}

\subsection{The multi-point decomposition}

We next extend the previous decomposition according to several i.i.d. uniform random vertices.

\begin{proof}[Proof of Lemma \ref{lem:lignees_multinomiales_k}]
First, the fact that the probability of $\Bin_k^+$ tends to $1$ can be seen as a consequence of \eqref{eq:Broutin_Marckert} and the fact that such a property holds almost surely for the Brownian tree. The rest of the event is similar to the previous proof and we omit the details to focus on the bound on the law of $\mathbf{A}(u_{n,1}, \dots, u_{n,k})$. Precisely, we shall prove that for every sequences $\m^{(1)}, \dots, \m^{(2k-1)} \in \Good(n,x)$, if $h_j = |\m^{(j)}|$ for each $1 \le j \le 2k-1$ and $h = h_1 + \dots + h_{2k-1}$, then
\begin{multline*}
\Prc{\mathbf{A}(u_{n,1}, \dots, u_{n,k}) = (\m^{(1)}, \dots, \m^{(2k-1)})}{\Bin_k^+}
\\
\le 2\left(\frac{\sigma^2_p}{2}\right)^{k-1} \frac{(k-1) \Delta_\n + \sum_{j = 1}^{2k-1} \LR(\m^{(j)})}{N_\n^{k-1} (N_\n-h-k+2)} 
\exp\left(\frac{h^2 + 2h(k-2)}{N_\n}\right)
\prod_{j=1}^{2k-1} \Pr{\Xi_\n^{(h_j)} = \m^{(j)}}
(1+o(1)).
\end{multline*}
Since $\Delta_\n$, each $h_j$ and each $\LR(\m^{(j)})$ is at most of order $N_\n^{1/2}$, the claim follows.

We treat in detail the case $k=2$ and comment on the general case at the end. Fix $r \ge 2$ and three sequences of non-negative integers $\m^{(1)}$, $\m^{(2)}$, $\m^{(3)}$ with $m_0^{(1)} = m_0^{(2)} = m_0^{(3)}=0$ and set $|\m_i^{(j)}| = h_j$ for each $j \in \{1, 2, 3\}$. For every $i \ge 0$, set
\[\underline{m}_i = m_i^{(1)} + m_i^{(2)} + m_i^{(3)}
\qquad\text{and}\qquad
\overline{m}_i = \underline{m}_i + \ind{i=r}.\]
Given $T_n$, we say that a pair of vertices $(u,v)$ is ``good'' if the reduced tree $T_n(u,v)$ satisfies $\Bin_2$. Observe that on the event $\{\max_{a \in T_n} |a| \le N_\n^{3/4}\}$, there are more than $N_\n^2 - o(N_\n^2) \ge N_\n^2/2$ good pairs. If $u_n$ and $v_n$ are independent uniform random vertices of $T_n$, then the conditional probability given $\{\max_{a \in T_n} |a| \le N_\n^{3/4}\}$ that this pair is good tends to $1$, and then on this event, $(u_n, v_n)$ has the uniform distribution in the set of good pairs. In the remainder of this proof, we thus assume that $(u_n, v_n)$ is a good pair sampled uniformly at random. Let $w_n$ be the most recent common ancestor of $u_n$ and $v_n$. Let $\hat{u}_n$ be the child of $w_n$ which is an ancestor of $u_n$ and define similarly $\hat{v}_n$ so this distribution. Let $w_n$ be the most recent common ancestor of $u_n$ and $v_n$. Let $\hat{u}_n$ be the child of $w_n$ which is an ancestor of $u_n$ and define similarly $\hat{v}_n$ so
\[F_n(u_n, v_n) =(\llbracket\varnothing, w_n\rrbracket, \llbracket \hat{u}_n, u_n\rrbracket, \llbracket \hat{v}_n, v_n\rrbracket).\]
Let $\Cont(u_n, v_n)$ be the triplet of contents of these branches, defined in a similar way as in \eqref{eq:content}. Let $\Gamma(\m^{(1)}, \m^{(2)}, \m^{(3)})$ be the set of possible such triplets when $\mathbf{A}(u_n, v_n) = (\m^{(1)}, \m^{(2)}, \m^{(3)})$; as previously,
\begin{align*}
\# \Gamma(\m^{(1)}, \m^{(2)}, \m^{(3)})
&= \prod_{j=1}^3 \binom{h_j}{(m_i^{(j)};i \ge 1)} \prod_{i \ge 1} i^{m_i^{(j)}}
\\
&= n_r \cdot \frac{N_\n^{h_1+h_2+h_3}}{\prod_{i \ge 1} n_i^{\overline{m}_i}} \cdot
\prod_{j=1}^3 \binom{h_j}{(m_i^{(j)};i \ge 1)} \prod_{i \ge 1} \left(\frac{i n_i}{N_\n}\right)^{m_i^{(j)}}.
\end{align*}

Observe that $\LR(\overline{\m}) = 1 + \sum_{i \ge 1} (i-1) \overline{m}_i = 2 + (r-2) + \sum_{i \ge 1} (i-1) \underline{m}_i$ denotes the number of trees in the forest obtained from $T_n$ by removing the reduced tree $T_n(u_n, v_n)$ when $\mathbf{A}(u_n, v_n) = (\m^{(1)}, \m^{(2)}, \m^{(3)})$ and $k_{w_n} = r$: there are $i-1$ components for each of the $\underline{m}_i$ elements of $\llbracket\varnothing, w_n\llbracket \cup \llbracket \hat{u}_n, u_n\llbracket \cup \llbracket \hat{v}_n, v_n\llbracket$ with $i$ children, as well as $r-2$ components corresponding to the children of $w_n$ different from $\hat{u}_n$ and $\hat{v}_n$, and the two components above $u_n$ and $v_n$. As previously, the triplet $(T_n, u_n, v_n)$ is characterised by the forest obtained by removing the reduced tree $T_n(u_n, v_n)$ and the content of the latter, which is $\Cont(u_n, v_n)$ plus the information $(k_{w_n}, \chi_{\hat{u}_n}, \chi_{\hat{v}_n})$ about the branch-point. We therefore have for every $C \in \Gamma(\m^{(1)}, \m^{(2)}, \m^{(3)})$ and every $B \in \{(r, i,j) ; 1 \le i < j \le r\}$,
\begin{align*}
\Prc{\Cont(u_n, v_n) = C \text{ and } (k_{w_n}, \chi_{\hat{u}_n}, \chi_{\hat{v}_n}) = B}{\Bin_2^+}
&\le \frac{2 \cdot \#\ensembles{F}(\n-\overline{\m})}{N_\n^2 \cdot \#\ensembles{F}(\n)}
\\
&= \frac
{2\frac{\LR(\overline{\m})}{|\n-\overline{\m}|} \binom{|\n-\overline{\m}|}{(n_i-\overline{m}_i;i \ge 1)}}
{N_\n^2 \frac{1}{N_\n+1} \binom{N_\n+1}{(n_i;i \ge 1)}}
\\
&= \frac{2}{N_\n} \frac{\LR(\overline{\m})}{N_\n |\n-\overline{\m}|} \frac{(|\n-\overline{\m}|)!}{N_\n!} \prod_{i \ge 1} \frac{n_i!}{(n_i-\overline{m}_i)!}.
\end{align*}
Since $|\n| = N_\n+1$ and $|\overline{\m}| = h_1+h_2+h_3+1 = h+1$, it follows that
\begin{align*}
&\Prc{\mathbf{A}(u_n, v_n) = (\m^{(1)}, \m^{(2)}, \m^{(3)}) \text{ and } k_{w_n} = r}{\Bin_2^+}
\\
&\le \frac{r(r-1)}{2} \cdot \# \Gamma(\m^{(1)}, \m^{(2)}, \m^{(3)}) \cdot \frac{2}{N_\n} \frac{\LR(\overline{\m})}{N_\n |\n-\overline{\m}|} \frac{(|\n-\overline{\m}|)!}{N_\n!} \prod_{i \ge 1} \frac{n_i!}{(n_i-\overline{m}_i)!}
\\
&=\frac{r(r-1) n_r}{N_\n} \cdot \frac{\LR(\overline{\m})}{N_\n (N_\n-h)} 
\cdot \frac{(N_\n-h)! N_\n^{h}}{N_\n!}
\cdot \prod_{i \ge 1} \frac{n_i!}{n_i^{\overline{m}_i} (n_i-\overline{m}_i)!} \cdot \prod_{j=1}^3 \binom{h_j}{(m_i^{(j)};i \ge 1)} \prod_{i \ge 1} \left(\frac{i n_i}{N_\n}\right)^{m_i^{(j)}}.
\end{align*}
First, under \eqref{eq:H},
\[\sum_{r \ge 2} \frac{r(r-1) n_r}{N_\n} \cv \sigma^2_p.\]
Also, note that we must have $r \le \Delta_\n$ and so
\[\LR(\overline{\m})
= r + \sum_{i \ge 1} (i-1) \underline{m}_i
= (r-3) + \sum_{j = 1}^3 \LR(\m^{(j)})
\le \Delta_\n + \sum_{j = 1}^3 \LR(\m^{(j)}).\]
Then, as previously, we have
\[\prod_{j=1}^3 \binom{h_j}{(m_i^{(j)};i \ge 1)} \prod_{i \ge 1} \left(\frac{i n_i}{N_\n}\right)^{m_i^{(j)}} = \prod_{j=1}^3 \Pr{\Xi^{(h_j)}_n = \m^{(j)}},
\qquad\text{and}\qquad
\prod_{i\ge 1}\frac{n_i!}{n_i^{\overline{m}_i} (n_i-\overline{m}_i)!} \le 1,\]
as well as, as soon as $h \le N_\n/2$,
\[\frac{(N_\n-h)! N_\n^{h}}{N_\n!}
= \prod_{i=0}^{h-1} \frac{1}{1 - i/N_\n}
\le \exp\left(h^2/N_\n\right).\]
This concludes the case $k=2$.

In the general case, the same argument applies. First, on the event $\{\max_{a \in T_n} |a| \le N_\n^{3/4}\}$, for every $n$ large enough, the number of $k$-tuples of vertices such that the associated reduced tree satisfies $\Bin_k$ is larger than $N_\n^k (1-o(1)) \ge N_\n^k/2$. Next, if $u_{n,1}, \dots, u_{n,k}$ is such a $k$-tuple sampled uniformly at random, then we may still decompose the tree according to the reduced tree $T_n(u_{n,1}, \dots, u_{n,k})$ to obtain an explicit expression of the joint law of $\mathbf{A}(u_{n,1}, \dots, u_{n,k})$ and the number of children of all the branch-points of $T_n(u_{n,1}, \dots, u_{n,k})$. Specifically, denote by $v_{n,1}, \dots, v_{n,k-1}$ these branch-points, fix $\m^{(1)}, \dots, \m^{(2k-1)}$ and $r_1, \dots, r_{k-1} \le \Delta_\n$, set $h_j = |\m^{(j)}|$ for $1 \le j \le 2k-1$ and $h= h_1 + \dots + h_{2k-1}$, as well as $\overline{m}_i = \sum_{j=1}^{2k-1} m_i^{(j)} + \sum_{j=1}^{k-1} \ind{i=r_j}$ for $i \ge 1$, so $|\overline{\m}| = h+k-1$. 
Then, we have
I\begin{align*}
&\Prc{\mathbf{A}(u_{n,1}, \dots, u_{n,k}) = (\m^{(1)}, \dots, \m^{(2k-1)}) \text{ and } k_{v_{n,j}} = r_j \text{ for every } 1 \le j \le k-1}{\Bin_k^+}
\\
&\le 2\prod_{j=1}^{k-1} \frac{r_j(r_j-1) n_{r_j}}{2 N_\n} 
\cdot \frac{\LR(\overline{\m})}{N_\n (N_\n+1-(h+k-1))} 
\cdot \frac{(N_\n+1-(h+k-1))! N_\n^h}{N_\n!}
\\
&\qquad\times \prod_{i \ge 1} \frac{n_i!}{n_i^{\overline{m}_i} (n_i-\overline{m}_i)!} \cdot \prod_{j=1}^{2k-1} \binom{h_j}{(m_i^{(j)};i \ge 1)} \prod_{i \ge 1} \left(\frac{i n_i}{N_\n}\right)^{m_i^{(j)}}.
\end{align*}
Nota that
\[\sum_{r_1, \dots, r_{k-1} \ge 2} \prod_{j=1}^{k-1} \frac{r_j(r_j-1) n_{r_j}}{2 N_\n}
= \left(\sum_{r \ge 2} \frac{r(r-1) n_r}{2 N_\n}\right)^{k-1}
\cv \left(\frac{\sigma^2_p}{2}\right)^{k-1},\]
as well as, for $h \le N_\n/2$,
\[\frac{(N_\n+1-(h+k-1))! N_\n^h}{N_\n!}
= \prod_{i=0}^{k-3} \frac{1}{N_\n - i} \cdot \prod_{i=0}^{h-1} \frac{1}{1 - (i+k-2)/N_\n}
\le \frac{1+o(1)}{N_\n^{k-2}} \cdot \exp\left(\frac{h^2 + 2h(k-2)}{N_\n}\right).\]
The rest of the proof is adapted verbatim.
\end{proof}

\section{On the maximal gap in a random walk bridge}
\label{sec:appendice_pont_echangeable}

Our aim in this section is to prove Lemma \ref{lem:borne_ecart_max_pont}. Recall that for $r \ge 1$, a discrete bridge of length $r$ is a vector $(B_0, \dots, B_r)$ satisfying $B_0 = B_r = 0$ and $B_{k+1}-B_k \in \Z$ for every $0 \le k \le r-1$. A random bridge is said to be \emph{exchangeable} if the law of its increments $(B_1, B_2-B_1, \dots, B_r-B_{r-1})$ is invariant under permutation.

\begin{lem}\label{lem:pont_echangeable}
Fix $r \ge 1$ and let $B = (B_0, \dots, B_r)$ be a discrete bridge. For every $x \ge 0$ fixed, if
\[\max_{0 \le k \le r} B_k - \min_{0 \le k \le r} B_k \ge 3x,\]
then at least one of the following quantities must be smaller than or equal to $-x$:
\[\min_{0 \le k \le \lceil r/2\rceil} B_k,
\qquad
\min_{0 \le k \le \lceil r/2\rceil} \left(B_{\lceil r/2\rceil} - B_{\lceil r/2\rceil-k}\right),\]
\[\min_{0 \le k \le \lceil r/2\rceil} \left(B_{\lceil r/2\rceil+k} - B_{\lceil r/2\rceil}\right),
\qquad
\min_{0 \le k \le \lceil r/2\rceil} \left(B_r - B_{r-k}\right).\]
Consequently, if $B$ is a random exchangeable bridge, then for every $x \ge 0$, we have
\[\Pr{\max_{0 \le k \le r} B_k - \min_{0 \le k \le r} B_k \ge 3x}
\le 4 \cdot \Pr{\min_{0 \le k \le \lceil r/2\rceil} B_k \le -x}.\]
\end{lem}

\begin{proof}
Let us write $r/2$ instead of $\lceil r/2\rceil$ and set
\[M_1 = \max_{0 \le k \le r/2} B_k,
\qquad
m_1 = \min_{0 \le k \le r/2} B_k,
\qquad
M_2 = \max_{r/2 \le k \le r} B_k,
\qquad
m_2 = \min_{r/2 \le k \le r} B_k.\]
Suppose that the four minima in the statement are (strictly) larger than $-x$, then, since $B_r=0$,
\[m_1 > -x,
\qquad
B_{r/2} - M_1 > -x,
\qquad
m_2 - B_{r/2} > -x,
\qquad
-M_2 > -x.\]
It follows that
\begin{align*}
M_1 - m_1 &< (B_{r/2} + x) + x < m_2 + 3x \le 3x,
\\
M_1 - m_2 &< (B_{r/2} + x) - (B_{r/2} - x) = 2x,
\\
M_2 - m_1 &< 2x,
\\
M_2 - m_2 &< x - (B_{r/2} - x) \le 2x - m_1 < 3x,
\end{align*}
We conclude that $\max_{0 \le k \le r} B_k - \min_{0 \le k \le r} B_k = \sup\{M_1, M_2\} - \inf\{m_1, m_2\} < 3x$.

The last claim follows after observing that if $B$ is exchangeable, then the three processes
\[
\left(B_{r/2} - B_{r/2-k} ; 0 \le k \le r/2\right),
\quad
\left(B_{r/2+k} - B_{r/2} ; 0 \le k \le r/2\right),
\quad\left(B_r - B_{r-k} ; 0 \le k \le r/2\right)
\]
are distributed as $(B_k ; 0 \le k \le r/2)$.
\end{proof}

\begin{proof}[Proof of Lemma \ref{lem:borne_ecart_max_pont}]
First note that on the event $\{S_r=0\}$, $\max_{0 \le k \le r} S_k - \min_{0 \le k \le r} S_k$ cannot exceed $br$. Moreover, on the event $\{S_r=0\}$, the path $(S_0, \dots, S_r)$ is an exchangeable bridge so, according to Lemma \ref{lem:pont_echangeable}, it suffices to show that there exists two constants $c, C > 0$ which only depend on $b$ and $\sigma$ such that for every $r \ge 1$ and $0 \le x \le br$,
\[\Prc{\min_{0 \le k \le \lceil r/2 \rceil} S_k \le -x}{S_r=0} \le C \e^{-cx^2/r}.\]
For every $k \ge 1$ and every $x \in \Z$, let us set $\theta_k(x) = \Pr{S_k=-x}$. According to the local limit theorem, for every $k \ge 1$ and $x \in \Z$,
\[\sqrt{k}\theta_k(x) = g(x/\sqrt{k}) + \varepsilon_k(x),\]
where $g(x) = (2\pi\sigma^2)^{-1/2} \e^{-x^2/(2\sigma^2)}$ and $\lim_{k \to \infty} \sup_{x \in \Z} |\varepsilon_k(x)| = 0$. It follows that
\[C \coloneq \sup_{r \ge 1, x \in \Z} \frac{\theta_{r-\lceil r/2 \rceil}(x)}{\theta_r(0)}
= \sup_{r \ge 1, x \in \Z} \sqrt{\frac{r}{r-\lceil r/2 \rceil}} \frac{g(-x/\sqrt{r-\lceil r/2 \rceil}) + \varepsilon_{r-\lceil r/2 \rceil}(x)}{g(0) + \varepsilon_r(0)}
< \infty.\]
Using the Markov property at time $\lceil r/2 \rceil$, we have thereby
\begin{align*}
\Prc{\min_{0\le k\le \lceil r/2 \rceil} S_k \le -x}{S_r = 0}
&= \frac{\Pr{\min_{0\le k\le \lceil r/2 \rceil} S_k \le -x \text{ and } S_r = 0}}{\Pr{S_r = 0}}
\\
&= \Es{\ind{\min_{0\le k\le \lceil r/2 \rceil} S_k \le -x} \frac{\theta_{r - \lceil r/2 \rceil}(S_{\lceil r/2 \rceil})}{\theta_r(0)}}
\\
&\le C \cdot \Pr{\min_{0\le k\le \lceil r/2 \rceil} S_k \le -x}.
\end{align*}
Finally, since $-S$ is a random walk with step distribution bounded above by $b$, centred and with variance $\sigma^2$, we have the following concentration inequality (see e.g. Mc Diarmid \cite{McDiarmid:Concentration}, Theorem 2.7 and the remark at the end of Section 2 there): for every $n \ge 1$ and every $x \ge 0$,
\[\Pr{\max_{0 \le k\le n} -S_k \ge x} \le \exp\left(-\frac{x^2}{2\sigma^2n + 2bx/3}\right).\]
We conclude that for every $r \ge 1$ and every $0 \le x \le br$, we have
\[\Prc{\min_{0 \le k \le \lceil r/2 \rceil} S_k \le -x}{S_r=0}
\le C \exp\left(-\frac{x^2}{2\sigma^2\lceil r/2 \rceil + 2bx/3}\right)
\le C \exp\left(-\frac{x^2}{(2\sigma^2 + 2b^2/3)r}\right),\]
and the proof is complete.
\end{proof}


\end{document}